\documentclass[11pt,a4paper,reqno]{amsart} 
\usepackage{amsmath}
\usepackage{amssymb} 
\usepackage[all]{xy}
\usepackage{graphicx}
\usepackage{color}
\usepackage[colorlinks=true,linktocpage=true,pagebackref=false, citecolor=black,linkcolor=black]{hyperref}
\usepackage{euscript}

\newtheorem{thm}{Theorem}[section]

\newtheorem{cor}[thm]{Corollary}
\newtheorem{lem}[thm]{Lemma}

\theoremstyle{definition}

\theoremstyle{remark}
\newtheorem{remark}[thm]{Remark}
\newtheorem{remarks}[thm]{Remarks}
\newtheorem{example}[thm]{Example}
\newtheorem{examples}[thm]{Examples}

\numberwithin{equation}{section}

 \newcommand{\N}{{\mathbb N}}
\newcommand{\Z}{{\mathbb Z}} \newcommand{\R}{{\mathbb R}}
\newcommand{\Q}{{\mathbb Q}} \newcommand{\C}{{\mathbb C}}

\newcommand{\psd}{{\mathcal P}}

\newcommand{\gtp}{{\mathfrak p}} \newcommand{\gtq}{{\mathfrak q}}
\newcommand{\gtm}{{\mathfrak m}} \newcommand{\gtn}{{\mathfrak n}}
\newcommand{\gta}{{\mathfrak a}} \newcommand{\gtb}{{\mathfrak b}}

\newcommand{\Qq}{{\EuScript Q}}

\newcommand{\qf}{\operatorname{qf}}

\newcommand{\supp}{\operatorname{supp}}
\newcommand{\Sper}{\operatorname{Sper}}

\newcommand{\ini}{\operatorname{In}}
\newcommand{\Ass}{\operatorname{Ass}}

\newcommand{\x}{{\tt x}} \newcommand{\y}{{\tt y}} \renewcommand{\v}{{\tt v}}
\newcommand{\z}{{\tt z}} \renewcommand{\t}{{\tt t}} 
\newcommand{\s}{{\tt s}} \renewcommand{\u}{{\tt u}}

\newcommand{\Sos}[1]{{\Sigma{#1}^2}} 
\newcommand{\Sosp}[1]{{\Sigma_p{#1}^2}} 
\newcommand{\Sosq}[1]{{\Sigma_{p_0}{#1}^2}} 
\newcommand{\Sosd}[1]{{\Sigma_2{#1}^2}} 
\newcommand{\veps}{\varepsilon}
\newcommand{\ol}{\overline}

\newcommand{\simr}{\sim_{\text{\tiny$\displaystyle R$}}}
\newcommand{\simc}{\sim_{\text{\tiny$\displaystyle C$}}}

\begin{document}

\title[Representation of positive semidefinite elements as sum of squares]{Representation of positive semidefinite elements\\
as sum of squares in $2$-dimensional local rings}

\author{Jos\'e F. Fernando}
\address{Departamento de \'Algebra, Geometr\'ia y Topolog\'ia, Facultad de Ciencias Matem\'aticas, Universidad Complutense de Madrid, 28040 MADRID (SPAIN)}
\email{josefer@mat.ucm.es}
\thanks{Author is supported by Spanish STRANO MTM2017-82105-P and Grupos UCM 910444.}

\date{27/12/2021}

\dedicatory{Dedicated to Prof. J.M. Ruiz on occasion of his 65th birthday}

\begin{abstract}
A classical problem in real geometry concerns the representation of positive semidefinite elements of a ring $A$ as sums of squares of elements of $A$. If $A$ is an excellent ring of dimension $\geq3$, it is already known that it contains positive semidefinite elements that cannot be represented as sums of squares in $A$. The one dimensional local case has been afforded by Scheiderer (mainly when its residue field is real closed). In this work we focus on the $2$-dimensional case and determine (under some mild conditions) which local excellent henselian rings $A$ of embedding dimension $3$ have the property that every positive semidefinite element of $A$ is a sum of squares of elements of $A$.
\end{abstract}

\maketitle

\setcounter{tocdepth}{1}
\tableofcontents

\section{Introduction}\label{s1}

In the study of positive semidefinite elements and sums of squares of a ring $A$ one main problem is to determine whether every positive semidefinite element is a sum of squares (qualitative problem). The positive semidefinite elements of an arbitrary commutative ring are defined by means of the theory of the real spectrum $\Sper(A)$ of the ring $A$, as follows: an element $f\in A$ is \em positive semidefinite \em if $f\geq_\alpha0$ for every prime cone $\alpha\in\Sper(A)$. Recall that a \em prime cone \em $\alpha$ can be understood as a pair $\alpha:=(\gtp_\alpha,\leq_\alpha)$ where $\gtp_\alpha$ is a prime ideal of $A$ (called the \em support \em $\supp(\alpha)$ of $\alpha$) and $\leq_\alpha$ is an ordering of the quotient field $\qf(A/\gtp_\alpha)$. Alternatively, $\alpha$ is the preimage under the canonical homomorphism $A\to\qf(A/\gtp_\alpha)$ of the set of non-negative elements of the ordered field $(\qf(A/\gtp_\alpha),\leq_\alpha)$. We denote the set of positive semidefinite elements of $A$ with $\psd(A)$ and the set of all (finite) sums of squares of $A$ with $\Sos{A}$. The problem stated above consists of determining under which conditions the equality $\psd(A)=\Sos{A}$ holds.

By \cite[Thm.4.3.7]{bcr} $\Sper(A)=\varnothing$ if and only if $-1$ is a finite sum of squares in $A$. If this is the case, we assume that all the elements of $A$ are positive semidefinite. If in addition the characteristic of $A$ is different from $2$ and $\frac{1}{2}\in A$, then each $a\in A$ is a sum of squares in view of the well-known relation
$$
a=\Big(\frac{a+1}{2}\Big)^2+(-1)\Big(\frac{a-1}{2}\Big)^2.
$$

An ideal $\gta\subset A$ is \em real \em if for each sequence $a_1,\ldots,a_r\in A$ such that $a_1^2+\cdots+a_r^2\in\gta$, we have $a_i\in\gta$ for $i=1,\ldots,r$. The \em real-radical \em of an ideal $\gta\subset A$ is the smallest real ideal $\sqrt[r]{\gta}$ of $A$ that contains $\gta$. By \cite[Prop.4.1.7]{bcr} 
$$
\sqrt[r]{\gta}=\{a\in A:\ \exists\ a_1,\ldots,a_r\in A, m\geq1\text{ such that }
a^{2m}+a_1^2+\cdots+a_r^2\in\gta\}.
$$ 
A ring $A$ is \em real \em (reduced) if the zero ideal is real. The \em real reduction \em of $A$ is the quotient $A/\sqrt[r]{(0)}$. In case $A$ is a field, it is \em (formally) real \em if and only if $-1$ is not a sum of squares in $A$ (that is, if its real spectrum is non-empty). Thus, a prime ideal $\gtp$ of $A$ is the support of a prime cone $\alpha$ if and only if it is real. 

\subsection*{Background} 
The property $\psd(A)=\Sos{A}$ is true for the total quotient ring of fractions of a real (reduced) ring $A$ by the general theory of Artin-Schreier for (formally) real fields \cite[\S1]{bcr}, but in the general case the situation is more complicated. In \cite[Lem.6.3]{sch1} it is proved that if a noetherian ring $A$ has a non-empty real spectrum (that is, not all the elements of $A$ are positive semidefinite) and the property $\psd(A)=\Sos{A}$ holds, then $A$ is a real (reduced) ring. In fact, strong dimensional restrictions appear even under mild hypotheses \cite[Cor.1.3]{sch1}. In case we focus on excellent rings, we have the following result \cite[Main Thm.1.1]{frs1} (see also \cite[Thm.1.2]{f3}).

\begin{thm}[{\cite[Main Thm.1.1]{frs1}}]\label{mfrs}
Let $A$ be an excellent ring of real dimension $\geq3$. Then $\psd(A)\neq\Sos{A}$.
\end{thm}

The previous result involves the concept of real dimension. Given two prime cones $\alpha,\beta\in\Sper(A)$, we say that $\alpha$ is a \emph{specialization} of $\beta$ (written $\beta\to \alpha$) if $f>_\alpha0$ implies $f>_\beta0$ for each $f\in A$. This implies $\gtq:=\supp(\beta)\subset\supp(\alpha)=:\gtp$. We set $\dim(\beta\to\alpha):=\dim(A_{\gtp}/\gtq A_{\gtp})$, and define the \emph{real dimension} of $A$ as
\begin{center}
$\dim_r(A):=\sup\{\dim(\beta\to\alpha)$: $\alpha$, $\beta\in
\Sper(A)$, $\beta\to\alpha\}$.
\end{center}
Therefore, $\dim_r(A)\le\dim(A)$. Let us show that this inequality is an equality in case $(A,\gtm)$ is a local henselian noetherian ring such that $\psd(A)=\Sos{A}$. By \cite[Prop.II.2.4]{abr} each non-refinable specialization chain finishes on a prime cone $\alpha$ whose support is $\gtm$. If $A$ is in addition noetherian and has the property $\psd(A)=\Sos{A}$, then $A$ is real (reduced), so all its minimal prime ideals are real \cite[Lem.4.1.5]{bcr}. Thus, $\dim_r(A)=\sup\{\dim(A/\supp(\beta)):\ \beta\in\Sper(A)\}=\dim(A)$. Consequently, if $A$ is a local excellent henselian ring with the property $\psd(A)=\Sos{A}$, then $\dim(A)\leq2$.

The one dimensional case was entirely solved by Scheiderer in \cite[\S3]{sch2}. The most conclusive case concerns the one when the residue field $\kappa:=A/\gtm$ is real closed. Recall that a prime ideal $\gtp$ of a ring $A$ is \em associated to $A$ \em if there exists a non-zero element $x\in A$ such that $\gtp=\{a\in A:\ ax=0\}$. The set of all associated prime ideals of $A$ is denoted with $\Ass(A)$.

\begin{thm}[{\cite[Thm.3.9]{sch2}}]
Let $(A,\gtm)$ be a one-dimensional local Nagata ring with (formally) real residue field $\kappa$ and completion $\widehat{A}$, and assume $\gtm\not\in\Ass(A)$. Consider the following conditions:
\begin{enumerate}
\item[(i)] $\psd(A)=\Sos{A}$.
\item[(ii)] $\psd(\widehat{A})=\Sos{\widehat{A}}$.
\item[(iii)] There is $n\geq 1$ such that $\widehat{A}\cong\kappa[[\x_1,\ldots,\x_n]]/(\x_i\x_j:\ i<j)$.
\end{enumerate}
Then the following assertions hold:
\begin{enumerate}
\item[(1)] Each of these conditions implies that $A$ is reduced.
\item[(2)] Conditions \em (i) \em and \em (ii) \em are equivalent, and both are implied by \em (iii)\em.
\item[(3)] If $\kappa$ is real closed, then all three conditions are equivalent.
\end{enumerate}
\end{thm}

Thus, we will focus on determining all local excellent henselian rings $(A,\gtm)$ of dimension $\leq2$ and embedding dimension $\leq3$ with the property $\psd(A)=\Sos{A}$. In the $2$-dimensional regular case, the most general result is the following:

\begin{thm}[{\cite[Thms.4.1 \& 4.8]{sch2}}]
If $A$ is a $2$-dimensional regular semilocal ring, then $\psd(A)=\Sos{A}$. In particular, for every field $\kappa$ we have that $\psd(\kappa[[\x,\y]])=\Sos{\kappa[[\x,\y]]}$.
\end{thm}

The $2$-dimensional case, when $A$ is an analytic ring of dimension $2$ and embedding dimension $\leq3$ over the real numbers $\R$, has been studied in \cite{f2,f4,f5,f7,fr1,rz2}. The results presented in that articles can be adapted straightforwardly to the local henselian excellent case with real closed residue field \cite{f6}. We can summarize the main results as follows:

\begin{thm}[{\cite[Thm.1.3]{f2}}, \cite{f4,f5}]\label{list1}
Let $(A,\gtm)$ be a local henselian excellent ring of dimension $2$ and embedding dimension $3$. Assume that the residue field $\kappa:=A/\gtm$ is real closed. Then $\psd(A)=\Sos{A}$ if and only if the completion $\widehat{A}$ is isomorphic to
\begin{itemize}
\item[(1)] $\kappa[[\x,\y]]$ or
\item[(2)] $\kappa[[\x,\y,\z]]/(\z\x,\z\y)$ or 
\item[(3)] $\kappa[[\x,\y,\z]]/(\z^2-F(\x,\y))$ 
\end{itemize}
where $F\in\kappa[[\x,\y]]$ is one of the series in the following list:
\begin{itemize}
\item [(i)] $\x^2+\y^k$ where $k\geq 2$,
\item [(ii)] $\x^2$,
\item [(iii)] $\x^2\y+(-1)^k\y^k$ where $k\geq 3$,
\item [(iv)] $\x^2\y$,
\item [(v)] $\x^3+\x\y^3$, 
\item [(vi)] $\x^3+\y^4$,
\item [(vii)] $\x^3+\y^5$.
\end{itemize} 
\end{thm}

\subsection*{Main results}
In this article we prove that if we do not impose that $\kappa$ is a real closed field, the amount of possible candidates increases because $\kappa$ may have more than one ordering and $k$th roots of positive elements of $\kappa$ need not to belong to $\kappa$. In order to apply freely Rotthaus results \cite{rt} on Artin's approximation Theorem, one needs that the involved ring $A$ we are working with contains a copy of $\Q$. As we deal with local henselian rings $(A,\gtm)$ with non-empty real spectrum, it is enough to ask that $\frac{1}{2}\in A$ (or equivalently that $2\not\in\gtm$). If such is the case, non-zero integers are units of $A$ (or equivalently, non-zero integers do not belong to $\gtm$) and consequently $A$ contains a copy of $\Q$. Otherwise there exists an integer $n\geq3$ such that $n\in\gtm$ and the roots of the polynomial equation $\t^2+n-1=0$ modulo $\gtm$ are $\pm1$. As $2\not\in\gtm$, both roots are simple. As $A$ is a henselian ring, there exists $a\in A$ such that $a+\gtm=1+\gtm$ and $a^2+1^2+\overset{(n-2)}{\cdots}+1^2=a^2+n-2=-1$ in $A$, against the fact that $A$ has non-empty real spectrum.

\begin{thm}[List of candidates]\label{gp=s}
Let $(A,\gtm)$ be a local henselian excellent ring of dimension $2$ and embedding dimension $3$. Assume that $A$ has non-empty real spectrum and $\frac{1}{2}\in A$. If $\psd(A)=\Sos{A}$, then the completion $\widehat{A}$ is isomorphic to
\begin{itemize}
\item[(1)] $\kappa[[\x,\y]]$ or
\item[(2)] $\kappa[[\x,\y,\z]]/(\z\x,\z\y)$ or 
\item[(3)] $\kappa[[\x,\y,\z]]/(\z^2-F(\x,\y))$ 
\end{itemize}
where $F\in\kappa[[\x,\y]]$ is one of the series in the following list:
\begin{itemize}
\item[(i)] $a\x^2+b\y^{2k}$ where $a\not\in-\Sos{\kappa}$, $b\neq0$ and $k\geq 1$,
\item[(ii)] $a\x^2+\y^{2k+1}$ where $a\not\in-\Sos{\kappa}$ and $k\geq 1$,
\item[(iii)] $a\x^2$ where $a\not\in-\Sos{\kappa}$, 
\item[(iv)] $\x^2\y+(-1)^ka\y^k$ where $a\not\in-\Sos{\kappa}$ and $k\geq 3$,
\item[(v)] $\x^2\y$,
\item[(vi)] $\x^3+a\x\y^2+b\y^3$ irreducible,
\item[(vii)] $\x^3+a\y^4$ where $a\not\in-\Sos{\kappa}$, 
\item[(viii)] $\x^3+\x\y^3$,
\item[(ix)] $\x^3+\y^5$. 
\end{itemize}
\end{thm}
\begin{remark}
The assumption \em $A$ has non-empty real spectrum and $\frac{1}{2}\in A$ \em is equivalent to the condition \em the residue field $\kappa$ of $A$ is (formally) real\em. The implication right to left is clear. To prove the converse assume that there exist elements $a_1,\ldots,a_r\in A\setminus\gtm$ such that $a_1^2+\cdots+a_r^2=-1$ modulo $\gtm$, so $a_1$ is a simple root of the polynomial equation $\t^2+a_2^2+\cdots+a_r^2+1=0$ modulo $\gtm$ (because $2\not\in\gtm$). As $A$ is a henselian ring, we may assume that $a_1$ is a root of $\t^2+a_2^2+\cdots+a_r^2+1=0$ in $A$, against the non-emptiness of the real spectrum of $A$.
\end{remark}

\subsubsection*{Pythagoras numbers and $\tau$-invariant.} An important invariant that allows us to use Artin's approximation techniques when approaching a converse to the previous result is Pythagoras number. The \em Pythagoras number $p(A)$ \em of a ring $A$ is the smallest integer $p\ge1$ such that every sum of squares of $A$ is a sum of $p$ squares. We write $p(A)=+\infty$ if such an integer does not exist. If we denote the elements of $A$ that are sums of $p$ squares in $A$ with $\Sosp{A}$, then $p(A)=\inf\{p\geq1:\ \Sos{A}=\Sosp{A}\}$. This is a very delicate invariant whose estimation (quantitative problem) has deserved a lot of attention from specialists in number theory, quadratic forms, real algebra and real geometry \cite{bcr,cdlr,clrr,l,pf,sch1,sch2}. In \cite{f4,f5} we proved that the local henselian excellent rings of dimension $2$ and embedding dimension $3$ with Pythagoras number $2$ coincide essentially with those in the list provided in Theorem \ref{list1}. In \cite{f7} we showed `positive extension properties' for the elements of the list inside Theorem \ref{list1}, whereas in \cite{fr3} we analyze relations between the Pythagoras numbers of real analytic germs.

In \cite{f1} we showed that a local henselian excellent ring of dimension $\leq2$ and real closed residue field $\kappa$ has a finite Pythagoras number. In \cite{frs2} we improved the previous result and proved the following.

\begin{thm}[{\cite[Prop.2.7, Thm.2.9]{frs2}}]\label{pyth}
Let $(A,\gtm)$ be a local henselian excellent ring of dimension $2$ such that its residue field $\kappa$ satisfies $p(\kappa[\t])<+\infty$. Let $m$ be the number of generators of the completion $\widehat{A}$ as a $\kappa[[\x_1,\x_2]]$-module. Then $p(A)\leq 2p(\kappa[\t])m$.
\end{thm}

For a higher dimension we proved in \cite{f1,frs1} that a real (reduced) ring $A$ of dimension $\geq3$ has an infinite Pythagoras number. The invariant $p(\kappa[\t])$ has been bounded by Scheiderer in \cite{sch2} following \cite{pf}. For a (formally) real field $\kappa$ define the invariant
$$
\tau(\kappa):=\sup\{s(F):\ F|\kappa \text{ finite, non-real}\},
$$ 
where $s(F)$ denotes the \em level of $F$\em, that is, the minimum number of elements of $F$ needed to represent $-1$ as a sum of squares in $F$, which is always a power of $2$, see \cite[Pfister's Thm.XI.2.2]{l}. Scheiderer proved in \cite[Prop.5.17]{sch2} (using in an essential way Pfister's results \cite{pf}) the following inequalities:
\begin{equation}\label{p2}
1+\tau(\kappa)\leq p(\kappa[\y])\leq p(\kappa[[\x,\y]])\leq p(\kappa[[\x]][\y])\leq 2\tau(\kappa).
\end{equation}
In addition, $\tau(\kappa)=\tau(\kappa((\x)))$, see \cite[Lem.5.13]{sch2}. The Pythagoras number $p(\kappa[[\x,\y]])$ has been also studied by Hu in \cite[\S3]{hu} where he showed that
\begin{multline*}
p(\kappa[[\x,\y]])=p(\kappa[\x][[\y]])=p(\kappa[[\x]][\y])=p(\kappa((\x,\y)))=p(\qf(\kappa[\x][[\y]]))\\
=p(\kappa((\x))(y))=\sup\{p(K(\x)):\ K|\kappa\text{ is a finite field extension}\}.
\end{multline*}
The last equality is due to Becher-Grimm-Van Geel \cite{bgv}. It is conjectured in \cite[Conj.4.16]{bgv} that the inequality $p(\kappa[\y])\leq p(\kappa((\x))(\y))$ is in fact an equality or, equivalently, that $p(K(\x))\leq p(\kappa(\x))$ for each finite extension $K|\kappa$. Recall that $4=p(\Q)\leq p(\Q[\y])=p(\Q((\x,\y)))=p(\Q[[\x,\y]])=5$ (see \cite{p} and \cite[Rem.3.5]{hu}), so $\tau(\Q)=4$ (use \eqref{p2}).

As a kind of converse of Theorem \ref{gp=s} we prove the following result.

\begin{thm}[Affirmative cases]\label{list2}
Let $(A,\gtm)$ be a local henselian excellent ring of dimension $2$ and embedding dimension $3$. Suppose that the residue field $\kappa$ is (formally) real and $\tau(\kappa)<+\infty$. Assume that the completion $\widehat{A}$ is isomorphic to: 
\begin{itemize}
\item[(1)] $\kappa[[\x,\y]]$ or
\item[(2)] $\kappa[[\x,\y,\z]]/(\z\x,\z\y)$ or 
\item[(3)] $\kappa[[\x,\y,\z]]/(\z^2-F(\x,\y))$,
\end{itemize}
where $F\in\kappa[[\x,\y]]$ is one of the series in the following list: 
\begin{itemize}
\item[(i)] $a\x^2+b\y^{2k}$ where $a\in\Sos{\kappa}$, $a,b\neq0$ and $k\geq 1$,
\item[(ii)] $a\x^2+\y^{2k+1}$ where $a\in\Sos{\kappa}$, $a\neq0$ and $k\geq 1$,
\item[(iii)] $a\x^2$ where $a\in\Sos{\kappa}$ and $a\neq0$,
\item[(iv)] $\x^2\y+(-1)^ka\y^k$ where $a\not\in-\Sos{\kappa}$ and $k\geq 3$,
\item[(v)] $\x^2\y$,
\item[(vi)] $\x^3+a\x\y^2+b\y^3$ irreducible,
\item[(vii)] $\x^3+a\y^4$ where $a\not\in-\Sos{\kappa}$, 
\item[(viii)] $\x^3+\x\y^3$,
\item[(ix)] $\x^3+\y^5$. 
\end{itemize} 
Then $\psd(A)=\Sos{A}$. In addition, $p(A)\leq 4\tau(\kappa)$. 
\end{thm}

Theorems \ref{gp=s} and \ref{list2} can be understood as the necessary (list of candidates) and sufficient (affirmative cases) conditions for a local henselian excellent ring of dimension $2$ and embedding dimension $3$ to enjoy the property $\psd(A)=\Sos{A}$. In order to join both Theorems \ref{gp=s} and \ref{list2}, we ask in addition that the residue field $\kappa$ admits a unique ordering (that is, $\kappa=-\Sos{\kappa}\cup\Sos{\kappa}$) and obtain straightforwardly the following full characterization.

\begin{cor}[Full characterization]\label{unique}
Let $(A,\gtm)$ be a local henselian excellent ring of dimension $2$ and embedding dimension $3$. Suppose that the residue field $\kappa$ admits a unique ordering and $\tau(\kappa)<+\infty$. Then $\psd(A)=\Sos{A}$ if and only if the completion $\widehat{A}$ is isomorphic to: 
\begin{itemize}
\item[(1)] $\kappa[[\x,\y]]$ or
\item[(2)] $\kappa[[\x,\y,\z]]/(\z\x,\z\y)$ or 
\item[(3)] $\kappa[[\x,\y,\z]]/(\z^2-F(\x,\y))$,
\end{itemize}
where $F\in\kappa[[\x,\y]]$ is one of the series in the following list: 
\begin{itemize}
\item[(i)] $a\x^2+b\y^{2k}$ such that $a,b>0$ and $k\geq 1$,
\item[(ii)] $a\x^2+\y^{2k+1}$ where $a>0$ and $k\geq 1$,
\item[(iii)] $a\x^2$ where $a>0$,
\item[(iv)] $\x^2\y+(-1)^ka\y^k$ where $a>0$ and $k\geq 3$,
\item[(v)] $\x^2\y$,
\item[(vi)] $\x^3+a\x\y^2+b\y^3$ irreducible,
\item[(vii)] $\x^3+a\y^4$ where $a>0$, 
\item[(viii)] $\x^3+\x\y^3$,
\item[(ix)] $\x^3+\y^5$. 
\end{itemize}
In addition, $p(A)\leq 4\tau(\kappa)$.
\end{cor}
\begin{remarks}
The case (3.vi) `$F=\x^3+a\x\y^2+b\y^3$ irreducible' requires a special comment. Denote $A:=\kappa[[\x,\y,\z]]/(\z^2-F)$, the real closure of $\kappa$ with $R$ (endowed with its unique ordering) and $B:=R[[\x,\y,\z]]/(\z^2-F)$.

(i) For simplicity suppose $\kappa=\Q$ and consider the irreducible polynomial $F:=\x^3+2\y^3\in\Q[\x,\y]$. By Corollary \ref{unique} we have $\psd(A)=\Sos{A}$. On the other hand, $F=(\x+\sqrt[3]{2}\y)((\x-\frac{1}{2}\sqrt[3]{2}\y)^2+\frac{3}{4}\sqrt[3]{4}\y^2)$ is reducible in $R[\x,\y]$ and $\x+\sqrt[3]{2}\y\in\psd(B)\setminus\Sos{B}$. 

(ii) This surprising situation can only appear in case (3.vi), because the condition of `irreducibility of $F$' disappears when we extend coefficients to the real closure $R$ of $\kappa$ because $F(\x,1)$ has degree $3$, whereas the obstructions for the remaining elements of the list keep the same when extending coefficients to the real closure. If $F(\x,1)=\x^3+a\x+b$ has only one root in $R$, then $\psd(B)\setminus\Sos{B}\neq\varnothing$ is always true. If $F(\x,1)$ has three roots in $R$ (which happens if and only if $a<0,4a^3+27b^2<0$, see Example \ref{irredrevisted}), then $\psd(B)=\Sos{B}$ (case (3.iv) for $k=3$).
\end{remarks}

\subsubsection*{Formally real fields with a unique ordering}
Apart from real closed fields there are well-known examples of fields with a unique ordering. The most simple one corresponds to the field $\Q$ of rational numbers (which has Pythagoras number $4$), but also each finite algebraic extension $\Q[\theta]\subset\R$ of $\Q$ such that the irreducible polynomial of $\theta$ over $\Q$ has odd degree and a unique root in $\R$ admits a unique ordering. The field $\Q[\theta]$ has Pythagoras number either $3$ or $4$, see \cite{si}. Other examples of fields with a unique ordering are the real constructible numbers \cite[Ex.p.236]{l} or \em Euclidean fields\em, which are those fields $\kappa$ in which every element is either a square or the opposite of a square \cite[Prop.VIII.1.6]{l}. The previous examples of fields contained in $\R$ that admit a unique ordering are all archimedean, but it is not difficult to construct examples of non-archimedean fields with a unique ordering that are non-euclidean (and consequently non-real closed). An example of this type of fields is $\Q((\t^{1/2^*})):=\bigcup_{n\geq0}\Q((\t^{1/2^n}))$, which is a field with Pythagoras number $4$. More generally, if $\kappa$ is a field with a unique ordering and Pythagoras number $p$, also $\kappa((\t^{1/2^*})):=\bigcup_{n\geq0}\kappa((\t^{1/2^n}))$ is a non-archimedean (formally) real field with a unique ordering and Pythagoras number $p$. In \cite[Thm.2]{h} the existence of (formally) real fields with a unique ordering and Pythagoras number $p$ is proved for each $p\geq1$.

\subsection*{Applications: Principal saturated preorderings of low order.}

In his articles \cite{sch4,sch5} Scheiderer approached the problem of determining when a finitely generated preordering $T$ in an excellent regular ring $(A,\gtm)$ of dimension $2$ is saturated. To that end, he established criteria to decide when the saturation of the preordering $\widehat{T}$ generated by $T$ in the completion $\widehat{A}$ implies the saturation of $T$. Recall that if $T$ is generated by $h_1,\ldots,h_r$, then $T:=\{\sum_{\nu\in\{0,1\}^r}\sigma_\nu h_1^{\nu_1}\cdots h_r^{\nu_r}:\ \sigma_i\in\Sos{A}\}$. We say that $T$ is \em saturated \em if it contains every $f\in A$ such that $f\geq_\alpha0$ for each $\alpha\in\Sper(A)$ satisfying $h_1\geq_\alpha0,\ldots,h_r\geq_\alpha0$. Observe that 
$$
\widehat{T}:=\Big\{\sum_{\nu\in\{0,1\}^r}\sigma_\nu h_1^{\nu_1}\cdots h_r^{\nu_r}:\ \sigma_i\in\Sos{\widehat{A}}\Big\}.
$$
In \cite[Cor.3.25]{sch4} the following is proved:

\begin{cor}\label{scht}
Let $A$ be an excellent henselian local ring with $p(A)<+\infty$ and let $T$ be a finitely generated preordering in $A$. Then $T$ is saturated in $A$ if and only if $\widehat{T}$ is saturated in $\widehat{A}$. 
\end{cor}

The previous result reduces the problem of studying the saturation of finitely generated preorderings on an excellent henselian local ring with a finite Pythagoras number to the study of the saturation of finitely generated preorderings on complete rings. Let $R$ be a real closed field. Taking advantage of \cite{f3,rz2}, Scheiderer characterizes the saturated preorderings of $A:=R[[\x,\y]]$ generated by an element $h\in R[[\x,\y]]$ of order $\leq3$ and he obtains the same list (up to right equivalence) as the one appearing in Theorem \ref{list1}(3). In the same vein we characterize in Section \ref{s6} (Corollaries \ref{ord2p} and \ref{ord3p}) the saturated preorderings of $A:=\kappa[[\x,\y]]$ generated by an element $F\in\kappa[[\x,\y]]$ of order $\leq3$ (up to right equivalence), where $\kappa$ is a (formally) real field that has $\tau(\kappa)<+\infty$. In addition, we ask that $\kappa$ admits a unique ordering when $\omega(F)=2$. The lists obtained in these cases (for $\omega(F)=2$ and $\omega(F)=3$) coincide with those proposed in Theorem \ref{list2}(3).

\subsection*{Structure of the article}

The article is organized as follows. In Section \ref{s2} we present the main finite determinacy tools in order to prove Theorem \ref{gp=s} in Section \ref{s3}. The proofs of some of the results in Section \ref{s2}, which are substantially different from the classical ones over the complex or the real numbers, are included in Appendix \ref{b} for the sake of completeness. In Section \ref{s4} we introduce the main tools to prove Theorem \ref{list2} in Section \ref{s5}. We highlight Theorem \ref{et}. In Section \ref{s6} we prove (see Corollaries \ref{ord2p} and \ref{ord3p}) the counterpart of the results of Scheiderer concerning principal preorderings on an excellent henselian local ring with a finite Pythagoras number, when the residue field is a field $\kappa$ with a unique ordering (instead of a real closed field). Finally, in Appendix \ref{a} we present two additional examples: one concerning the ring $\widehat{A}$ in Theorem \ref{list2}(3.iv) (Example \ref{irredrevisted}) and another one (quite tricky!) concerning a ring $A$ with the property $\psd(A)=\Sos{A}$ that does not appear in the list provided in Theorem \ref{list2}(3) (Example \ref{newtrends}). Such example suggests that there is still further work to do (surely with the aid of new techniques) when the residue field $\kappa$ admits more than one ordering and the series $F\in\kappa[[\x,\y]]$ has order $2$.

\subsection*{Acknowledgements}

The author is indebted with C. Scheiderer for sharing his knowledge and proposing him (in the earlier 2002, during a 6 months post-doc research stay at Universit\"at Duisburg) this type of problems when the residue field $\kappa$ is (formally) real but not necessarily real closed. It has taken quite a long time to obtain satisfactory results because cases (3.vi) and (3.ix) in Theorem \ref{list2} stayed apart from any progress (except for very restrictive situations) until very recently. The author is also very grateful to S. Schramm for a careful reading of the final version and for the suggestions to refine its redaction. The author thanks the anonymous referees for their clever suggestions to substantially improve the presentation of the article.

\section{Basic tools when dealing with formal rings}\label{s2}

As we work in the environment of excellent henselian local rings $(A,\gtm)$ (that contain $\Q$) with a finite Pythagoras number $p$, the qualitative problem $\psd(A)=\Sos{A}$ has an affirmative solution in $A$ if and only if the equation
$$
f={\tt X}_1^2+\cdots+{\tt X}_p^2
$$
has a solution in the completion $\widehat{A}$ for each $f\in\psd(A)$. This is a straightforward consequence of Rotthaus Theorem on Artin's approximation property \cite[Thm.4.2]{rt}. Denote the residue field of $A$ with $\kappa:=A/\gtm$. As $A$ contains $\Q$, Cohen's structure theorem \cite[Thm.7.7]{e} states that $\widehat{A}\cong\kappa[[\x_1,\ldots,\x_n]]/\gta$ where $\gta$ is an ideal of $\kappa[[\x_1,\ldots,\x_n]]$ and $n:=\dim_{\kappa}(\gtm/\gtm^2)$ is the embedding dimension of $A$. In order to take advantage of formal rings, we recall some useful properties.

\subsection{Formal rings}
A \em formal ring \em over a field $\kappa$ is a ring $A:=\kappa[[\x]]/\gta$ where $\gta$ is an ideal of the ring $\kappa[[\x]]$ of (formal power) series in the variables $\x:=(\x_1,\ldots,\x_n)$. If $f\in\kappa[[\x]]$, we write $f:=\sum_{k=0}^{\infty}f_k$ where $f_k\in\kappa[\x]$ is (either $0$ or) an homogeneous polynomial of degree $k$. Denote the \em order of $f$ \em with $\omega(f):=\inf\{k:\ f_k\neq0\}$ and the \em initial form of $f$ \em with $\ini(f):=f_{\omega(f)}$. We write $\gtm_n:=(\x_1,\ldots,\x_n)\kappa[[\x]]$ to refer to the maximal ideal of $\kappa[[\x]]$ and $\|\x\|^2:=\x_1^2+\cdots+\x_n^2$. In addition, $\kappa((\x))$ stands for the quotient field of $\kappa[[\x]]$. Recall that a series $f\in\kappa[[\x_1,\ldots,\x_n]]$ is \em regular with respect to $\x_n$ (of order $d$) \em if $f(0,\ldots,0,\x_n)=a_d\x_n^d+\cdots$ with $a_d\neq0$. Let $\x':=(\x_1,\ldots,\x_{n-1})$. A \em Weierstrass polynomial \em $P\in\kappa[[\x']][\x_n]$ is a monic polynomial of degree $d$ such that $P(0,\x_n)=\x_n^d$. By \cite[Ch.7.\S.1.Thm.5 \& Cor.1]{zs} the ring $\kappa[[\x]]$ enjoys Weierstrass division and preparation theorems, which are fundamental tools when dealing with rings of series with coefficients in a field. 

\begin{thm}[Weierstrass division theorem {\cite[Ch.7.\S.1.Thm.5]{zs}}]
Let $f,g\in\kappa[[\x]]$ be series such that $f$ is a regular series with respect to $\x_n$ of order $d$. Then there exist $Q\in\kappa[[\x]]$ and a polynomial $R\in\kappa[[\x']][\x_n]$ of degree $\leq d-1$ such that $f=gQ+R$. 
\end{thm}

\begin{thm}[Weierstrass preparation theorem {\cite[Ch.7.\S.1.Cor.1]{zs}}]
Let $f\in\kappa[[\x]]$ be a regular series with respect to $\x_n$ of order $d$. Then there exist a Weierstrass polynomial $P\in\kappa[[\x']][\x_n]$ of degree $d$ and a unit $U\in\kappa[[\x]]$ such that $f=PU$. 
\end{thm}

As a consequence of the previous results, Implicit and Inverse Function Theorems arise standardly. We will use them freely in the article.

\begin{thm}[Implicit Function Theorem]\label{ift}
Let $f_1,\ldots,f_m\in\kappa[[\x,\y]]$ be such that $f_i(0,0)=0$ and $\det(\frac{\partial f_i}{\partial\y_j}(0,0))_{1\leq i,j\leq m}\neq0$.
Then there exist unique series $g_1,\ldots,g_m\in\kappa[[\x]]$ such that $g_j(0)=0$ and $f_i(\x,g_1,\ldots,g_m)=0$ for $i=1,\ldots,m$.
\end{thm}

\begin{thm}[Inverse Function Theorem]\label{ift2}
Let $f_1,\ldots,f_n\in\kappa[[\x]]$ be series such that $f_i(0)=0$ and $\det(\frac{\partial f_i}{\partial\x_j}(0))_{1\leq i,j\leq n}\neq0$. Then there exist unique series $g_1,\ldots,g_n\in\kappa[[\y]]$ such that $g_i(0)=0$ and $f_i(g_1,\ldots,g_n)=\y_i$ and $g_i(f_1,\ldots,f_n)=\x_i$ for $i=1,\ldots,n$. 
\end{thm}

Let $\kappa$ be a (formally) real field and $U\in\kappa[[\x,\y]]$ a unit such that $U(0,0)\in\Sosp{\kappa}$. By the Implicit Function Theorem the quotient $\frac{U}{U(0,0)}$ is a square in $\kappa[[\x,\y]]$, so $U\in\Sosp{\kappa[[\x,\y]]}$.

\subsection{Equivalence of series and finite determinacy tools}

In order to simplify the discussions when dealing with series, it is usual to use `equivalence of series'. Let $f,g\in \gtm$ where $\gtm$ is the maximal ideal of $\kappa[[\x]]$. We say that:
\begin{itemize}
\item[(i)] $f$ is \em right equivalent to $g$ \em (and we write $f\simr g$) if there exists an automorphism $\Phi:\kappa[[\x]]\to\kappa[[\x]]$ such that $\Phi(f)=g$.
\item[(ii)] $f$ is \em contact equivalent to $g$ \em (and we write $f\simc g$) if there exists an automorphism $\Phi:\kappa[[\x]]\to\kappa[[\x]]$ such that $(\Phi(f))=(g)$. This is equivalent to the existence of a unit $u\in\kappa[[\x]]$ such that $\Phi(f)=ug$. 
\end{itemize}

Two series $f,g$ are contact equivalent if and only if the $\kappa$-algebras $\kappa[[\x]]/(f)$, $\kappa[[\x]]/(g)$ are isomorphic. In relation to this we recall the meaning of $k$-determinacy and $k$-quasideterminacy. A series $f\in\gtm$ is $k$-determined (resp. $k$-quasidetermined) if each $g\in\kappa[[\x]]$ with $f-g\in\gtm^{k+1}$ is right equivalent to $f$ (resp. contact equivalent to $f$). If $f$ is $k$-determined (resp. $k$-quasidetermined) for some $k\geq1$, then $f$ is right (resp. contact) equivalent to a polynomial of $\kappa[\x]$. This fact will be `squeezed to the extreme' in the following sections. As it seems natural, it will be useful to have a criterion to determine when a series is finitely determined (resp. quasi determined), that is, when it is $k$-determined (resp. $k$-quasidetermined) for some $k\geq1$.

\begin{thm}[Finite Determinacy and Quasideterminacy Theorem]\label{fdqd} 
Let $f\in\gtm^2\subset\kappa[[\x]]$. Suppose that $\gtm^k\subset\gtm J(f)$ (resp. $\gtm^k\subset\gtm J(f)+(f)$) where $J(f):=(\frac{\partial f}{\partial\x_1},\ldots,\frac{\partial f}{\partial\x_n})$. Then $f$ is $k$-determined (resp. $k$-quasidetermined).
\end{thm}

The previous result is well-known when $\kappa=\C$ is the field of complex numbers and it is easy to adapt when $\kappa=\R$ or, more generally, when $\kappa$ is algebraically closed or real closed. A quite straightforward proof of the previous result when $\kappa=\C$ can be found in \cite[Thm.I.2.23]{gls} and \cite[Thm.9.1.7]{jp}. The details there can be adapted to approach the case when $\kappa$ is a field of characteristic $0$. For the sake of completeness we provide an elementary proof of Theorem \ref{fdqd} in Appendix \ref{b}.

\begin{examples}\label{examples}
Let $\kappa$ be a field of characteristic $0$ and $0\neq a,b\in\kappa$ and write $\gtm:=(\x,\y)\kappa[[\x,\y]]$. We have:

(i) If $F:=a\x^2+b\y^\ell$ where $\ell\geq 2$, then $J(F)=(\x,\y^{\ell-1})$ and $\gtm^{\ell}\subset\gtm J(F)$. Thus, $F$ is $\ell$-determined.

(ii) If $F:=a\x^2\y+b\y^\ell$ with $\ell\geq 3$, then $J(F)=(\x\y,a\x^2+\ell b\y^{\ell-1})$ and $\gtm^{\ell}\subset\gtm J(F)=(\x^2\y,\x\y^2,\x^3,\y^{\ell})$. Thus, $F$ is $\ell$-determined.

(iii) If $F:=\x^3+a\x\y^2+b\y^3$, then $J(F)=(3\x^2+a\y^2,2a\x\y+3b\y^2)$. We have $\gtm^{3}\subset\gtm J(F)=(3\x^3+a\x\y^2,3\x^2\y+a\y^3,2a\x^2\y+3b\x\y^2,2a\x\y^2+3b\y^3)$ if and only if the cubic forms $\{3\x^3+a\x\y^2,3\x^2\y+a\y^3,2a\x^2\y+3b\x\y^2,2a\x\y^2+3b\y^3\}$ are linearly independent or, equivalently, if $4a^3+27b^2\neq0$. Thus, $F$ is $3$-determined if the discriminant $4a^3+27b^2$ of $F(\x,1)$ is non-zero.

(iv) If $F:=\x^3+\x\y^3$, then $J(F)=(3\x^2+\y^3,\x\y^2)$ and $\gtm^{5}\subset(\x^3,\x^2\y^2,\x\y^3,\y^5)\subset\gtm J(F)=(3\x^3+\x\y^3,3\x^2\y+\y^4,\x^2\y^2,\x\y^3)$. Thus, $F$ is $5$-determined.

(v) If $F:=\x^3+a\y^k$ where $k\geq3$, then $J(F)=(\x^2,\y^{k-1})$ and 
$$
\gtm^k\subset\gtm J(F)=(\x^3,\x^2\y,\x\y^{k-1},\y^k).
$$ 
Thus, $F$ is $k$-determined. 

(vi) If $F:=\x^3+a\x\y^\ell+b\y^k$ where $2\leq\ell<k$ and $\frac{\ell}{k}\neq\frac{2}{3}$, then $F$ is $k$-quasidetermined. Observe that $J(F)=(3\x^2+a\y^\ell,\ell a\x\y^{\ell-1}+kb\y^{k-1})$ and 
\begin{multline*}
\gtm J(F)+(F)=(3\x^3+a\x\y^\ell,3\x^2\y+a\y^{\ell+1},\ell a\x^2\y^{\ell-1}+kb\x\y^{k-1},\\
\ell a\x\y^\ell+kb\y^k,\x^3+a\x\y^\ell+b\y^k).
\end{multline*}
As $\frac{\ell}{k}\neq\frac{2}{3}$, we deduce $\x^3,\x^2\y^{\ell-1},\x\y^\ell,\y^k\in\gtm J(F)+(F)$, so $\gtm^k\subset\gtm J(F)+(F)$ and $F$ is $k$-quasidetermined.

(vii) Let $F:=\x^3+a\x\y^{2\rho}+b\y^{3\rho}(1+h)$ where $\rho\geq2$, $4a^3+27b^2\neq0$, $h\in\kappa[[\y]]$ and either $h=0$ or $1\leq\omega(h)\leq\rho-1$. If $k:=\min\{3\rho+1+\omega(\frac{d h}{d\y}),4\rho-1\}$, then $F$ is $k$-quasidetermined. Define $g:=1+h+\frac{\y}{3\rho}\frac{d h}{d\y}$ and observe that $J(F)=(3\x^2+a\y^{2\rho},2\rho a\x\y^{2\rho-1}+3\rho b\y^{3\rho-1}g)$, so
\begin{multline*}
\gtm J(F)+(F)=(3\x^3+a\x\y^{2\rho},3\x^2\y+a\y^{2\rho+1},2\rho a\x^2\y^{2\rho-1}+3\rho b\x\y^{3\rho-1}g,\\
2\rho a\x\y^{2\rho}+3\rho b\y^{3\rho}g,\x^3+a\x\y^{2\rho}+b\y^{3\rho}(1+h)).
\end{multline*}
In particular, $\y^{3\rho+1}\frac{d h}{d\y}\in\gtm J(F)+(F)$.
{\em As $4a^3+27b^2\neq0$, also $\x\y^{3\rho-1},\x^2\y^{2\rho-1},\x^3\y^{\rho-1},\x^4,\y^{4\rho-1}\in\gtm J(F)+(F)$}, so $\gtm^k\subset\gtm J(F)+(F)$ and $F$ is $k$-quasidetermined.
\end{examples}
\begin{proof}[Proof of \em(vii)]
First, observe that
$$
\y^{3\rho+1}\frac{d h}{d\y}=\tfrac{\rho}{b}(3\x^3+a\x\y^{2\rho})+\tfrac{1}{b}(2\rho a\x\y^{2\rho}+3\rho b\y^{3\rho}g)-\tfrac{3\rho}{b}(\x^3+a\x\y^{2\rho}+b\y^{3\rho}(1+h))\in\gtm J(F)+(F).
$$
Consequently, $f_1:=2a\x\y^{2\rho}+3b\y^{3\rho}(1+h)\in\gtm J(F)+(F)$. In addition,
$$
f_2:=9b\x\y^{3\rho-1}g-2a^2\y^{4\rho-1}=3(2a\x^2\y^{2\rho-1}+3b\x\y^{3\rho-1}g)-2a\y^{2\rho-2}(3\x^2\y+a\y^{2\rho+1})\in\gtm J(F)+(F)
$$
and we conclude
$$
(4a^3+27b^2g(1+h))\y^{4\rho-1}=9b\y^{\rho-1}gf_1-2af_2\in\gtm J(F)+(F).
$$
As $4a^3+27b^2\neq0$, we deduce $\y^{4\rho-1}\in\gtm J(F)+(F)$. Thus, $9b\x\y^{3\rho-1}g=f_2+2a^2\y^{4\rho-1}\in\gtm J(F)+(F)$, so also $\x\y^{3\rho-1}\in\gtm J(F)+(F)$. Now, 
$$
2\rho a\x^2\y^{2\rho-1}=(2\rho a\x^2\y^{2\rho-1}+3\rho b\x\y^{3\rho-1}g)-3\rho b\x\y^{3\rho-1}g\in\gtm J(F)+(F),
$$
so $\x^2\y^{2\rho-1}\in\gtm J(F)+(F)$. Then $\x^4=\frac{1}{3}\x(3\x^3+a\x\y^{2\rho})-\frac{1}{3}a\y(\x^2\y^{2\rho-1})\in\gtm J(F)+(F)$. Finally, $\x^3\y^{\rho-1}=\frac{1}{3}\y^{\rho-1}(3\x^3+a\x\y^{2\rho})-\frac{1}{3}a\x\y^{3\rho-1}\in\gtm J(F)+(F)$. Now, as $k:=\min\{3\rho+1+\omega(\frac{d h}{d\y}),4\rho-1\}\geq\rho+2,2\rho+1,3\rho$, we have
$$
x^{\ell}\y^{k-\ell}=\begin{cases}
\x^4(x^{\ell-4}\y^{k-\ell})\in\gtm J(F)+(F)&\text{if $4\leq\ell\leq k$,}\\
\x^3\y^{\rho-1}(\y^{k-2-\rho})\in\gtm J(F)+(F)&\text{if $\ell=3$,}\\
\x^2\y^{2\rho-1}(\y^{k-2\rho-1})\in\gtm J(F)+(F)&\text{if $\ell=2$,}\\
\x\y^{3\rho-1}(\y^{k-3\rho})\in\gtm J(F)+(F)&\text{if $\ell=1$,}\\
\y^k=\y^{4\rho-1}(\y^{k-4\rho+1})\in\gtm J(F)+(F)&\text{if $\ell=0$,}
\end{cases}
$$
so $\gtm^k\subset\gtm J(F)+(F)$ and $F$ is $k$-quasidetermined.
\end{proof}

\subsection{Exchanging positiveness}
Along this work we will freely use the following facts.
\begin{lem}\label{trans}
Let $\varphi:A\to B$ be a homomorphism of rings. We have:
\begin{itemize}
\item[(i)] For each $\beta\in\Sper(B)$ there exists $\alpha\in\Sper(A)$ such that ${\rm sign}_\alpha(g)={\rm sign}_\beta(\varphi(g))$ for each $g\in A$.
\item[(ii)] If $f\in\psd(A)$, then $\varphi(f)\in\psd(B)$.
\end{itemize}
\end{lem}
\begin{proof}
By \cite[Prop.7.1.7]{bcr} the real spectral map $\Sper(\varphi):\Sper(B)\to\Sper(A),\ \beta\mapsto\varphi^{-1}(\beta)$ is well-defined. 

(i) Observe that $\varphi(g)\in\beta$ if and only if $g\in\alpha:=\varphi^{-1}(\beta)$ and $\varphi(g)\in\supp(\beta)$ if and only if $g\in\supp(\alpha)=\varphi^{-1}(\supp(\beta))$.

(ii) If $\beta\in\Sper(B)$, then $f\in\varphi^{-1}(\beta)$ (because $f\in\psd(A)$), so $\varphi(f)\in\beta$. Thus, $\varphi(f)\in\psd(B)$, as required.
\end{proof}

The following result will be useful when dealing with blow-ups in the proof of Theorem \ref{list2}.

\begin{lem}\label{blowup}
Let $A,B$ be two rings and $g\in B$ a non-nilpotent element. Denote the localization of $B$ at the multiplicative set $S:=\{g^n:\ n\geq0\}$ with $B_g$. Let $\varphi:A\to B_g$ be a homomorphism of rings and $f\in\psd(A)$. Write $\varphi(f)=\frac{b}{g^\ell}$ for some $\ell\geq1$ and assume $g$ does not divide $b$. Define 
$$
k:=\begin{cases}
1&\text{if $\ell$ is odd,}\\
2&\text{if $\ell$ is even.}
\end{cases}
$$
Then $g^kb\in\psd(B)$.
\end{lem}
\begin{proof}
By Lemma \ref{trans} we deduce $\varphi(f)\in\psd(B_g)$. Let ${\tt i}:B\to B_g,\ x\mapsto\frac{x}{1}$ be the canonical map. By \cite[Thm.13.3.7 \& Rem.13.3.8(ii)]{dst} the real spectrum $\Sper(B_g)$ is isomorphic via $\Sper({\tt i})$ to the constructible subset ${\mathcal U}:=\{\alpha\in\Sper(B):\ g^2>_\alpha0\}$ of $\Sper(B)$. This means that $g^kb>_\alpha0$ for each $\alpha\in{\mathcal U}$ because $\frac{g^kb}{g^{k+\ell}}\in\psd(B_g)$ and $k+\ell$ is even. As $k\geq1$, we have $g^kb\in\supp(\beta)$ for each $\beta\in\Sper(B)\setminus{\mathcal U}=\{\alpha\in\Sper(B):\ g\in\supp(\alpha)\}$. We conclude $g^kb\in\psd(B)$, as required.
\end{proof}

\subsection{Applications of Weierstrass polynomials and Puiseux series}
We prove next that series in two variables that are close enough in the $\gtm$-adic topology have similar structures. 

\begin{lem}\label{PU}
Let $P\in\kappa[[\y]][\x]$ be Weierstrass polynomials of degree $d$, let $U\in\kappa[[\x,\y]]$ be a unit and set $f:=PU$. For each $m\geq1$ there exists $r\geq1$ such that if $Q\in\kappa[[\y]][\x]$ is a Weierstrass polynomial of degree $d$, $V\in\kappa[[\x,\y]]$ is a unit with $V(0,0)=U(0,0)$ and $g:=QV$ satisfies $f-g\in\gtm_2^r$, then $P-Q,U-V\in\gtm_2^m$.
\end{lem}
\begin{proof}
Assume first $\omega(f)=d$ and let $r:=d+1$. Pick $Q,V$ satisfying the conditions of the statement. We have
$$
Q=V^{-1}UP+V^{-1}(g-f)\equiv V^{-1}UP\mod\gtm_2^r.
$$
We claim: $Q\equiv P\mod\gtm_2^{r-1}$.

Consider the homogeneous components of $P,Q,W:=V^{-1}U$ and write:
\begin{align*} 
&P:=p_d+p_{d+1}+\cdots+p_{d+k}+\cdots,\\ 
&Q:=q_d+q_{d+1}+\cdots+q_{d+k}+\cdots,\\ 
&W:=w_0+w_1+\cdots+w_k+\cdots, 
\end{align*} 
where $p_d$ and $q_d$ are homogeneous polynomials of degree $d$ monic with respect to $\x$ and 
$$
\deg_{\x}(p_{d+k}),\deg_{\x}(q_{d+k})\leq d-1
$$ 
for $k\geq 1$. As $Q\equiv WP \mod \gtm_2^r$ for each $k=0,\ldots,r-d-1$, we deduce 
$$
q_{d+k}=\sum_{i+j=d+k}p_iw_j.
$$
We claim: \em $q_{d+k}=p_{d+k}$ for $k=0,\ldots,r-d-1$, $w_0=1$ and $w_k=0$ for $k=1,\ldots,r-d-1$\em.

For $k=0$ we have $q_d=p_dw_0$, so $w_0=1$ and $q_d=p_d$. Assume $q_{d+j}=p_{d+j}$ for
$j=0,\ldots,k<r-d-1$ and $w_j=0$ for $j=1,\ldots,k$. Then:
$$
q_{d+k+1}=p_{d+k+1}w_0+p_{d+k}w_1+\cdots+p_dw_{k+1}=p_{d+k+1}+p_dw_{k+1}.
$$ 
If $w_{k+1}\ne0$,
$$
d\le\deg_{\x}(p_dw_{k+1})=\deg_{\x}(q_{d+k+1}-p_{d+k+1})
\leq\max\{\deg_{\x}(q_{d+k+1}),\deg_{\x}(p_{d+k+1})\}\leq d-1,
$$
which is impossible. Consequently, $w_{k+1}=0$ and $q_{d+k+1}=p_{d+k+1}$.

If $\omega(f)<d$, we substitute $\y=\t^\ell$ for some $2\leq\ell\leq d$ such that $\omega(f(\x,\t^\ell))=d$. Define $r:=m\ell+1$ and let $Q,V$ satisfy the conditions of the statement. Observe that by the previous case $P(\x,\t^\ell)\equiv Q(\x,\t^\ell)\mod(\x,\t)^{m\ell}$ and $U(\x,\t^\ell)\equiv V(\x,\t^\ell)\mod(\x,\t)^{m\ell}$. Consequently, $P\equiv Q\mod\gtm_2^m$ and $U\equiv V\mod\gtm_2^m$, as required. 
\end{proof}

Let $\kappa$ be a field of characteristic $0$ and $\ol{\kappa}$ its algebraic closure. Let $\ol{\kappa}[[\t^*]]=\bigcup_{p\geq1}\ol{\kappa}[[\t^{1/p}]]$ be the ring of Puiseux series with coefficients in $\ol{\kappa}$ and $\ol{\kappa}((\t^*))$ its field of fractions, which is an algebraically closed field.

\begin{lem}\label{roots}
Let $P\in\kappa[[\y]][\x]$ be a Weierstrass polynomial of degree $d$, let $\ol{\kappa}$ be the algebraic closure of $\kappa$ and let $\alpha_1,\ldots,\alpha_d\in\ol{\kappa}[[\y^*]]$ be the roots of $P$ in $\ol{\kappa}((\y^*))$. For each $m\geq1$ there exists $r\geq1$ such that if $Q\in\kappa[[\y]][\x]$ is a Weierstrass polynomial of degree $d$ with roots $\beta_1,\ldots,\beta_d\in\ol{\kappa}((\y^*))$ satisfying $P-Q\in\gtm_2^r$, then (after reordering the indices of the roots $\beta_i$) $\omega(\alpha_i-\beta_i)\geq m$ for each $i=1,\ldots,n$. 
\end{lem}
\begin{proof}
Write $P:=\x^d+\sum_{j=0}^{d-1}a_j\x^j$, where $a_j\in\kappa[[\y]]$. Let $Q:=\x^d+\sum_{j=0}^{d-1}b_j\x^j\in\kappa[[\y]][\x]$ be another Weierstrass polynomial of degree $d$. For each $r\geq1$ 
$$
P-Q=\sum_{j=0}^{d-1}(a_j-b_j)\x^j\in\gtm_2^r
$$
if and only if $a_j-b_j\in\gtm_1^{r-j}$ for $j=0,\ldots,d-1$ (where $\gtm_1$ is the maximal ideal of $\kappa[[\y]]$). By \cite[Prop.8.3.3]{che} the roots $\alpha_1,\ldots,\alpha_d,\beta_1,\ldots,\beta_d$ of $P,Q$ in $\ol{\kappa}((\y^*))$ belong to $\ol{\kappa}[[\y^{1/p}]]$ where $p:=d!$. Let ${\tt s}_j\in\Z[\z_1,\ldots,\z_d]$ be the $j$th symmetric elementary form and consider the polynomial system
$$
a_j(\y^p)-(-1)^j{\tt s}_{d-j}(\z_1,\ldots,\z_d)=0
$$
for $j=1,\ldots,d-1$. We have 
$$
a_j(\y^p)-(-1)^j{\tt s}_{d-j}(\beta_1,\ldots,\beta_d)=a_j(\y^p)-b_j(\y^p)\in\gtm_1^{(r-j)p}
$$
for $j=0,\ldots,d-1$. By Artin's approximation theorem there exists $k\geq1$ such that if $\omega(a_j(\y^p)-b_j(\y^p))\geq k$ for $j=0,\ldots,d$, then $\alpha_j-\beta_j\in\gtm_1^{mp}$. Thus, it is enough to take $r\geq\frac{k}{p}+d$.
\end{proof}

\begin{cor}\label{roots2}
Let $f\in\kappa[[\x,\y]]$ be such that $f(\x,0)\neq0$ and let $m\geq1$. Denote the algebraic closure of $\kappa$ with $\ol{\kappa}$. There exists $r\geq1$ such that if $g\in\kappa[[\x,\y]]$ and $f-g\in\gtm_2^r$, then $\omega(f(\zeta,\y)-g(\zeta,\y))>\min\{\omega(f(\zeta,\y)),m\}$ for each $\zeta\in\ol{\kappa}[[\y^*]]$. In addition, if $f,g$ are Weierstrass polynomials, $\omega(f(\zeta,\y)-g(\zeta,\y))>m$.
\end{cor}
\begin{proof}
As $f(\x,0)\neq0$ is a regular series with respect to $\x$ there exists by Weierstrass preparation theorem a Weierstrass polynomial $P\in\kappa[[\y]][\x]$ and a unit $U\in\kappa[[\x,\y]]$ such that $f=PU$. Let $d:=\omega(f(\x,0))$ and let $\alpha_1,\ldots,\alpha_d\in\ol{\kappa}[[\y^{1/(d!)}]]$ be the roots of $P$ in $\ol{\kappa}((\y^*))$ (use \cite[Prop.8.3.3]{che}).

If $r\geq d+1$, we have $\ini(g(\x,0))=\ini(f(\x,0))\neq0$ (so $\omega(g(\x,0))=d$) for each $g\in\kappa[[\x,\y]]$ such that $f-g\in\gtm_2^r$. For such a $g$ there exists by Weierstrass preparation theorem a Weierstrass polynomial $Q\in\kappa[[\y]][\x]$ and a unit $V\in\kappa[[\x,\y]]$ such that $g=QV$. In addition, $V(0,0)\x^d=\ini(g(\x,0))=\ini(f(\x,0))=U(0,0)\x^d$, so $c:=U(0,0)=V(0,0)\in\kappa\setminus\{0\}$. By \cite[Prop.8.3.3]{che} the roots $\beta_1,\ldots,\beta_d$ of $Q$ in $\ol{\kappa}((\y^*))$ belong to $\ol{\kappa}[[\y^{1/(d!)}]]$. For each $k\geq1$ there exists by Lemma \ref{PU} an integer $r:=r(k)\geq d+1$ such that if $f-g\in\gtm_2^r$, then $P-Q,U-V\in\gtm_2^k$. We will see that there exists $k$ large enough such that the corresponding $r(k)$ is the $r$ sought in the statement.

By Lemma \ref{roots} if $k\geq1$ is large enough, the roots of $Q$ suitably reordered satisfy $\omega(\beta_j-\alpha_j)>m$ for each $j=1,\ldots,d$. If $\zeta\in\ol{\kappa}[[\y^*]]$, then
\begin{align*}
&f(\zeta,\y)=U(\zeta,\y)\prod_{j=1}^d(\zeta-\alpha_j),\\
&g(\zeta,\y)=V(\zeta,\y)\prod_{j=1}^d(\zeta-\beta_j).
\end{align*}
As $\delta_j:=(\zeta-\alpha_j)-(\zeta-\beta_j)=\beta_j-\alpha_j$ and $\omega(\beta_j-\alpha_j)>m$, the difference
$$
P(\zeta,\y)-Q(\zeta,\y)=\prod_{j=1}^d(\zeta-\alpha_j)-\prod_{j=1}^d(\zeta-\beta_j)=\prod_{j=1}^d(\zeta-\beta_j+\delta_j)-\prod_{j=1}^d(\zeta-\beta_j)
$$
has order $>m$. Write $U:=c+h_1$, $V:=c+h_2$ and $U-V=h_3$ where each $h_i\in\gtm_2$. We have $f-g=P(U-V)+V(P-Q)$, so
\begin{align*}
&f(\zeta,\y)=P(\zeta,\y)(c+h_1(\zeta,\y)),\\
&f(\zeta,\y)-g(\zeta,\y)=P(\zeta,\y)h_3(\zeta,\y)+(c+h_2(\zeta,\y))(P(\zeta,\y)-Q(\zeta,\y))
\end{align*}
and the latter series has order $>\min\{\omega(f(\zeta,\y)),m\}$. In addition, if $f=P$ and $g=Q$, then $h_3=0$ and $\omega(f(\zeta,\y)-g(\zeta,\y))>m$, as required.
\end{proof}

\subsection{Applications of curve selection lemma and Puiseux series}
We recall here the following version of the curve selection lemma, which has relevant consequences for the prime cones of $\Sper(\kappa[[\x,\y]])$. As usual we identify $\Sper(\kappa)$ with the subspace of $\Sper(\kappa[[\x,\y]])$ consisting of all prime cones of $\kappa[[\x,\y]]$ with support $\gtm_2$.

\begin{lem}[Curve selection lemma]\label{dimen1}
Let $\kappa$ be a (formally) real field and $\beta\in\Sper(\kappa[[\x,\y]])$ a prime cone such that $\y>_\beta0$ and $\dim(\kappa[[\x,\y]]/\supp(\beta))=1$. Let $\alpha\in\Sper(\kappa)$ be such that $\beta\to\alpha$ and let $R(\alpha)$ be the real closure of $(\kappa,\leq_\alpha)$. Then there exists a homomorphism $\phi:\kappa[[\x,\y]]\to R(\alpha)[[\y]]$ such that $\phi(\y)=\y^q$ for some $q\geq1$ and $f\geq_\beta0$ if and only if $\phi(f)=f(\phi(\x),\y^q)\geq0$. In addition, $f\in\supp(\beta)$ if and only if $f(\phi(\x),\y^q)=0$.
\end{lem}
\begin{proof}
As ${\rm ht}(\supp(\beta))=1$, we have $\supp(\beta)$ is a principal prime ideal. Write $(P):=\supp(\beta)$ where $P\in\kappa[[\x,\y]]$ is irreducible. As $\y\not\in\supp(\beta)$, we have $P(\x,0)\neq0$. By Weierstrass preparation theorem we may assume $P\in\kappa[[\y]][\x]$ is a Weierstrass polynomial with respect to $\x$. Thus,
$$
{\tt j}:\kappa[[\y]]\hookrightarrow\kappa[[\x,\y]]/\supp(\beta)\cong\kappa[[\y]][\x]/(P).
$$
Write $\alpha_1:={\tt j}^{-1}(\beta)$. By \cite[Ex.II.3.13]{abr} the real closure of $(\kappa((\y)),\leq_{\alpha_1})$ is $R(\alpha)((\y^*))$. As $(\kappa((\y))[\x]/(P),\leq_\beta)$ is a finite algebraic extension of $(\kappa((\y)),\leq_{\alpha_1})$ and $\alpha_1={\tt j}^{-1}(\beta)$, we deduce $\varphi:\kappa((\y))[\x]/(P)\hookrightarrow R(\alpha)((\y^*))$. As $P\in\kappa[[\y]][\x]$ is a monic polynomial, there exist $\zeta\in R(\alpha)[[\y]]$ and $q\geq1$ such that $\varphi(\x)=\zeta(\y^{1/q})$. Define $\phi:\kappa[[\x,\y]]\to R(\alpha),\ f\mapsto f(\zeta,\y^q)$ and observe that $f\in\kappa[[\x,\y]]$ satisfies $f\geq_\beta0$ if and only if $\phi(f)\geq0$, as required.
\end{proof}

\begin{cor}\label{orderings}
Let $f\in\kappa[[\x,\y]]$, $M\in\kappa\setminus\{0\}$ and $n\geq1$. Let $C\subset\Sper(\kappa[[\x,\y]])$ and assume $f\geq_\beta0$ for each $\beta\in C$ with ${\rm ht}(\supp(\beta))=1$. We have:
\begin{itemize}
\item[(i)] There exists $r\geq1$ such that if $g\in\kappa[[\x,\y]]$ and $f-g\in\gtm_2^r$, then $g+M^2(\x^2+\y^2)^n>_\beta0$ for each $\beta\in C$ with ${\rm ht}(\supp(\beta))=1$.
\item[(ii)] If in addition $f(\x,0)\neq0$, there exists $r\geq1$ such that if $f-g\in\gtm_2^r$, then $g+M^2\y^{2n}>_\beta0$ for each $\beta\in C$ with ${\rm ht}(\supp(\beta))=1$. 
\end{itemize}
\end{cor}
\begin{proof}
We prove both statements simultaneously. Let $\beta\in C$ with ${\rm ht}(\supp(\beta))=1$. We distinguish two cases:

\noindent{\sc Case 1.} If $\y\in\supp(\beta)$, then $\supp(\beta)=(\y)$ and $f\geq_\beta0$ if and only if $f(\x,0)\geq_\beta0$. Assume first $f(\x,0)=0$ and let $r\geq 2n+1$. Set $h:=g+M^2(\x^2+\y^2)^n$ and observe that $h(\x,0)=M^2\x^{2n}u^2$ for some unit $u\in\kappa[[\x]]$ such that $u(0)=1$. Thus, $h+\supp(\beta)=h(\x,0)+\supp(\beta)>0$. 

Assume next $f(\x,0)\neq0$ and write $f(\x,0)=a\x^\ell u^2$ where $\ell\geq0$, $a\in\kappa\setminus\{0\}$ and $u\in\kappa[[\x]]$ is a unit such that $u(0)=1$. If $f-g\in\gtm_2^{\ell+1}$, then $g(\x,0)=a\x^\ell v^2$ where $v\in\kappa[[\x]]$ is a unit such that $v(0)=1$. Thus, 
\begin{align*}
&g+M^2\y^{2n}+\supp(\beta)=g(\x,0)+\supp(\beta)=f(\x,0)\frac{v^2}{u^2}+\supp(\beta)>0,\\
&g+M^2(\x^2+\y^2)^n+\supp(\beta)=g(\x,0)+M^2\x^{2n}+\supp(\beta)=f(\x,0)\frac{v^2}{u^2}+M^2\x^{2n}+\supp(\beta)>0.
\end{align*}

By Lemma \ref{roots2} there exists (maybe a larger) $r\geq1$ such that if $f-g\in\gtm_2^r$, then $\omega(f(\zeta,\y)-g(\zeta,\y))>\min\{\omega(f(\zeta,\y)),2n+1+\omega(f(\x,0))\}$ for each $\zeta\in\ol{\kappa}[[\y^*]]$.

\noindent{\sc Case 2.} Assume next $\y\not\in\supp(\beta)$ (without loss of generality we may assume that $\y>_{\beta}0$). Let $\alpha\in\Sper(\kappa)$ be such that $\beta\to\alpha$ and let $R(\alpha)$ be the real closure of $(\kappa,\leq_\alpha)$. By Lemma \ref{dimen1} there exists a homomorphism $\phi:\kappa[[\x,\y]]\to R(\alpha)[[\y]]$ such that $\phi(\y)=\y^q$ for some $q\geq1$ and $h\geq_\beta0$ if and only if $\phi(h)\geq0$. Write $\xi:=\varphi(\x)\in R(\alpha)[[\y]]$ and $\zeta:=\xi(\y^{1/q})\in R(\alpha)[[\y^{1/q}]]$. Consider the homomorphism 
$$
\psi:\kappa[[\x,\y]]\to R(\alpha)[[\y^{1/q}]],\ h\mapsto h(\zeta,\y)
$$ 
and observe that $h\geq_\beta0$ if and only if $\psi(h)\geq0$. Denote 
$$
s:=\omega(f(\zeta,\y)-g(\zeta,\y))>\min\{\omega(f(\zeta,\y)),2n+1+\omega(f(\x,0))\}\geq\min\{\omega(f(\zeta,\y)),2n+1\}
$$
and $\theta:=\frac{g(\zeta,\y)-f(\zeta,\y)}{\y^s}\in\kappa[[\y^{1/q}]]$ in case $f(\zeta,\y)-g(\zeta,\y)\neq0$ and $s:=2n+1$ and $\theta:=0$ otherwise. As $f\geq_\beta0$, we have
\begin{align*}
g(\zeta,\y)+M^2\y^{2n}&=f(\zeta,\y)+(g(\zeta,\y)-f(\zeta,\y))+M^2\y^{2n}\\
&=\begin{cases}
M^2\y^{2n}+\y^s\theta>0&\text{if $f(\zeta,\y)=0$,}\\
f(\zeta,\y)+M^2\y^{2n}+\y^s\theta>0&\text{if $f(\zeta,\y)\neq0$,}
\end{cases}\\
g(\zeta,\y)+M^2(\zeta^2+\y^2)^n&=f(\zeta,\y)+(g(\zeta,\y)-f(\zeta,\y))+M^2(\zeta^2+\y^2)^n\\
&=\begin{cases}
M^2(\zeta^2+\y^2)^n+\y^s\theta>0&\text{if $f(\zeta,\y)=0$,}\\
f(\zeta,\y)+M^2(\zeta^2+\y^2)^n+\y^s\theta>0&\text{if $f(\zeta,\y)\neq0$,}
\end{cases}
\end{align*}
so $g+M^2\y^{2n}>_\beta0$ and $g+M^2(\x^2+\y^2)^n>_\beta0$, as required.
\end{proof}

If we set
$$
\psd^+(\kappa[[\x,\y]]):=\{f\in\psd(\kappa[[\x,\y]]):\ f>_\alpha0\ \forall\alpha\in\Sper(\kappa[[\x,\y]]),\ \supp(\alpha)\neq\gtm_2\},
$$
we obtain the following result.

\begin{cor}\label{psdkxy}
Let $f\in\psd^+(\kappa[[\x,\y]])$. There exists $r\geq1$ such that if $g\in\kappa[[\x,\y]]$ and $f-g\in\gtm_2^r$, then $g\in\psd^+(\kappa[[\x,\y]])$.
\end{cor}
\begin{proof}
As $f\in\psd^+(\kappa[[\x,\y]])$, we have $f(\x,0)\neq0$. By Weierstrass preparation theorem there exists a Weierstrass polynomial $P\in\kappa[[\y]][\x]$ of degree $d$, $c\in\kappa\setminus\{0\}$ and a unit $U\in\kappa[[\x,\y]]$ such that $f=cPU^2$ and $U(0,0)=1$. A straightforward argument shows that $c\in\Sos{\kappa}$ and $d=2k$ is even. If $P=1$ (that is, $d=0$), it is enough to take $r=2$. So let us assume $d\geq2$ and observe that $P\in\psd^+(\kappa[[\x,\y]])$. 

Let $\ol{\kappa}$ be the algebraic closure of $\kappa$ and $\xi\in\ol{\kappa}[[\y^{1/p}]]$ (where $p=d!$) a root of $P$. Write $\xi:=\sum_{\ell\geq0}a_\ell\y^{\ell/p}\in\ol{\kappa}[[\y^{1/p}]]$. Denote $L_\xi:=\kappa[a_\ell:\ \ell\geq0]$. By \cite[Thm.2.2]{ck} the extension $L_\xi|\kappa$ is finite. As $P\in\psd^+(\kappa[[\x,\y]])$, we deduce from Lemma \ref{dimen1} that $\xi\not\in R(\alpha)[[\y^*]]$ for each $\alpha\in\Sper(\kappa)$. Thus, $L_\xi$ is not a (formally) real field. As $L_\xi|\kappa$ is finite, there exists $m_\xi$ such that $L_\xi=\kappa[a_0,\ldots,a_{m_\xi}]$. Consequently, \em each series $\zeta\in\ol{\kappa}[[\y^{1/p}]]$ such that $\omega(\zeta-\xi)>m_\xi$ satisfies $\zeta\not\in R(\alpha)[[\y^*]]$ for each $\alpha\in\Sper(\kappa)$\em.

Let $\xi_1,\ldots,\xi_d\in\ol{\kappa}[[\y^{1/p}]]$ be the roots of $P$ and set $m:=\max\{m_{\xi_k}:\ k=1,\ldots,d\}+1$. By Lemmas \ref{PU} and \ref{roots} there exists $r\geq d+1$ such that if $f-g\in\gtm_2^r$ and we write $g=cQV^2$ where $Q:=(\x-\theta_1)\cdots(\x-\theta_d)\in\kappa[[\y]][\x]$ is a Weierstrass polynomial, $V\in\kappa[[\x,\y]]$ is a unit with $V(0,0)=1$ and $\theta_1,\ldots,\theta_d\in\ol{\kappa}[[\y^{1/p}]]$ are the roots of $Q$ in $\ol{\kappa}((\y^*))$, then (after reordering the roots $\theta_k$ if necessary) $\omega(\xi_k-\theta_k)\geq m$ for each $k=1,\ldots,d$. Thus, the roots of $Q$ do not belong to $R(\alpha)[[\y^*]]$ for each $\alpha\in\Sper(\kappa)$. In addition, the involution 
$$
\sigma_\alpha:R(\alpha)[\sqrt{-1}][[\y^{1/p}]]\to R(\alpha)[\sqrt{-1}][[\y^{1/p}]],\ \lambda+\sqrt{-1}\mu\mapsto \lambda-\sqrt{-1}\mu
$$ 
satisfies: \em if $\eta\in R(\alpha)[\sqrt{-1}][[\y^{1/p}]]$ is a root of $Q$, then $\sigma_\alpha(\eta)$ is also a root of $Q$ of the same multiplicity\em.
 
Consequently, for each $\alpha\in\Sper(\kappa)$ we deduce $Q\in\Sosd{(R(\alpha)[[\y^{1/p}]][\x])}$. By Lemma \ref{dimen1} we deduce that $Q>_\beta0$ for each $\beta\in\Sper(\kappa[[\x,\y]])$ with ${\rm ht}(\supp(\beta))=1$. By \cite[Prop.VII.5.1]{abr} we conclude $Q>_\beta0$ for each $\beta\in\Sper(\kappa[[\x,\y]])$ such that $\supp(\beta)\neq\gtm_2$. As $Q\in\gtm_2$, we conclude $Q\in\psd^+(\kappa[[\x,\y]])$, so also $g=cQV^2\in\psd^+(\kappa[[\x,\y]])$, as required.
\end{proof}

\section{List of candidates}\label{s3}

The purpose of this section is to prove Theorem \ref{gp=s}. This will be conducted in several steps after some preliminary preparation. Denote $\x:=(\x_1,\ldots,\x_n)$ and $\|\x\|^2:=\x_1^2+\cdots+\x_n^2$.

\begin{lem}\label{genlist0} 
Let $\kappa$ be a (formally) real field and $f\in\kappa[[\x]]$ a series of order $\geq2s$. Then there exists $M\in\Sos{\kappa}$ such that $M^2\|\x\|^{2s}-f\in\psd(\kappa[[\x]])$. 
\end{lem}
\begin{proof}
As $\omega(f)\geq2s$, we can write $f=\sum_{|\nu|=2s}b_\nu(\x)\x^{\nu}$. Observe that $b_\nu(0)^2+1-b_\nu\in\Sos{\kappa[[\x]]}$ for each $\nu$, because the initial coefficient $b_\nu(0)^2+1-b_\nu(0)=(b_\nu(0)-\frac{1}{2})^2+3\frac{1}{2^2}\in\Sos{\kappa}$. In addition,
$$
\|\x\|^2-\x_i\x_j=\sum_{k\neq i,j}\x_k^2+\Big(\x_i-\frac{\x_j}{2}\Big)^2+3\Big(\frac{\x_j}{2}\Big)^2\in\Sos{\kappa[[\x]]}
$$ 
for each $1\leq i,j\leq n$. Let $\x^\nu$ be a monomial such that $|\nu|=2s$. Write
$$
\x^\nu=\prod_{(i,j)\in\Lambda}\x_i\x_j
$$
where $\#\Lambda=s$ and $1\leq i,j\leq n$. As $\|\x\|^2-\x_i\x_j\in\psd(\kappa[[\x]])$, we deduce
$$
\|\x\|^{2s}-\x^\nu=\|\x\|^{2s}-\prod_{(i,j)\in\Lambda}\x_i\x_j\geq_\alpha0
$$
for each $\alpha\in\Sper(\kappa[[\x]])$, so $\|\x\|^{2s}-\x^\nu\in\psd(\kappa[[\x]])$ for each $\nu$ with $|\nu|=2s$. Thus, 
$$
(b_\nu(0)^2+1)\|\x\|^{2s}-b_\nu(\x)\x^\nu\in\psd(\kappa[[\x]])
$$ 
for each $\nu$ satisfying $|\nu|=2s$. Consequently,
$$
\sum_{|\nu|=2s}(b_\nu(0)^2+1)\|\x\|^{2s}-f\in\psd(\kappa[[\x]]).
$$
If we set $M:=\sum_{|\nu|=2s}(b_\nu(0)^2+1)$, we obtain $M^2\|\x\|^{2s}-f\in\psd(\kappa[[\x]])$, as required.
\end{proof}

Given an ideal $\gta\subset\kappa[[\x]]$, we denote the minimal order of a series in $\gta$ with $\omega(\gta)$.

\begin{lem}\label{genlist} 
Let $\kappa$ be a (formally) real field and $A:=\kappa[[\x]]/\gta$ a formal ring such that $\psd(A)=\Sos{A}$. Then $\omega(\gta)=2$.
\end{lem}
\begin{proof}
Let $F\in\gta$ be a series of order $r>0$. Set $\x':=(\x_1,\ldots,\x_{n-1})$ and $\|\x'\|^2:=\x_1^2+\cdots+\x_{n-1}^2$. After a linear change of coordinates, we may assume by Weierstrass preparation theorem
$$
F:=\x_n^r+a_{r-1}\x_n^{r-1}+\cdots+a_1\x_n+a_0,
$$ 
where $a_j\in\kappa[[\x']]$ and $\omega(a_j)\geq r-j$ for $0\leq j\leq r-1$. 

\paragraph{}\label{genlist1}
We claim: \em There exists $N\in\Sos{\kappa}$ such that $N^2\|\x'\|^{2(r-j)}-a_j^2\in\psd(\kappa[[\x']])$ for each $0\leq j\leq r-1$\em.

As $\omega(a_j^2)\ge2(r-j)$, there exists by Lemma \ref{genlist0} $M_j\in\Sos{\kappa}$ such that $M_j\|\x'\|^{2(r-j)}-a_j^2\in\psd(\kappa[[\x']])$. Denote $M:=M_0+\cdots+M_{r-1}$ and $N:=M+1$. Observe that $N^2-M_j=(M+1)^2-M_j=M^2+M+(M-M_j)+1\in\Sos{\kappa}$ for each $0\leq j\leq r-1$. Thus, $N^2\|\x'\|^{2(r-j)}-a_j^2\in\psd(\kappa[[\x']])$ for each $0\leq j\leq r-1$, as claimed.

\paragraph{}\label{genlist2} Let us prove: \em The quadratic form $q:=(N^2r^2+1)^2\|\x'\|^2-\x_n^2\in\psd(A)$\em. 

Otherwise, there exists a prime cone $\beta\in\Sper(A)$ such that $q<_\beta0$. In particular, $\x_n\not\in\supp(\beta)$ and $\|\x'\|^2<_\beta\frac{1}{(N^2r^2+1)^2}\x_n^2$. As $A$ is a henselian ring, $\beta$ induces by \cite[Prop.II.2.4]{abr} an ordering $\alpha\in\Sper(\kappa)$ such that $\beta\to\alpha$. Consider the \em absolute value \em associated to $\beta$:
$$
|\,.\,|_{\beta}:A\to A,\ a\mapsto\begin{cases}
a&\text{if $a\geq_\beta0$,}\\
-a&\text{if $a<_\beta0$.}
\end{cases}
$$
By claim \ref{genlist1} we obtain 
$$
a_j^2\leq_\beta N^2\|\x'\|^{2(r-j)}<_\beta\frac{N^2}{(N^2r^2+1)^{2(r-j)}}\x_n^{2(r-j)}\quad\leadsto\quad|a_j|_\beta<_\beta\frac{N}{(N^2r^2+1)^{r-j}}|\x_n|_\beta^{r-j}, 
$$
for $0\leq j\leq r-1$. As the Weierstrass polynomial $F\in\gta$, it holds $F\in\supp(\beta)$. Hence,
\begin{equation*}
\begin{split} 
\x_n^{2r}&=\Big(\sum_{j=0}^{r-1}a_j(\x')\x_n^j\Big)^2=\sum_{j,k=0}^{r-1}a_j(\x')a_k(\x')\x_n^{j+k}\\
&\leq_\beta\sum_{j,k=0}^{r-1}|a_j(\x')|_{\beta}|a_k(\x')|_{\beta}|\x_n|_{\beta}^{j+k}\leq_\beta N^2\sum_{j,k=0}^{r-1}\frac{1}{(N^2r^2+1)^{2r-j-k}}|\x_n|_{\beta}^{2r-j-k}|\x_n|_\beta^{j+k}\\
&<_\beta N^2\sum_{j,k=0}^{r-1}\frac{|\x_n|_\beta^{2r}}{(N^2r^2+1)^{2r-j-k}}<_\beta\x_n^{2r}\sum_{j,k=0}^{r-1}\frac{N^2}{N^2r^2+1}\\
&=\x_n^{2r}\frac{N^2r^2}{N^2r^2+1}.
\end{split}
\end{equation*} 
As $x_n^{2r}>_\beta0$, we deduce $1<_\beta\frac{N^2r^2}{N^2r^2+1}$, which is a contradiction.

\paragraph{}\label{genlist3} We claim: \em $q\in\psd(A)\setminus\Sos{A}$ if $\omega(\gta)\geq 3$\em.

Otherwise, $q\in\Sos{A}$, so $q=h_1^2+\cdots+h_s^2+h$ for some $h_i\in\kappa[[\x]]$ and $h\in\gta$ satisfying $\omega(h)\geq 3$. Comparing initial forms in the previous expression, we find homogeneous polynomials $a_1,\ldots,a_r\in\kappa[\x]$ such that $q=a_1^2+\cdots+a_r^2$, which is impossible because $\kappa$ is a (formally) real field and $-1$ is not a sum of squares in $\kappa$.

We conclude: $\omega(\gta)\leq 2$, as required.
\end{proof}

\begin{cor}\label{genlistdim2}
Let $A:=\kappa[[\x,\y,\z]]/(\varphi)$ where $\varphi\in\kappa[[\x,\y,\z]]$ and $\kappa$ is a (formally) real field. Assume $\psd(A)=\Sos{A}$. Then $A$ is isomorphic to either $\kappa[[\x,\y]]$ or $\kappa[[\x,\y,\z]]/(\z^2-F)$ where $-F\in\kappa[[\x,\y]]$ is not a sum of squares and $\omega(F)\geq2$.
\end{cor}
\begin{proof}
By Lemma \ref{genlist} we have $\omega(\varphi)\leq 2$. After a linear change of coordinates we may assume (by Weierstrass preparation theorem) either $\varphi=\z+a(\x,\y)$ with $\omega(a)\geq1$ or $\varphi=\z^2+2a_1(\x,\y)+a_2(\x,\y)=(\z+a_1(\x,\y))^2+(a_2(\x,\y)-a_1(\x,\y)^2)$ where $\omega(a_k)\geq k$. After the change of coordinates $\z_1:=\z+a(\x,\y)$ in the first case and $\z_1:=\z+a_1(\x,\y)$ in the second, we may assume $\varphi=\z$ in the first case and $\varphi=\z^2-F(\x,\y)$ where $F\in\kappa[[\x,\y]]$ has order $\geq2$ in the second. As $\kappa$ is a (formally) real field, $\Sper(A)\neq\varnothing$. By \cite[Lem.6.3]{sch1} $A$ is a real (reduced) ring, so $-F\not\in\Sos{\kappa[[\x,\y]]}$, as required.
\end{proof}

In order to prove Theorem \ref{gp=s} let us find certain order restrictions for the series $F\in\kappa[[\x,\y]]$ such that the ring $A:=\kappa[[\x,\y,\z]]/(\z^2-F)$ has the property $\psd(A)=\Sos{A}$. Beforehand we need to characterize the positive semidefinite elements of $A$.

\subsection{Characterization of positive semidefinite elements.}\label{chpsd} 
Let $F\in\gtm_2\subset\kappa[[\x,\y]]$ where $\kappa$ is a (formally) real field and let $A:=\kappa[[\x,\y,\z]]/(\z^2-F)$. The ring $A$ is a rank $2$ free module over $\kappa[[\x,\y]]$. This means that the elements of $A$ are uniquely represented in the form $f+\z g$ where $f,g\in\kappa[[\x,\y]]$. The elements of $A$ are operated using the relation $\z^2=F(\x,\y)$. Consider the inclusion ${\tt i}:\kappa[[\x,\y]]\hookrightarrow A$. Denote 
$$
\psd(\{F\geq0\}):=\{f\in\kappa[[\x,\y]]:\ f\geq_\alpha0,\ \forall\alpha\in\Sper(\kappa[[\x,\y]])\text{ with }F\geq_\alpha0\},
$$
the maximal ideal of $\kappa[[\x,\y]]$ with $\gtm_2$ and the maximal ideal of $A$ with $\gtm_A$.

\begin{lem}\label{cluepsd}
Let $\alpha\in\Sper(\kappa[[\x,\y]])$ and let $\Sper({\tt i}):\Sper(A)\to\Sper(\kappa[[\x,\y]])$ be the real spectral map associated to ${\tt i}$. There exists a prime cone $\beta\in\Sper(A)$ such that $\Sper({\tt i})(\beta)=\alpha$ if and only if $F\geq_\alpha0$\em. In addition, $\supp(\beta)=\gtm_A$ if and only if $\supp(\alpha)=\gtm_2$.
\end{lem}
\begin{proof}
Let $\alpha\in\Sper(\kappa[[\x,\y]])$ and $R(\alpha)$ be the real closure of the ordered field 
$$
(K(\alpha):=\qf(\kappa[[\x,\y]]/\supp(\alpha)),\leq_\alpha).
$$

Suppose there exists $\beta\in\Sper(A)$ such that $\Sper({\tt i})(\beta)=\alpha$. Then $F=\z^2\geq_\beta0$ in $A$ implies $F\geq_\alpha0$. If $\supp(\beta)=\gtm_A$, then $\supp(\alpha)=\supp(\beta)\cap\kappa[[\x,\y]]=\gtm_A\cap\kappa[[\x,\y]=\gtm_2$.

Suppose conversely that $F\geq_\alpha0$. Define $\gta:=(\supp(\alpha),\z^2-F)/(\z^2-F)$. Then 
\begin{equation*}
\begin{split}
A\to A/\gta&\cong\kappa[[\x,\y,\z]]/(\supp(\alpha),\z^2-F)\cong\kappa[[\x,\y]][\z]/(\supp(\alpha),\z^2-F)\\
&\cong(\kappa[[\x,\y]]/\supp(\alpha))[\z]/(\z^2-F+\supp(\alpha))\hookrightarrow K(\alpha)[\z]/(\z^2-F+\supp(\alpha)).
\end{split}
\end{equation*}

As $F\geq_\alpha0$, the polynomial $\z^2-F$ has two roots in $R(\alpha)$. Let $\xi\in R(\alpha)$ be one of them and consider the evaluation homomorphism $K(\alpha)[\z]/(\z^2-F+\supp(\alpha))\to R(\alpha),\ P\mapsto P(\xi)$. Thus, we have a homomorphism $A\to R(\alpha)$ and the inverse image of the set of non-negative elements of $R(\alpha)$ defines a prime cone $\beta\in\Sper(A)$ such that $\Sper({\tt i})(\beta)=\alpha$. If $\supp(\alpha)=\gtm_2$ and $\Sper({\tt i})(\beta)=\alpha$, then $(\z^2,\gtm_2)=(\z^2-F,\gtm_2)\subset\supp(\beta)$. As $\supp(\beta)$ is prime, we deduce $\z\in\supp(\beta)$ and $\supp(\beta)=\gtm_A$, as required.
\end{proof}

\begin{cor}\label{psd3}
An element $f+\z g\in A$ (where $f,g\in\kappa[[\x,\y]]$) belongs to $\psd(A)$ if and only if $f\in\psd(\{F\geq0\})$ and $f^2-Fg^2\in\psd(\kappa[[\x,\y]])$. 
\end{cor}
\begin{proof}
Consider the involution $\sigma:A\to A,\ f+\z g\mapsto f-\z g$. Thus, $f+\z g\in\psd(A)$ if and only if $f-\z g\in\psd(A)$. These two elements are both positive semidefinite in $A$ if and only if $2f=(f+\z g)+(f-\z g)$ and $f^2-Fg^2=(f+\z g)(f-\z g)$ are positive semidefinite in $A$. As $f,f^2-Fg^2\in\kappa[[\x,\y]]$, we know by Lemma \ref{cluepsd} that $f,f^2-Fg^2\in\psd(A)$ if and only if $f,f^2-Fg^2\in\psd(\{F\geq0\})$. Let $\alpha\in\Sper(\kappa[[\x,\y]])$ be such that $F<_\alpha0$, so $-F>_\alpha0$ and $f^2-Fg^2\geq_\alpha0$. Thus, $f^2-Fg^2\in\psd(\{F\geq0\})$ if and only if $f^2-Fg^2\in\psd(\kappa[[\x,\y]])$. Consequently, the statement follows.
\end{proof}

\subsection{Proof of Theorem \ref{gp=s}}
After this preliminary work, we prove Theorem \ref{gp=s} in several steps. Let $A:=\kappa[[\x,\y,\z]]/(\z^2-F)$ where $\kappa$ is a (formally) real field and $F\in\kappa[[\x,\y]]$. Denote the maximal ideal of $\kappa[[\x,\y]]$ with $\gtm_2$.

\begin{lem}[General restriction]\label{genord} 
If $\psd(A)=\Sos{A}$, then $\omega(F)\leq 3$.
\end{lem}
\begin{proof}
Assume $\omega(F)\geq 4$. By Lemma \ref{genlist0} (applied for $s=2$) there exists $M\in\Sos{\kappa}$ such that $M^2(\x^2+\y^2)^2-F\in\psd(\kappa[[\x,\y]])$. As $M(\x^2+\y^2)\in\psd(\kappa[[\x,\y]])\subset\psd(\{F\geq0\})$, we deduce $\varphi:=M(\x^2+\y^2)+\z\in\psd(A)=\Sos{A}$. Thus, there exist $a_1,\ldots,a_r,b\in\kappa[[\x,\y,\z]]$ such that
$$
\varphi=M(\x^2+\y^2)+\z=a_1^2+\ldots+a_r^2+(\z^2-F)b.
$$
Comparing orders we conclude $1=\omega(M(\x^2+\y^2)+\z)=\omega(a_1^2+\ldots+a_r^2+(\z^2-F)b)$, which is a contradiction, so $\omega(F)\leq 3$, as required.
\end{proof}

\begin{lem}[Order $2$ restrictions]\label{ord2} 
If $\omega(F)=2$ and $\psd(A)=\Sos{A}$, then $F$ is right equivalent to one of the following:
\begin{itemize}
\item[(i)] $a\x^2$ where $a\not\in-\Sos{\kappa}$.
\item[(ii)] $a\x^2+\y^{2k+1}$ where $a\not\in-\Sos{\kappa}$ and $k\geq 1$.
\item[(iii)] $a\x^2+b\y^{2k}$ where $a\not\in-\Sos{\kappa}$, $b\neq0$ and $k\geq 1$.
\end{itemize}
\end{lem}
\begin{proof}
After a linear change of coordinates, we may assume (by Weierstrass preparation theorem) that there exist $a\in\kappa\setminus\{0\}$, a unit $U\in\kappa[[\x,\y]]$ such that $U(0,0)=1$ and a Weierstrass polynomial $P:=\x^2+2a_1(\y)\x+a_2(\y)\in\kappa[[\y]][\x]$ of degree $2$ (with $\omega(a_1)\geq1$ and $\omega(a_2)\geq2$) such that
$$
F=a(\x^2+2a_1\x+a_2)U^2=a((\x+a_1)^2+a_2-a_1^2)U^2.
$$
After the change of coordinates $(\x,\y,\z)\mapsto(\x-a_1,\y,\z U)$, we may assume
$$
F=a\x^2+\psi(\y)
$$
where $\omega(\psi)\geq2$. If $\psi=0$, then $F=a\x^2$. Otherwise, write $\psi=b\y^\ell u^\ell$ where $b\in\kappa\setminus\{0\}$ and $u\in\kappa[[\y]]$ is a unit such that $u(0)=1$. After the change of coordinates $(\x,\y,\z)\mapsto(\x,\frac{\y}{u},\z U)$, we may assume $$
F=a\x^2+b\y^\ell
$$
where $a,b\in\kappa\setminus\{0\}$. If $\ell=2k+1$ is odd (where $k\geq1$), after the change of coordinates $(\x,\y,\z)\mapsto(b^{k+1}\x,b\y,b^{k+1}\z)$ we can suppose $F=a\x^2+\y^{2k+1}$. Let us explain now the restrictions concerning the coefficients $a,b\in\kappa\setminus\{0\}$ in the statement:

\noindent{\sc Case 1.} If $F=a\x^2$ and $\psd(A)=\Sos{A}$, then $a\not\in-\Sos{\kappa}$ by \cite[Lem.6.3]{sch1}.

\noindent{\sc Case 2.} If $F=a\x^2+b\y^2$ and $\psd(A)=\Sos{A}$, then either $a$ or $b\not\in-\Sos{\kappa}$ by \cite[Lem.6.3]{sch1}. Interchanging $\x$ and $\y$, we may assume $a\not\in-\Sos{\kappa}$.

\noindent{\sc Case 3.} If $F=a\x^2+\y^{2k+1}$ (where $k\geq1$) and $a\in-\Sos{\kappa}$, then $\y\in\psd(\{F\geq0\})\setminus\Sos{A}\subset\psd(A)\setminus\Sos{A}$.

\noindent{\sc Case 4.} If $F=a\x^2+b\y^{2k}$ (where $k\geq2$) and $a\in-\Sos{\kappa}$, let us find $\varphi\in\psd(A)\setminus\Sos{A}$. We get 
\begin{equation}\label{b0}
(b^2+1)^2-b=b^4+b^2+\Big(b-\frac{1}{2}\Big)^2+\frac{3}{4}\in\Sos{\kappa}.
\end{equation}
Set $a:=-\sum_{i=1}^qa_i^2$ where $a_i\in\kappa\setminus\{0\}$.

Assume first that $k$ is even. Then
\begin{align*}
(b^2+1)^2\y^{2k}-a_1^2\x^2&=((b^2+1)^2-b)\y^{2k}+\sum_{k=2}^qa_k^2\x^2+(a\x^2+b\y^{2k})\\
&=((b^2+1)^2-b)\y^{2k}+\sum_{k=2}^qa_k^2\x^2+\z^2\in\psd(A)\cap\kappa[[\x,\y]]=\psd(\{F\geq0\}),\\
(b^2+1)\y^k&\in\psd(\kappa[[\x,\y]])\subset\psd(\{F\geq0\}).
\end{align*}
Thus, $\varphi:=(b^2+1)\y^k+a_1\x\in\psd(\{F\geq0\})\setminus\Sos{A}\subset\psd(A)\setminus\Sos{A}$, because it has order $1$.

Assume next that $k$ is odd (and recall that $k\geq2$). Then
\begin{align*}
(b^2+1)^2\y^{2k+2}&-a_1^2\x^2\y^2=((b^2+1)^2-b)\y^{2k+2}+\sum_{k=2}^qa_k^2\x^2\y^2+(a\x^2+b\y^{2k})\y^2\\
&=((b^2+1)^2-b)\y^{2k+2}+\sum_{k=2}^qa_k^2\x^2\y^2+\z^2\y^2\in\psd(A)\cap\kappa[[\x,\y]]=\psd(\{F\geq0\}),\\
(b^2+1)\y^{k+1}&\in\psd(\kappa[[\x,\y]])\subset\psd(\{F\geq0\}).
\end{align*}
Thus, $\varphi:=(b^2+1)\y^{k+1}+a_1\x\y\in\psd(\{F\geq0\})$. Let us check: $\varphi\not\in\Sos{A}$. 

Otherwise, there exist $h_1,\ldots,h_p,h\in\kappa[[\x,\y,\z]]$ such that
$$
\varphi=(b^2+1)\y^{k+1}+a_1\x\y=\sum_{i=1}^ph_i^2-(\z^2-a\x^2-b\y^{2k})h.
$$
Comparing initial forms, we find $c\in\Sos{\kappa}$ such that the quadratic form $\psi:=a_1\x\y+c\z^2-ca\x^2$ is a sum of squares of linear forms in the variables $\x,\y,\z$ and coefficients in $\kappa$, so $-\frac{a_1^2}{2}(a+c)^2-1=\psi(a_1,-\frac{c^2}{2}-\frac{a^2}{2}-\frac{1}{a_1^2},0)\in\Sos{\kappa}$, which is a contradiction because $\kappa$ is a (formally) real field. Thus, $\varphi\in\psd(\{F\geq0\})\setminus\Sos{A}\subset\psd(A)\setminus\Sos{A}$, as required.
\end{proof}

\begin{lem}[Order $3$ restrictions]\label{ord3} 
If $\omega(F)=3$ and $\psd(A)=\Sos{A}$, then $F$ is right equivalent to one of the following series:
\begin{itemize}
\item[(i)] $\x^2\y$ or $\x^2\y+(-1)^ka\y^k$ where $a\not\in-\Sos{\kappa}$ and $k\geq3$.
\item[(ii)] $\x^3+a\x\y^2+b\y^3$ irreducible.
\item[(iii)] $\x^3+a\y^4$, $\x^3+\x\y^3$ or $\x^3+\y^5$ where $a\not\in\Sos{\kappa}$.
\end{itemize}
\end{lem}
\begin{proof}[First part of the proof of Lemma \em\ref{ord3}]
After a linear change of coordinates (Tschirnhaus trick) we may assume that the initial form of $F$ is of type $P:=\lambda(\x^3+a\x\y^2+b\y^3)$ for some $\lambda\in\kappa\setminus\{0\}$. After the change of coordinates $(\x,\y,\z)\mapsto(\lambda\x,\lambda\y,\lambda^2\z)$, we may assume $P:=\x^3+a\x\y^2+b\y^3$. If $P$ is reducible, after a linear change of coordinates that only involves the variables $\x,\y$ we can suppose that $P$ is a homogeneous polynomial of one of the following types:
$$
\x^3,\x^2\y,\y(\x^2-a\y^2) 
$$
where $a\neq0$. If $P$ is either irreducible or $P=\x(\x^2+a\y^2)$ with $a\neq0$, then its discriminant is non-zero and $P$ is $3$-determined as we have seen in Example \ref{examples}(iii). Thus, after a change of coordinates we may assume $F=P$. 

If $F=\y(\x^2-a\y^2)$ and $a\in-\Sos{\kappa}$, then $\y\in\psd(\{F\geq0\})\setminus\Sos{A}\subset\psd(A)\setminus\Sos{A}$. Thus, $F=\x^2\y-a\y^3$ where $a\not\in-\Sos{\kappa}$.

Assume next $P:=\x^2\y$, but $F\neq\x^2\y$. Let $s\geq4$ be the degree of the next non-zero homogeneous form of $F$ and set $F:=\x^2\y+a\y^s+b\x\y^{s-1}+\x^2\varphi$ where $a,b\in\kappa$ and $\varphi\in\gtm_2^{s-2}$. After the change of coordinates $(\x,\y)\mapsto(\x-\frac{1}{2}b\y^{s-2},\y-\varphi)$ we may assume $F:=\x^2\y+a\y^s+\psi$ where $\psi\in\gtm^{s+1}$. If $a\neq0$, after an additional change of coordinates we can suppose $F:=\x^2\y+a\y^s$ because $\x^2\y+a\y^s$ is $s$-determined (Example \ref{examples}(ii)). If $a=0$, we iterate the previous process until we find $\ell>s$ such that $F$ is right equivalent to $\x^2\y+a'\y^\ell$ for some $a'\neq0$ or we conclude that $F$ is right equivalent to $\x^2\y$ if such an $\ell$ does not exist.

If $F:=\x^2\y+a\y^{2k+1}$ and $a\in\Sos{\kappa}$, then $\y\in\psd(\{F\geq0\})\setminus\Sos{A}\subset\psd(A)\setminus\Sos{A}$. If $F:=\x^2\y-a\y^{2k}$ and $a\in\Sos{\kappa}$, then $\x^2\y=\z^2+a\y^{2k}$ and $\y\in\psd(\{F\geq0\})\setminus\Sos{A}\subset\psd(A)\setminus\Sos{A}$. 
\end{proof}

Before finishing the proof of Lemma \ref{ord3} we need an intermediate result.

\begin{lem}\label{clue0}
Let $F\in\kappa[[\x,\y]]$ be a series with inital form $\x^3$. Then 
\begin{itemize}
\item[(i)] There exist a unit $U\in\kappa[[\x,\y]]$ with $U(0,0)=1$, a unit $W\in\kappa[[\y]]$ with $W(0)=1$, $b,c\in\kappa$ and $k,\ell\geq0$ such that $F$ is right equivalent to $(\x^3+b\x\y^{k+3}+c\y^{\ell+4}W)U^2$. 
\item[(ii)] If $c\neq0$ and $k\geq\ell+1$, we may assume $b=0$ and $W=1$. 
\item[(iii)] If $\psd(A)=\Sos{A}$, then either $k=0$ or $\ell\leq 1$ and we may assume $U=1$.
\end{itemize} 
\end{lem}
\begin{proof}
(i) By the Weierstrass preparation theorem and the Tschirnhaus trick there exist a Weierstrass polynomial $P:=\x^3+B(\y)x+C(\y)\in\kappa[[\y]][\x]$ and a unit $U_0\in\kappa[[\x,\y]]$ such that $F=PU_0$. As the initial form of $F$ is $\x^3$, we have in addition $U_0(0,0)=1$, $\omega(B)\geq 3$ and $\omega(C)\geq 4$. Let $b,c\in\kappa$, let $k,\ell\geq0$ be integers and $V,W_0\in\kappa[[\y]]$ units such that $V(0)=1$, $W_0(0)=1$ and $F=(\x^3+b(V\y)^{k+3}\x+cW_0\y^{\ell+4})U_0$. After the change of coordinates $V\y\mapsto\y$ we assume $V=1$, so $F=(\x^3+b\y^{k+3}x+cW\y^{\ell+4})U^2$ for units $U\in\kappa[[\x,\y]]$ and $W\in\kappa[[\y]]$ such that $U(0,0)=1$ and $W(0)=1$.

(ii) If $c\neq0$ and $k\geq\ell+1$, then $W':=W+\frac{b}{c}\x\y^{k-\ell-1}$ is a unit. After the change of coordinates $W'\y\mapsto\y$, we get $F=(\x^3+c\y^{\ell+4})U'^2$ for a unit $U'\in\kappa[[\x,\y]]$ with $U'(0,0)=1$.

(iii) After the change of coordinates $\z\rightarrow U\z$ we assume $F=\x^3+b\x\y^{k+3}+c\y^{\ell+4}W$ with the restrictions described in (i) and (ii). We claim: \em If $k\geq 1$ and $\ell\geq 2$, there exists $M\in\Sos{\kappa}$ such that
$\varphi=\x+M^2\y^2\in\psd(\{F\geq0\})\setminus\Sos{A}\subset\psd(A)\setminus\Sos{A}$\em. 

It is enough to find $M\in\Sos{\kappa}$ such that: \em if $\alpha\in\Sper(\kappa[[\x,\y]])$ satisfies $\varphi<_\alpha0$, then $F<_\alpha0$\em. 

If $\y\in\supp(\alpha)$, then $\varphi+\supp(\alpha)=\x+\supp(\alpha)$ and $F+\supp(\alpha)=\x^3+\supp(\alpha)$ and both have the same sign with respect to $\alpha$.

In order to find the suitable $M\in\Sos{\kappa}$, let us make first some computations valid for each $M\in\Sos{\kappa}$. If $\y\not\in\supp(\alpha)$ and $\varphi<_\alpha0$, then $\x<_\alpha-M^2\y^2$, so $-\x>_\alpha M^2\y^2>_\alpha0$. As $k\geq1$, we have $((b^2+1)\y^4+b\y^{k+3})>_\alpha0$. Thus,
$$
-\x b\y^{k+3}>_\alpha-\x(-(b^2+1))\y^4)>_\alpha-(b^2+1)M^2\y^6,
$$
so $\x b\y^{k+3}<_\alpha(b^2+1)M^2\y^6$ and $\x^3<_\alpha-M^6\y^6$. Consequently, as $\ell\geq2$,
$$
F=(\x^3+b\x\y^{k+3}+c\y^{\ell+4}W)<_\alpha(-M^6+(b^2+1)M^2+(c^2+1))\y^6<_\alpha0.
$$
To guarantee the last inequality for each $\alpha\in\Sper(\kappa[[\x,\y]])$ such that $\varphi<_\alpha0$, it is enough to find $M\in\Sos{\kappa}$ such that $M^6-(b^2+1)M^2-(c^2+1)\in\Sos{\kappa}$. We choose $M:=2(b^2+1)(c^2+1)\in\Sos{\kappa}$ and observe that
\begin{multline*}
M^6-(b^2+1)M^2-(c^2+1)\\
=(c^2+1)(16(b^2+1)^3(c^2+1)^4-1)(4(b^2+1)^3(c^2+1)-1)\in\Sos{\kappa},
\end{multline*}
as required.
\end{proof}

We are now ready to finish the proof of Lemma \ref{ord3}.

\begin{proof}[Second part of the proof of Lemma \em\ref{ord3}]
If $\psd(A)=\Sos{A}$, there exist by Lemma \ref{clue0} a unit $W\in\kappa[[\y]]$ with $W(0)=1$ and $b,c\in\kappa$ such that $F$ is one among
\begin{itemize}
\item[(i)] $\x^3+b\x\y^{3}+c\y^{4}W$ where $c\neq0$,
\item[(ii)] $\x^3+b\x\y^{3}+c\y^{\ell+5}W$ where $b\neq0$ and $\ell\geq0$,
\item[(iii)] $\x^3+b\x\y^{4}+c\y^{5}W$ where $c\neq0$.
\end{itemize}
We now approach these three cases:

(i) If $F=\x^3+b\x\y^{3}+c\y^{4}W$ and $c\neq0$, after the change of coordinates
\begin{align*}
\x&\mapsto\x+\frac{b^4}{256c^3}\x^2-\frac{b^3}{24c^2}\x\y+\frac{b^2}{8c}\y^2,\\
\y&\mapsto\y-\frac{b}{4c}\x
\end{align*}
and we may assume $F=\x^3+c\y^4+\psi$ where $\psi\in\gtm_2^5$. As $\x^3+c\y^4$ is by Example \ref{examples}(v) $4$-determined, we can suppose after an additional change of coordinates that $F=\x^3+c\y^4$. If $c\in-\Sos{\kappa}$, we have $\x^3=\z^2-c\y^4\in\Sos{A}$, so $\x\in\psd(\{F\geq0\})\setminus\Sos{A}\subset\psd(A)\setminus\Sos{A}$. Thus, $F=\x^3+c\y^4$ where $c\not\in-\Sos{\kappa}$.

(ii) If $F=\x^3+b\x\y^{3}+c\y^{\ell+5}W$ and $b\neq0$, after the change of coordinates $(\x,\y,\z)\mapsto(b^2\x,b\y,b^3\z)$ we can suppose $b=1$. As $\x^3+\x\y^3$ is by Example \ref{examples}(iv) $5$-determined, we can suppose after an additional change of coordinates that $F=\x^3+\x\y^3$ if $\ell\geq1$. Otherwise $\ell=0$, $c\neq0$ and $F=\x^3+\x\y^3+c\y^5W$. After the change of coordinates
\begin{align*}
\x&\mapsto\x-c\y^2-\frac{1}{3}c^3\x^2-c^2\x\y-\frac{1}{3}c^6\x^3-\frac{5}{3}c^5\x^2\y-2c^4\x\y^2-\frac{5}{9}c^3\y^3,\\
\y&\mapsto\y+c\x-\frac{4}{3}c^2\y^2
\end{align*}
there exists a series $\psi\in\gtm_2^6$ such that $F=\x^3+\x\y^3+\psi$. As $\x^3+\x\y^3$ is $5$-determined, $F$ is right equivalent to $\x^3+\x\y^3$ and we suppose $F=\x^3+\x\y^3$.

(iii) If $F=\x^3+b\x\y^{4}+c\y^{5}W$ and $c\neq0$, after the change of coordinates
\begin{align*}
\x&\mapsto\x-\frac{4b^5}{9375c^4}\x^3+\frac{b^4}{125c^3}\x^2\y-\frac{4b^3}{75c^2}\x\y^2+\frac{2b^2}{15c}\y^3,\\
\y&\mapsto\y-\frac{b}{5c}\x,
\end{align*}
we assume $F=\x^3+c\y^5+\psi$ where $\psi\in\gtm_2^6$. As $\x^3+c\y^5$ is by Example \ref{examples}(v) $5$-determined, there exists an additional change of coordinates after which $F=\x^3+c\y^5$. After the linear change of coordinates $(\x,\y,\z)\mapsto(c^2\x,c\y,c^3\z)$ we get $F=\x^3+\y^5$, as required.
\end{proof}

\subsubsection{The non-principal case}
To finish the proof of Theorem \ref{gp=s} we explore the case of a ring $A:=\kappa[[\x,\y,\z]]/\gta$ of dimension $2$ such that $\gta$ is not a principal ideal, but $A$ has the property $\psd(A)=\Sos{A}$. Before we need the following example.

\begin{example}\label{dim1}
Let $\kappa$ be a (formally) real field and $\varphi\in\kappa[[\x,\y]]$ such that $A:=\kappa[[\x,\y]]/(\varphi)$ has dimension $1$ and satisfies $\psd(A)=\Sos{A}$. Then $A$ is isomorphic to either $\kappa[[\x]]$ or $\kappa[[\x,\y]]/(\x^2-b\y^2)$ where $b\not\in-\Sos{\kappa}$.

First, by Lemma \ref{genlist} $\omega(\varphi)\leq2$. If $\omega(\varphi)=1$, we may assume $\varphi=\y$ and $A\cong\kappa[[\x]]$. Assume next $\omega(\varphi)=2$. By Weierstrass preparation theorem we can suppose after a linear change of coordinates $\varphi=\x^2+2a_1(\y)\x+a_2(\y)\in\kappa[[\x,\y]]$ where $\omega(a_i)\geq i$. We write
$$
\varphi=(\x+a_1)^2+a_2-a_1^2
$$ 
and after the change of coordinates $\x\mapsto\x-a_1$ we assume
$$
\varphi=\x^2-b\y^\ell u^\ell
$$
where $b\in\kappa$, $\ell\geq2$ and $u\in\kappa[[\y]]$ is a unit such that $u(0)=1$. After the change $u\y\mapsto\y$ we may assume $u=1$. By \cite[Lem.6.3]{sch1} the ideal $(\varphi)$ is real radical, so $b\neq0$ and if $\ell$ is even, then $b\not\in-\Sos{\kappa}$. 

If $\ell\geq4$, then by \eqref{b0} $(b^2+1)^2\y^4-b\y^\ell\in\psd(\kappa[[\y]])$. By Lemma \ref{psd3} adapted to two variables $\varphi:=(b^2+1)\y^2+\x\in\psd(A)\setminus\Sos{A}$ (because $\varphi$ has order $1$). Thus, $\ell\leq 3$.

If $\ell=3$, then $\x^2=(b\y)\y^2$, so $b\y\in\psd(A)\setminus\Sos{A}$ (because it has order $1$). Consequently, $\ell=2$ and $A$ is isomorphic to $\kappa[[\x,\y]]/(\x^2-b\y^2)$ where $b\not\in-\Sos{\kappa}$, as required.\qed
\end{example} 

\begin{thm}\label{nonprincipal}
Let $\gta$ be a non-principal ideal of $\kappa[[\x,\y,\z]]$ such that $A:=\kappa[[\x,\y,\z]]/\gta$ has dimension $2$. Then $\psd(A)=\Sos{A}$ if and only if $A$ is isomorphic to $\kappa[[\x,\y,\z]]/(\z\x,\z\y)$.
\end{thm}
\begin{proof}
Assume first that $A=\kappa[[\x,\y,\z]]/(\z\x,\z\y)$ and let $f\in\psd(A)$. If $f$ is a unit, then $f=bU^2$ where $b\in\kappa\setminus\{0\}$ and $U(0,0,0)=1$. If $\gtm$ is the maximal ideal of $A$, then $A/\gtm=\kappa$. As $f\in\psd(A)$, then $b\in\psd(\kappa)=\Sos{\kappa}$, so $f\in\Sos{A}$. As a consequence, we assume that $f$ is not a unit. As $\z\x=0,\z\y=0$, we have $f=f_1(\x,\y)+f_2(\z)$ where $f_1\in\kappa[[\x,\y]]$ and $f_2\in\kappa[[\z]]$ satisfy $f_1(0,0)=0$ and $f_2(0)=0$. Observe that $f_1\in\psd(\kappa[[\x,\y]])$ and $f_2\in\psd(\kappa[[[\z]])$ because $\gta:=(\z\x,\z\y)=(\z)\cap(\x,\y)$, $A/((\z)/\gta)\cong\kappa[[\x,\y]]$ and $A/((\x,\y)/\gta)\cong\kappa[[\z]]$. Thus, there exist $a_i\in\kappa[[\x,\y]]$ and $b_j\in\kappa[\z]$ such that $f=f_1+f_2=\sum_ia_i^2+\sum_jb_j^2\in\Sos{A}$. 

Assume next $\psd(A)=\Sos{A}$. By Lemma \ref{genlist} we know $\omega(\gta)\leq2$. If $\omega(\gta)=1$, we may assume $\z\in\gta$. But this is impossible because $\dim(A)=2$ and $\gta$ is not a principal ideal. Thus, $\omega(\gta)=2$ and by \cite[Lem.6.3]{sch1} $\gta$ is a real radical ideal. Let $\gta=\gtp_1\cap\cdots\cap\gtp_r\cap\gtq_1\cap\cdots\cap\gtq_s$ be the irredundant primary decomposition of $\gta$ and assume that ${\rm ht}(\gtp_i)=1$ for each $i=1,\ldots,r$ and ${\rm ht}(\gtq_j)=2$ for $j=1,\ldots,s$. Let $\gta_1:=\gtp_1\cap\cdots\cap\gtp_r$ and $\gta_2:=\gtq_1\cap\cdots\cap\gtq_s$. As each ideal $\gtp_i$ is prime and has height $1$, there exist $\varphi_i\in\kappa[[\x,\y,\z]]$ irreducible such that $\gtp_i=(\varphi_i)$ for $i=1,\ldots,r$. It holds $\gta_1=(\varphi)$ where $\varphi:=\prod_{i=1}^r\varphi_i$. We claim: $\gta=(\varphi)\cdot\gta_2$.

The inclusion right to left is clear, so let $f\in\gta=\gta_1\cap\gta_2$. As $\gta_1=(\varphi)$, there exists $h\in\kappa[[\x,\y,\z]]$ such that $f=\varphi h$. Let us check: $h\in\gta_2$. 

Otherwise, we may assume $h\not\in\gtq_1$. As $(\prod_{i=1}^r\varphi_i)h=\varphi h\in\gta_2\subset\gtq_1$, we can suppose $\varphi_1\in\gtq_1$, so $\gtp_1\subset\gtq_1$, which is a contradiction because the primary decomposition of the real radical ideal $\gta$ is irredundant. Consequently, $h\in\gta_2$ and $f\in(\varphi)\cdot\gta_2$, so $\gta=(\varphi)\cdot\gta_2$. 

Thus, $2=\omega(\gta)=\omega(\varphi)+\omega(\gta_2)$, so $\omega(\varphi)=\omega(\gta_2)=1$. We may assume $\varphi=\z$ and there exists $\psi\in\gta_2$ such that $\omega(\psi)=1$. We claim: \em The initial form of $\psi$ is not a multiple of $\z$\em.

Otherwise, we assume that the initial form of $\psi$ is equal to $\z$. By Weierstrass preparation theorem we may write $\psi=\z+2g(\x,\y)$ for some $g\in\kappa[[\x,\y]]$ of order $\geq 2$. Observe that 
$$
\z(\z+2g)=(\z+g)^2-g^2\in\gta.
$$ 
As $\omega(g)\geq2$, there exists by Lemma \ref{genlist0} $M\in\Sos{\kappa}$ such that $M^2(\x^2+\y^2)^2-g^2\in\psd(\kappa[[\x,\y]])$. Using the chain of homomorphisms 
$$
\kappa[[\x,\y]]\hookrightarrow\kappa[[\x,\y,\z]]\to A
$$
one deduces (using that $(\z+g)^2-g^2\in\gta$)
\begin{multline*}
(M(\x^2+\y^2)-g-\z)(M(\x^2+\y^2)+g+\z)=M^2(\x^2+\y^2)^2-(g+\z)^2\\
=M^2(\x^2+\y^2)^2-g^2\in\psd(A) 
\end{multline*}
and $(M(\x^2+\y^2)-g-\z)+(M(\x^2+\y^2)+g+\z)=2M(\x^2+\y^2)\in\psd(A)$. Thus, 
$$
M(\x^2+\y^2)-g-\z,M(\x^2+\y^2)+g+\z\in\psd(A)\setminus\Sos{A}
$$ 
because they have order $1$ and $\omega(\gta)=2$. As $\psd(A)=\Sos{A}$, we deduce the initial form of $\psi$ is not a multiple of $\z$.

Thus, we assume $\x\in\gta_2$ and observe that $\kappa[[\x,\y,\z]]/\gta_2\cong\kappa[[\y,\z]]/\gta_2'$ for some ideal $\gta_2'$ of $\kappa[[\y,\z]]$. As ${\rm ht}(\gta_2)=2$, we deduce ${\rm ht}(\gta_2')=1$, so $\gta_2'$ is a principal ideal and there exists $\phi\in\kappa[[\y,\z]]$ such that $\gta_2'=(\phi)$. Note that $\gta_2=(\x,\phi)$ and $\gta=(\z\x,\z\phi)$. Define $\gtb:=(\z\phi)$ and $B:=\kappa[[\y,\z]]/\gtb$. We claim: $\psd(B)=\Sos{B}$.

Consider the inclusion of rings $B:=\kappa[[\y,\z]]/\gtb\hookrightarrow A/\gta$. If $f\in\psd(B)$, then $f\in\psd(A)=\Sos{A}$, so there exist $a_1,\ldots,a_p,b_1,b_2\in\kappa[[\x,\y,\z]]$ such that 
$$
f(\y,\z)=a_1^2(\x,\y,\z)+\cdots+a_p^2(\x,\y,\z)+\z\x b_1(\x,\y,\z)+\z\phi(\y,\z)b_2(\x,\y,\z). 
$$
Substituting $\x=0$ in the previous equality we deduce 
$$
f(\y,\z)=a_1^2(0,\y,\z)+\cdots+a_p^2(0,\y,\z)+\z\phi(\y,\z)b_2(0,\y,\z), 
$$
so $f\in\Sos{B}$. Thus, $\psd(B)=\Sos{B}$. 

By Example \ref{dim1} we deduce that $B$ is isomorphic to $\kappa[[\y,\z]]/(\z^2-b\y^2)$ for some $b\not\in-\Sos{\kappa}$. This means that $\phi\in\kappa[[\y,\z]]$ has order $1$ and its initial form is not a multiple of $\z$. After a change of coordinates, we may assume $\phi=\y$. Consequently, $\gta=(\z\x,\z\y)$, as required.
\end{proof}

\section{Polynomial density and polynomial reduction}\label{s4}

In this section we develop several tools that will be crucial to prove Theorem \ref{list2}. We need to analyze the good properties concerning `polynomial approximation' of the positive definite elements of the rings $\kappa[[\x,\y,\z]]/(\z^2-F)$ where $F\in\kappa[[\x,\y]]$ and $\kappa$ is a (formally) real field.

\subsection{Positive definite elements.}\label{pde}
Let $F\in\kappa[[\x,\y]]$ and $A:=\kappa[[\x,\y,\z]]/(\z^2-F)$. Denote 
\begin{align*}
\psd^+(\{F\geq0\}):=&\{f\in\kappa[[\x,\y]]:\ f\in\psd(\{F\geq0\}),\ f>_\alpha0\\ &\hspace{4.9cm}\forall\alpha\in\Sper(\kappa[[\x,\y]]),\ \supp(\alpha)\neq\gtm_2,\ F\geq_\alpha0\},\\
\psd^+(A):=&\{f+\z g\in A:\ f+\z g\in\psd(A),\ f+\z g>_\beta0\ \forall\beta\in\Sper(A), \supp(\beta)\neq\gtm_A\},\\
\psd^{\oplus}(A):=&\{f+\z g\in A:\ f\in\psd^+(\{F\geq0\}),\ f^2-Fg^2\in\psd^+(\kappa[[\x,\y]])\}.
\end{align*}

The following result allows us to construct elements of $\psd^{\oplus}(A)$ from elements of $\psd(A)$ that are very similar to the original ones (and as close as needed in the $\gtm_2$-adic topology).

\begin{lem}[Construction of positive semidefinite elements]\label{pd+}
We have:
\begin{itemize}
\item[(i)] $\psd^{\oplus}(A)\subset\psd^+(A)$.
\item[(ii)] Let $f+\z g\in\psd(A)$ and let $f_i\in\kappa[[\x,\y]]$ be such that $\sqrt[r]{(f+\sum_{i=1}^pf_i^2,g)}=\gtm_2$ and $\sum_{i=1}^pf_i^2>_\beta0$ for those $\beta\in\Sper(\kappa[[\x,\y]])$ satisfying: $\supp(\beta)\neq\gtm_2$, $F\geq_\beta0$ and either $f\in\supp(\beta)$ or $f^2-Fg^2\in\supp(\beta)$. Then $f+\sum_{i=1}^pf_i^2+\z g\in\psd^{\oplus}(A)$.
\item[(iii)] Let $f+\z g\in\psd(A)$ be such that $f(\x,0)\neq0$, $f^2(\x,0)-F(\x,0)g^2(\x,0)\neq0$ and $g\neq0$. For each $M\in\kappa\setminus\{0\}$ there exists a finite set $S\subset\N$ such that $f+M^2\y^{2n}+\z g\in\psd^{\oplus}(A)$ for each $n\in\N\setminus S$.
\item[(iv)] Let $f+\z g\in\psd(A)$ be such that $g\neq0$. Fix $M\in\kappa\setminus\{0\}$ and let $h_n$ be either $M^2(\x^2+\y^2)^n$ or $M^2(\x^{2n}+\y^{2n})$. There exists a finite set $S\subset\N$ such that $f+h_n+\z g\in\psd^{\oplus}(A)$ for each $n\in\N\setminus S$.
\item[(v)] If $f\in\psd(\{F\geq0\})$, then $f+\z f=(1+\z)f\in\psd(A)$.
\end{itemize}
\end{lem}
\begin{proof}
(i) Assume $f\in\psd^+(\{F\geq0\})$ and $f^2-Fg^2\in\psd^+(\kappa[[\x,\y]])$. By Lemma \ref{cluepsd} we have 
\begin{align*}
&2f\in\psd^+(\{F\geq0\})\subset\psd^+(A),\\
&f^2-Fg^2\in\psd^+(\kappa[[\x,\y]])\subset\psd^+(A).
\end{align*} 
This means $(f+\z g)+(f-\z g)=2f\in\psd^+(A)$ and $(f+\z g)(f-\z g)=f^2-Fg^2\in\psd^+(A)$, so both $f+\z g,f-\z g\in\psd^+(A)$.

(ii) As $f\in\psd(\{F\geq0\})$ and $\sum_{i=1}^pf_i^2>_\beta0$ for those $\beta\in\Sper(\kappa[[\x,\y]])$ satisfying $\supp(\beta)\neq\gtm_2$, $f\in\supp(\beta)$ and $F\geq_\beta0$, we deduce $f+\sum_{i=1}^pf_i^2\in\psd^+(\{F\geq0\})$.

We claim: $h:=(f+\sum_{i=1}^pf_i^2)^2-Fg^2\in\psd^+(\kappa[[\x,\y]])$.

Let $\beta\in\Sper(\kappa[[\x,\y]])$ be such that $\supp(\beta)\neq\gtm_2$. We distinguish two cases:

\noindent{\sc Case 1.} If $F\geq_\beta0$, then $f\geq_\beta0$. As $f^2-Fg^2\in\psd(\kappa[[\x,\y]])$, we obtain $f^2-Fg^2\geq_\beta0$. Consequently,
$$
\Big(f+\sum_{i=1}^pf_i^2\Big)^2-Fg^2=(f^2-Fg^2)+2f\Big(\sum_{i=1}^pf_i^2\Big)+\Big(\sum_{i=1}^pf_i^2\Big)^2\geq_\beta0.
$$
If $(f+\sum_{i=1}^pf_i^2)^2-Fg^2\in\supp(\beta)$, then $f^2-Fg^2\in\supp(\beta)$ and $\sum_{i=1}^pf_i^2\in\supp(\beta)$, which contradicts the hypothesis. Thus, $(f+\sum_{i=1}^pf_i^2)^2-Fg^2>_\beta0$.

\noindent{\sc Case 2.} If $F<_\beta0$, then $(f+\sum_{i=1}^pf_i^2)^2-Fg^2\geq_\beta0$ and $(f+\sum_{i=1}^pf_i^2\Big)^2-Fg^2\in\supp(\beta)$ if and only if $f+\sum_{i=1}^pf_i^2,g\in\supp(\beta)$. As $\supp(\beta)$ is a real prime ideal, 
$$
\gtm_2=\sqrt[r]{\Big(f+\sum_{i=1}^pf_i^2,g\Big)}\subset\supp(\beta)\neq\gtm_2,
$$
which is a contradiction. Thus, $(f+\sum_{i=1}^pf_i^2)^2-Fg^2>_\beta0$.

Consequently, $(f+\sum_{i=1}^pf_i^2)^2-Fg^2\in\psd^+(\kappa[[\x,\y]])$ and $f+\sum_{i=1}^pf_i^2+\z g\in\psd^{\oplus}(A)$.

(iii) By (ii) it is enough to prove that: 
\begin{itemize}
\item[(1)] There exists $S\subset\N$ finite such that $\sqrt[r]{(f+M^2\y^{2n},g)}=\gtm_2$ for each $n\in\N\setminus S$. 
\item[(2)] $\y^{2n}>_\beta0$ for those $\beta\in\Sper(\kappa[[\x,\y]])$ satisfying: $\supp(\beta)\neq\gtm_2$, $F\geq_\beta0$ and either $f\in\supp(\beta)$ or $f^2-Fg^2\in\supp(\beta)$.
\end{itemize}

As $g\neq0$, if $\gta:=\sqrt[r]{(f+M^2\y^{2n},g)}\subsetneq\gtm_2$, the associated real prime ideals $\gtp_i$ of $\gta$ are prime ideals of height $1$. As $f(\x,0)\neq0$, then $\gtp_i\neq(\y)$. As $\gtp_1$ is a real prime ideal, there exists a prime cone $\beta\in\Sper(\kappa[[\x,\y]])$ such that $\supp(\beta)=\gtp_1$. By \cite[Prop.II.2.4]{abr} there exists $\alpha\in\Sper(\kappa)$ such that $\beta\to\alpha$. Let $R(\alpha)$ be the real closure of $(\kappa,\leq_\alpha)$ and $\ol{\kappa}:=R(\alpha)[\sqrt{-1}]$ the algebraic closure of $\kappa$. By Lemma \ref{dimen1} there exists a homomorphism $\phi:\kappa[[\x,\y]]\to R(\alpha)[[\y]]$ such that $\phi(\y)=\y^q$ for some $q\geq1$ and $h\geq_\beta0$ if and only if $\phi(h)\geq0$ (for $\x>0$). As $\gta\subset\supp(\beta)$, we have $\phi(f+M^2\y^{2n})=0$ and $\phi(g)=0$. Define $\zeta:=\phi(\y)\in R(\alpha)[[\y]]$. Observe that $\phi(f+M^2\y^{2n})=(f+M^2\y^{2n})(\zeta,\y^q)=0$ and $\phi(g)=g(\zeta,\y^q)=0$, that is, $\zeta(\y^{1/q})\in\ol{\kappa}((\y^*))$ is a common root of $f+M^2\y^{2n}$ and $g$ in $\ol{\kappa}((\y^*))$. As $g\neq0$, the series $g$ has finitely many roots in $\ol{\kappa}((\x^*))$ and we denote $S:=\{\omega(f(\eta,\y))/2:\ \eta\in\ol{\kappa}((\y^*)),\ g(\eta,\y)=0\}$. 

If $n\in\N\setminus S$, the series $f+M^2\y^{2n}$ and $g$ have no common root in $\ol{\kappa}((\y^*))$, so $\gta=\gtm_2$ (that is, (1) holds).

Let $\beta\in\Sper(\kappa[[\x,\y]])$ be a prime cone such that $\supp(\beta)\neq\gtm_2$, $F\geq_\beta0$ and either $f\in\supp(\beta)$ or $f^2-Fg^2\in\supp(\beta)$. If $\y^{2n}\in\supp(\beta)$, then $\y\in\supp(\beta)\neq\gtm_2$, so $\supp(\beta)=(\y)$. As $f(\x,0)\neq0$ and $f^2(\x,0)-F(\x,0)g^2(\x,0)\neq0$, it holds $f\not\in\supp(\beta)$ and $f^2-Fg^2\not\in\supp(\beta)$, which is a contradiction. Consequently, $\y^{2n}>_\beta0$ (that is, (2) holds).

(iv) As $h_n>_\beta0$ for those $\beta\in\Sper(\kappa[[\x,\y]])$ with $\supp(\beta)\neq\gtm_2$, by (ii) it is enough to prove: 
\begin{itemize}
\item[(1)] There exists $S\subset\N$ finite such that $\sqrt[r]{(f+h_n,g)}=\gtm_2$ for each $n\in\N\setminus S$.
\end{itemize}

The proof of (1) is analogous to the one of (iii.1) and we leave the details to the reader. 

(v) If $f\in\psd(\{F\geq0\})$, then $f+\z f\in\psd(A)$ because $(1+\z)$ is a positive unit of $A$.
\end{proof}

\subsection{Polynomial density}
As an application of Lemma \ref{orderings} we prove the following result that concerns `polynomial approximation' of certain distinguished elements of $\psd^{\oplus}(A)$.

\begin{lem}[Local polynomial density]\label{density0}
Suppose that $\psd(\{F\geq0\})\neq\varnothing$. We have:
\begin{itemize}
\item[(i)] Let $f+\z g\in A$ be such that $f(\x,0)\neq0$, $f^2(\x,0)-F(\x,0)g^2(\x,0)\neq0$ and $g\neq0$. Assume $f\in\psd(\{F\geq0\})$ and let $n\geq1$ be such that $(f+M^2\y^{2n})^2-Fg^2\in\psd^+(\kappa[[\x,\y]])$. There exists $r\geq1$ such that if $f_1,g_1\in\kappa[[\x,\y]]$ satisfy $f-f_1,g-g_1\in\gtm_2^r$, then $f_1+M^2\y^{2n}+\z g_1\in\psd^{\oplus}(A)$. 
\item[(ii)] Let $f+\z g\in A$ be such that $g\neq0$. Assume $f\in\psd(\{F\geq0\})$ and let $n\geq1$ be such that $(f+M^2(\x^2+\y^2)^n)^2-Fg^2\in\psd^+(\kappa[[\x,\y]])$. There exists $r\geq1$ such that if $f_1,g_1\in\kappa[[\x,\y]]$ satisfy $f-f_1,g-g_1\in\gtm_2^r$, then $f_1+M^2(\x^2+\y^2)^n+\z g_1\in\psd^{\oplus}(A)$.
\end{itemize}
\end{lem}
\begin{proof}
We prove both statements simultaneously. Let $\beta\in\Sper(\kappa[[\x,\y]])$ be such that $F\geq_\beta0$. By \cite[Prop.II.2.4]{abr} there exists $\alpha\in\Sper(\kappa)$ such that $\beta\to\alpha$. Let $R(\alpha)$ be the real closure of $(\kappa,\leq_\alpha)$. If $\supp(\beta)=\gtm_2$, then $\alpha=\beta$ and for each $r\geq1$ we have 
\begin{align*}
f+M^2\y^{2n}+\supp(\beta)&=f_1+M^2\y^{2n}+\supp(\beta),\\
(f+M^2\y^{2n})^2-Fg^2+\supp(\beta)&=(f_1+M^2\y^{2n})^2-Fg^2+\supp(\beta),\\
f+M^2(\x^2+\y^2)^n+\supp(\beta)&=f_1+M^2\y^{2n}+\supp(\beta),\\
(f+M^2(\x^2+\y^2)^n)^2-Fg^2+\supp(\beta)&=(f_1+M^2(\x^2+\y^2)^n)^2-Fg^2+\supp(\beta)
\end{align*}
if $f_1,g_1\in\kappa[[\x,\y]]$ satisfy $f-f_1,g-g_1\in\gtm_2^r$.

Assume $\supp(\beta)\neq\gtm_2$. If $\supp(\beta)=0$, then by \cite[Prop.VII.5.1]{abr} there exists $\beta_1\in\Sper(\kappa[[\x,\y]])$ such that $\beta\to\beta_1\to\alpha$ and $\supp(\beta_1)$ is a real prime ideal of height $1$, so we may assume ${\rm ht}(\supp(\beta))=1$. Observe that if $r\geq1$ and $f-f_1,g-g_1\in\gtm_2^r$, then
\begin{multline*}
(f+M^2\y^{2n})^2-Fg^2-((f_1+M^2\y^{2n})^2-Fg_1^2)\\
=(f-f_1)(f+f_1)+2M^2\y^{2n}(f-f_1)-F(g-g_1)(g+g_1)\in\gtm_2^r.
\end{multline*}
\begin{multline*}
(f+M^2(\x^2+\y^2)^n)^2-Fg^2-((f_1+M^2(\x^2+\y^2)^n)^2-Fg_1^2)\\
=(f-f_1)(f+f_1)+2M^2(\x^2+\y^2)^n(f-f_1)-F(g-g_1)(g+g_1)\in\gtm_2^r.
\end{multline*}
By Lemmas \ref{orderings} and \ref{psdkxy} we find $r\geq1$ such that if $f-f_1,g-g_1 \in\gtm_2^r$, then the statements hold, as required.
\end{proof}

In the proof of Theorem \ref{order2} (which corresponds to a part of Theorem \ref{list2}) we will make use of quadratic transformations. To take advantage of them we need a global `polynomial density' result (Lemma \ref{density}) and a `polynomial reduction' result (Corollary \ref{polred}).

\begin{lem}[Global polynomial density]\label{density}
Let $F\in\kappa[\x,\y]$ be such that $F(0,0)=0$, let $\gtm:=(\x,\y,\z)\kappa[\x,\y,\z]$ and $B:=\kappa[\x,\y,\z]/(\z^2-F)$. Let $\widehat{B}:=\kappa[[\x,\y,\z]]/(\z^2-F)$ be its $\gtm$-adic completion and $f+\z g\in\psd(\widehat{B})$. For each $n\geq1$ there exist polynomials $f_n,g_n\in\kappa[\x,\y]$ such that $f_n+\z g_n\in\psd(B)$ and $\omega(f-f_n),\omega(g-g_n)\geq n$. 
\end{lem}
\begin{proof}
By Lemmas \ref{pd+} and \ref{density0} there exist polynomials $f_n',g_n'\in\kappa[\x,\y]$ such that $f_n'+\z g_n'\in\psd^\oplus(\widehat{B})$ and $\omega(f-f_n'),\omega(g-g_n')\geq n$. Let us construct from $f_n',g_n'$ the desired polynomials $f_n,g_n\in\kappa[\x,\y]$ in the statement. Define ${\mathcal U}:=\{\beta\in\Sper(B):\ f_n'+\z g_n'<_\beta0\}$. If ${\mathcal U}=\varnothing$, we choose $f_n:=f_n'$ and $g_n:=g_n'$. Assume ${\mathcal U}\neq\varnothing$ and pick $\beta\in{\mathcal U}$. We claim: \em There exists $\veps_\beta\in\kappa\setminus\{0\}$ such that $\veps_\beta^4<_\beta\x^2+\y^2+\z^2$\em.

Suppose that $\x^2+\y^2+\z^2<_\beta\veps^4$ for each $\veps\in\kappa\setminus\{0\}$. Then $\x^2,\y^2,\z^2<_\beta\veps^4$ for each $\veps\in\kappa\setminus\{0\}$. Consider the \em absolute value \em associated to $\beta$:
$$
|\,.\,|_{\beta}:B\to B,\ a\mapsto\begin{cases}
a&\text{if $a\geq_\beta0$,}\\
-a&\text{if $a<_\beta0$.}
\end{cases}
$$
We have $|\x|_\beta<_\beta\veps^2$, $|\y|_\beta<_\beta\veps^2$ and $|\z|_\beta<_\beta\veps^2$ for each $\veps\in\kappa\setminus\{0\}$. Thus, if $P\in\kappa[[\x,\y,\z]]$ and $P(0,0,0)\neq0$, then $|P-P(0,0,0)|_\beta<_\beta|P(0,0,0)|_\beta$, so $P\geq_\beta0$ if and only if $P(0,0,0)>_\beta0$. Let $P_1,P_2\geq_\beta0$ be such that $P_1+P_2\in\gtm$. If $P_1(0,0,0)\neq0$, then $P_2(0,0,0)\neq0$ and $P_1(0,0,0)=-P_2(0,0,0)$. As $P_i\geq_\beta0$ and $P_i(0,0,0)\neq0$, we have $P_i(0,0,0)>_\beta0$, which is a contradiction because $P_1(0,0,0)=-P_2(0,0,0)$. Thus, $P_i(0,0,0)=0$ and $P_i\in\gtm$ for $i=1,2$. By \cite[Prop.4.3.8]{bcr} there exists a specialization $\beta\to\alpha$ such that $\supp(\alpha)=\gtm$. By \cite[Thm.VII.3.2]{abr} there exists a specialization $\widehat{\beta}\to\widehat{\alpha}$ of $\widehat{B}$ lying over $\beta\to\alpha$, that is, $\widehat{\alpha}\cap B=\alpha$ and $\widehat{\beta}\cap B=\beta$. As $f_n'+\z g_n'\in\psd^\oplus(\widehat{B})$, we deduce $f_n'+\z g_n'>_\beta0$, which is a contradiction. Consequently, the claim holds.

As $\veps_\beta^4<_\beta\x^2+\y^2+\z^2$, we obtain $1<_\beta\frac{\x^2+\y^2+\z^2}{\veps_\beta^4}$. If $\x^\nu\y^\mu\z^\rho$ is a monomial and $\nu+\mu+\rho\leq d$, 
$$
\x^{2\nu}\y^{2\mu}\z^{2\rho}\leq_\beta(\x^2+\y^2+\z^2)^{\nu+\mu+\rho}<_\beta\frac{(\x^2+\y^2+\z^2)^{2d}}{\veps_\beta^{4(2d-\nu-\mu-\rho)}}\quad\leadsto\quad|\x^\nu\y^\mu\z^\rho|_\beta<_\beta\frac{(\x^2+\y^2+\z^2)^d}{\veps_\beta^{2(2d-\nu-\mu-\rho)}}.
$$
In addition, $|a|_\beta<_\beta a^2+1$ for each $a\in\kappa$. Thus, if $P\in\kappa[\x,\y,z]$ is a polynomial of degree $\leq d$, there exists $N_\beta\in\Sos{\kappa}$ such that $|P|_\beta<_\beta N_\beta^2(\x^2+\y^2+\z^2)^d$. In particular, for $P:=f_n'+\z g_n'$ we find $M_\beta\in\Sos{\kappa}$ such that $|f_n'+\z g_n'|_\beta\leq M_\beta^2(\x^2+\y^2+\z^2)^d$ where $d:=\max\{\deg(f_n'+\z g_n'),n/2\}$. We deduce
$$
f_n'+\z g_n'+M_\beta^2(\x^2+\y^2+\z^2)^d>_\beta0
$$
for each $\beta\in{\mathcal U}$. Define 
$$
{\mathcal V}_\beta:=\{\gamma\in\Sper(B):\ f_n'+\z g_n'+M_\beta^2(\x^2+\y^2+\z^2)^d>_\gamma0\},
$$
which is an open subset of $\Sper(B)$ that contains $\beta$. As ${\mathcal U}$ is by \cite[Cor.7.1.13]{bcr} compact, there exist $\beta_1,\ldots,\beta_s\in{\mathcal U}$ such that ${\mathcal U}\subset\bigcup_{i=1}^s{\mathcal V}_{\beta_i}$. Define 
\begin{align*}
f_n&:=f_n'+\sum_{i=1}^sM_{\beta_i}^2(\x^2+\y^2+F^2)^d,\\
g_n&:=g_n'.
\end{align*}
As $\z^2-F=0$ in $B$, we deduce $f_n+\z g_n\in\psd(B)$, $\omega(f-f_n)\geq\min\{2d,n\}=n$ and $\omega(g-g_n)=\omega(g-g_n')=n$, as required.
\end{proof}

\subsubsection{Strong Artin's approximation and bounded Pythagoras numbers.}
To prove the property $\psd(A)=\Sos{A}$ for the rings $A=\kappa[[\x,\y,\z]]/(\z^2-F)$ in the list of Theorem \ref{list2} we need also that the Pythagoras numbers of the rings $A$ are finite. This is due to the use of Strong Artin's approximation we make: To represent $f+\z g\in\psd(A)$ as a sum of squares in $A$ we find for each $n\geq1$ elements $f_n+\z g_n\in\Sos{A}$ such that $\omega(f-f_n),\omega(g-g_n)\geq n$. Thus, there exist $a_{ni},b_{ni},q_n\in\kappa[[\x,\y]]$ for $i=1,\ldots,p_n$ such that
$$
f_n+\z g_n=\sum_{i=1}^{p_n}(a_{ni}+\z b_{ni})^2+(\z^2-F)q_n.
$$
Consequently, the polynomial equation
$$
f+\z g=({\tt X}_1+\z{\tt Y}_1)^2+\cdots+({\tt X}_{p_n}+\z{\tt Y}_{p_n})^2+(\z^2-F){\tt Z}
$$
has a solution $\!\!\!\mod\gtm_2^n$ in $\kappa[[\x,\y]]$ for each $n\geq1$. To apply Strong Artin's approximation we need that the sequence of positive integers $\{p_n\}_{n\geq1}$ is bounded by some positive integer $p<+\infty$. Indeed, this fact implies that the polynomial equation
$$
f+\z g=({\tt X}_1+\z{\tt Y}_1)^2+\cdots+({\tt X}_p+\z{\tt Y}_p)^2+(\z^2-F){\tt Z}
$$
has a solution $\!\!\!\mod\gtm_2^n$ in $\kappa[[\x,\y]]$ for each $n\geq1$. Now, by Strong Artin's approximation there exist $a_i,b_i,q\in\kappa[[\x,\y]]$ such that
$$
f+\z g=(a_1+\z b_1)^2+\cdots+(a_p+\z b_p)^2+(\z^2-F)q,
$$ 
so $f+\z g\in\Sosp{A}\subset\Sos{A}$. 

To guarantee that the sequence of positive integers $\{p_m\}_{m\geq1}$ associated to each $f+\z g\in\psd(A)$ is bounded, it seems natural to ask that the Pythagoras number $p(A)$ of the ring $A$ is bounded. By Theorem \ref{pyth} and Equation \eqref{p2} we have $p(A)\leq 4\tau(\kappa)$ and thefore the hypothesis $\tau(\kappa)<+\infty$ appears in the statements of Theorem \ref{list2} and Corollary \ref{unique}. 

We are ready to present the polynomial reduction.

\begin{cor}[Polynomial reduction]\label{polred} 
Let $F\in\kappa[\x,\y]$ be such that $F(0,0)=0$, let $\gtm:=(\x,\y,\z)\kappa[\x,\y,\z]$ and $B:=\kappa[\x,\y,\z]/(\z^2-F)$. Let $\widehat{B}:=\kappa[[\x,\y,\z]]/(\z^2-F)$ be its $\gtm$-adic completion and ${\tt i}:B\hookrightarrow\widehat{B}$ the canonical map. If there exists an integer $p\geq 1$ such that ${\tt i}(\psd(B))\subset\Sosp{\widehat{B}}$, then $\psd(\widehat{B})=\Sosp{\widehat{B}}$. 
\end{cor}
\begin{proof}
Let $f+\z g\in\psd(\widehat{B})$ and fix $n\geq1$. By Lemma \ref{density} there exist polynomials $f_n,g_n\in\kappa[\x,\y]$ such that $f_n+\z g_n\in\psd(B)$ and $\omega(f-f_n),\omega(g-g_n)\geq n$. By hypothesis ${\tt i}(\psd(B))\subset\Sosp{\widehat{B}}$ for some $p\geq1$. There exist $a_{ni},b_{ni},q_n\in\kappa[[\x,\y]]$ such that
$$
f_n+\z g_n=\sum_{i=1}^p(a_{ni}+\z b_{ni})^2+(\z^2-F)q_n.
$$ 
Consequently,
$$
f+\z g\equiv\sum_{i=1}^p(a_{ni}+\z b_{ni})^2+(\z^2-F)q_n\mod\gtm_2^n.
$$
Thus, the polynomial equation
$$
f+\z g=\sum_{i=1}^p({\tt X}_i+\z {\tt Y}_i)^2+(\z^2-F){\tt Z}
$$
has a solution $\!\!\!\mod\gtm_2^n$ for each $n\geq1$. By Strong Artin's approximation we conclude $f+\z g\in\Sosp{\widehat{B}}$, as required.
\end{proof}

\section{Proof of Theorem \ref{list2}}\label{s5}

In this section we prove Theorem \ref{list2}. We may assume: $A=\kappa[[\x,\y,\z]]/\gta$ is a formal ring and by Theorem \ref{nonprincipal} that $\gta$ is a principal ideal generated by a series of the type $\z^2-F$ where $F\in\kappa[[\x,\y]]$. The strategy to prove \cite[Thm.1.3]{f2} consists of starting with two special cases: \em the rings $\R\{\x,\y\}$ and $\R\{\x,\y,\z\}/(\z^2-\x^3-\y^5)$\em. Both rings have the following common property: \em They are factorial rings and their complexifications $\C\{\x,\y\}$ and $\C\{\x,\y,\z\}/(\z^2-\x^3-\y^5)$ are also factorial rings \em \cite{ch,sj}. Using these two facts and that $\tau(\R((\t)))=1$, one shows that every positive semidefinite element of these rings is a sum of two squares of elements of such rings, see \cite[Thm.2.5 (3)]{clrr} and \cite{rz2}. In our setting it would be natural to choose as special cases: $\kappa[[\x,\y]]$ and $\kappa[[\x,\y,\z]]/(\z^2-\x^3-\y^5)$ where $\kappa$ is a (formally) real field with $\tau(\kappa)<+\infty$. Both rings are factorial and if we tensor them by $-\otimes_\kappa\kappa[\sqrt{-1}]$, these new rings are again factorial \cite{ch,sj}. The problem is that the property $\tau(\kappa((\t)))=1$ does not hold anymore! (if $\tau(\kappa)>1$) and we cannot imitate the same procedure. This means that different ideas were needed (Theorem \ref{et}) and the problem stayed apart from any significant progress for quite long. The case $\kappa[[\x,\y]]$ where $\kappa$ is any field was solved by Scheiderer in \cite[Thm.4.1]{sch2} (in fact, he proved $\psd(A)=\Sos{A}$ for all 2-dimensional regular local rings). We include for the sake of completeness here a slightly different proof of \cite[Thm.4.1]{sch2} when $\kappa$ is a (formally) real field.

\begin{thm}\label{liso}
Let $\kappa$ be a (formally) field. Then $\psd(\kappa[[\x]][\y])=\Sos{\kappa[[\x]][\y]}$ and $\psd(\kappa[[\x,\y]])=\Sos{\kappa[[\x,\y]]}$.
\end{thm}
\begin{proof}
Let $f$ belong to either $\psd(\kappa[[\x]][\y])$ or $\psd(\kappa[[\x,\y]])$. By Lemma \ref{density} and Corollary \ref{polred} we may assume
$$
f\in\psd(\kappa[\x,\y])\subset\psd(\kappa((\x))[\y])\subset\psd(\kappa((\x))(\y))=\Sos{\kappa((\x))(\y)}.
$$
By \cite[Thm.1]{ca} $f\in\Sos{\kappa((\x))[\y]}$. Thus, there exist $a_1,\ldots,a_p\in\kappa((\x))[\y]$ such that $f=\sum_{k=1}^pa_k^2$. Let $\ell\geq1$ be such that $b_k=\x^\ell a_k\in\kappa[[\x]][\y]$ for each $k=1,\ldots,p$. We have
$$
\x^{2\ell}f=\sum_{k=1}^pb_k^2(\x,\y).
$$
Setting $\x=0$, we deduce $0=\sum_{k=1}^pb_k(0,\y)^2$. As $\kappa$ is a (formally) real field, each $b_k(0,\y)=0$, so $\x$ divides each $b_k$. Thus, there exists $b_k'\in\kappa[[\x]][\y]$ such that $b_k=\x b_k'$. Hence,
$$
\x^{2\ell-2}f=\sum_{k=1}^pb_k'^2(\x,\y).
$$
Proceeding inductively we conclude $f\in\Sos{\kappa[[\x]][\y]}\subset\Sos{\kappa[[\x,\y]]}$, as required.
\end{proof}

\subsection{Order two cases}
Let $\kappa$ be a (formally) real field such that $\tau(\kappa)<+\infty$. To prove the property $\psd(A)=\Sos{A}$ for those rings $A=\kappa[[\x,\y,\z]]/(\z^2-F)$ where $F\in\kappa[[\x,\y]]$ is one of the series in the following list: 
\begin{itemize}
\item[(i)] $a\x^2+b\y^{2k}$ where $a\in\Sos{\kappa}$, $a,b\neq0$ and $k\geq 1$,
\item[(ii)] $a\x^2+\y^{2k+1}$ where $a\in\Sos{\kappa}$, $a\neq0$ and $k\geq 1$,
\item[(iii)] $a\x^2$ where $a\in\Sos{\kappa}$ and $a\neq0$,
\end{itemize}
we develop a strategy similar to the one presented in \cite{f2} to prove its main result. Namely, we relate the problem of proving $\psd(A)=\Sos{A}$ for each ring $A$ above to the already known property (Theorem \ref{liso}) for the ring $\kappa[[\x,\y]]$ via suitable blow-ups. During the process certain `denominators' appear that we have to `erase'. To that end we develop the following procedure.

\begin{lem}[Erasure of denominators]\label{grd}
Let $\kappa$ be a (formally) real field, $k\geq 1$ a positive integer and $h,g\in\kappa[[\x,\y]]$ relatively prime series such that $h$ generates a real radical ideal. Suppose that there exist polynomials $f,a_1,\ldots,a_p,b\in\kappa[[\x,\y]][\z]$ of degree $\leq k-1$ (with respect to $\z$) and an integer $r\geq 0$ such that
\begin{equation}\label{denome}
h^{2r}f=a_1^2+\cdots+a_p^2-b(\z^k-hg).
\end{equation}
Then there exist polynomials $a_1',\ldots,a_p',b'\in\kappa[[\x,\y]][\z]$ of degree $\leq k-1$ (with respect to $\z$) such that $f=a_1'^2+\cdots+a_p'^2-b'(\z^k-hg)$.
\end{lem}
\begin{proof}
Set $a_i:=\sum_{j=0}^{k-1}a_{ij}\z^j$, $b:=\sum_{j=0}^{k-1}b_{j}\z^j$ and $f:=\sum_{j=0}^{k-1}f_{j}\z^j$ where $a_{ij},b_j,f_j\in\kappa[[\x,\y]]$. Comparing coefficients with respect to $\z$ in \eqref{denome} we find the following collection of equalities
$$
\begin{array}{rrcl} 
(0)& h^{2r}f_0-hgb_0&=&\sum_{i=1}^pa_{i0}^2,\\ 
(1)& h^{2r}f_1-hgb_1&=&2\sum_{i=1}^pa_{i0}a_{i1},\\ 
&&\vdots& \\
(\ell)&h^{2r}f_\ell-hgb_\ell&=&\sum_{i=1}^p\Big(\sum_{j+m=\ell}a_{ij}a_{im}\Big),\\
&&\vdots& \\ 
(k-1)&h^{2r}f_{k-1}-hgb_{k-1}&=&\sum_{i=1}^p\Big(\sum_{j+m=k-1}a_{ij}a_{im}\Big),\\
(k)&b_0&=&\sum_{i=1}^p\Big(\sum_{j+m=k}a_{ij}a_{im}\Big),\\ 
&&\vdots& \\
(k+\ell)& b_\ell&=&\sum_{i=1}^p\Big(\sum_{j+m=k+\ell}a_{ij}a_{im}\Big),\\
&&\vdots& \\ 
(2k-2)& b_{k-2}&=&\sum_{i=1}^pa_{ik-1}^2,\\
(2k-1)& b_{k-1}&=&0.
\end{array}
$$ 
We claim: \em $h$ divides $a_{i\ell}$ and $b_\ell$ for each $\ell=0,\ldots,k-1$ and $i=1,\ldots,p$\em. 

We use induction hypothesis to prove the claim. For $\ell=0$ we have $h^{2r}f_0-hgb_0=\sum_{i=1}^pa_{i0}^2$. As $h$ generates a real radical ideal and $h,g$ are relatively prime, we deduce that $h$ divides each $a_{i0}$ and $b_0$. Let $\ell<k$ and assume $h$ divides $a_{ij},b_j$ for $i=1,\ldots,p$ and $1\leq j\leq\ell-1$. If $2\ell\leq k-1$,
$$
h^{2r}f_{2\ell}-hgb_{2\ell}=\sum_{i=1}^p\Big(\sum_{j+m=2\ell,j\neq
m}a_{ij}a_{im}\Big)+\sum_{i=1}^pa_{i\ell}^2.
$$ 
As $h$ divides $a_{i0},\ldots,a_{i,\ell-1}$, we deduce $h$ divides $a_{i\ell}$ for each $i$. If $2\ell>k-1$, then:
$$
b_{2\ell-k}=\sum_{i=1}^p\Big(\sum_{j+m=2\ell,j\neq m}a_{ij}a_{im}\Big)+\sum_{i=1}^pa_{i\ell}^2.
$$ 
As $\ell\leq k-1$, we have $2\ell-k\leq\ell-1$ and by induction hypothesis $h$ divides $b_{2\ell-k}$. By induction hypothesis also $h$ divides $a_{i0},\ldots,a_{i,\ell-1}$, so $h$ divides $a_{i\ell}$ for each $i$. Using equation $(\ell)$ we conclude that $h$ divides $b_\ell$, as claimed.

Now we use equations $(k)$, $\ldots$, $(2k-2)$ to prove that in fact $h^2$ divides $b_\ell$ for $0\leq\ell\leq k-2$. Thus, we can divide the equation \eqref{denome} by $h^2$ (use here that $\kappa[[\x,\y]][\z]$ is an integral domain). Repeating the previous process $r-1$ times one proves the required statement.
\end{proof}

Before proving Theorem \ref{list2} for the order two cases we need an additional result.

\begin{lem}\label{ext}
Let $\kappa$ be a (formally) real field, let $a\in\Sos{\kappa}\setminus\{0\}$ and $\varphi,\psi\in\kappa[[\x,\y,\z]]$ be such that $\psi$ is a sum of squares in $B:=\kappa[\sqrt{a}][[\x,\y,\z]]/(\varphi)$. Then $\psi$ is a sum of squares in $A:=\kappa[[\x,\y,\z]]/(\varphi)$.
\end{lem}
\begin{proof}
If $\sqrt{a}\in\kappa$, the result is trivial, so we assume $\sqrt{a}\not\in\kappa$. Consider the $\kappa$-involution
$$
\sigma:\kappa[\sqrt{a}]\to\kappa[\sqrt{a}],\ b+\sqrt{a}c\mapsto b-\sqrt{a}c
$$
and the induced $\kappa$-involution
\begin{multline*}
(\cdot)^\sigma:\kappa[\sqrt{a}][[\x,\y,\z]]/(\varphi)\to\kappa[\sqrt{a}][[\x,\y,\z]]/(\varphi),\\ 
f:=\sum_{\nu}a_\nu\x^{\nu_1}\y^{\nu_2}\z^{\nu_3}\mapsto f^\sigma:=\sum_{\nu}\sigma(a_\nu)\x^{\nu_1}\y^{\nu_2}\z^{\nu_3}
\end{multline*}
where $\nu:=(\nu_1,\nu_2,\nu_3)$. As $\psi\in\Sos{B}$, there exist $a_i,b_i,q_1,q_2\in\kappa[[\x,\y,\z]]$ such that 
$$
\psi=\sum_{i=1}^p(a_i+\sqrt{a}b_i)^2+(q_1+\sqrt{a}q_2)\varphi=\sum_{i=1}^pa_i^2+a\Big(\sum_{i=1}^pb_i^2\Big)+q_1\varphi+\sqrt{a}\Big(q_2\varphi+2\sum_{i=1}^pa_ib_i\Big).
$$
We apply the $\kappa$-involution $(.)^\sigma$ to the previous equality and obtain
$$
\psi=\sum_{i=1}^p(a_i-\sqrt{a}b_i)^2+(q_1-\sqrt{a}q_2)\varphi=\sum_{i=1}^pa_i^2+a\Big(\sum_{i=1}^pb_i^2\Big)+q_1\varphi-\sqrt{a}\Big(q_2\varphi+2\sum_{i=1}^pa_ib_i\Big).
$$
Adding both equalities and dividing by $2$ we achieve
$$
\psi=\sum_{i=1}^pa_i^2+a\Big(\sum_{i=1}^pb_i^2\Big)+q_1\varphi.
$$
As $a\in \Sos{\kappa}$, we deduce $\psi\in\psd(A)$, as required.
\end{proof}

We are now ready to prove Theorem \ref{list2} for the rings $A=\kappa[[\x,\y,\z]]/(\z^2-F)$ where $F\in\kappa[[\x,\y]]$ is one of the series of order two quoted in its statement.

\begin{thm}[Order two]\label{order2}
Let $\kappa$ be a (formally) real field such that $\tau(\kappa)<+\infty$. Set $A:=\kappa[[\x,\y,\z]]/(\z^2-F)$ where $F\in\kappa[[\x,\y]]$ is one of the following series:
\begin{itemize}
\item[(i)] $a\x^2+b\y^{2k}$ where $a\in\Sos{\kappa}$, $b\neq0$ and $k\geq 1$,
\item[(ii)] $a\x^2+\y^{2k+1}$ where $a\in\Sos{\kappa}$ and $k\geq 1$,
\item[(iii)] $a\x^2$ where $a\in\Sos{\kappa}$. 
\end{itemize}
Then $\psd(A)=\Sos{A}$ and $p(A)\leq4\tau(\kappa)$.
\end{thm}
\begin{proof}
We denote $B:=\kappa[\x,\y,\z]/(\z^2-F)$ in such a way that $\widehat{B}=A$. As $\tau(\kappa)<+\infty$, it holds $\psd(\kappa[[\x,\y]])=\Sosq{\kappa[[\x,\y]]}$ and $\psd(\kappa[[\x]][\y])=\Sosq{\kappa[[\x]][\y]}$ for $p_0:=2\tau(\kappa)$. Let ${\tt i}:B\to\widehat{B}=A$. Along the proof it is enough to show that ${\tt i}(\psd(B))\subset\Sosp{A}$ for $p:=4\tau(\kappa)$.

(i) and (ii) Assume first $a=1$ and $F:=\x^2+b\y^k$. After the linear change of coordinates $(\x,\y,\z)\mapsto(\frac{\z-\x}{2},\y,\frac{\z+\x}{2})$ we assume that our equation is $b\y^k-\x\z$. Consider the rational substitution $(\y,\z)\mapsto(\frac{b\y^k}{\z},\y,\z)$. Take $P\in\psd(B)$ (of degree $\leq k-1$ with respect to $\y$) and write
$$
P\Big(\frac{\y^k}{\z},\y,\z\Big)=\frac{Q}{\z^r}
$$ 
where $r\geq 1$ and $Q\in\kappa[\y,\z]$. By Lemma \ref{blowup} the product $\z^rQ\in\psd(\kappa[\y,\z])\subset\psd(\kappa[[\y,\z]])=\Sosp{\kappa[[\x,\y]]}$. Thus, there exist $a_1,\ldots,a_p\in\kappa[[\y,\z]]$ such that
$$
\z^{2r}P\Big(\frac{\y^k}{\z},\y,\z\Big)=\z^rQ=a_1^2+\cdots+a_p^2.
$$
We rewrite the previous equation as
$$
\z^{2r}P(\x,\y,\z)=a_1^2(\y,\z)+\cdots+a_p^2(\y,\z)+(b\y^k-\x\z)q(\x,\y,\z),
$$
where $q\in\kappa[[\x,\y,\z]]$. By Weierstrass division theorem there exist $a_1',\ldots,a_p',q'\in\kappa[[\x,\z]][\y]$ of degree $\leq k-1$ with respect to $\y$ such that
$$ 
\z^{2r}P(\x,\y,\z)=a_1'^2+\cdots+a_p'^2+(b\y^k-\x\z)q'.
$$ 
By Lemma \ref{grd} $P$ is a sum of $p$ squares in $A$.

In the following we approach the case when $a\in\Sos{\kappa}\setminus\{0\}$. Consider the inclusion 
$$
{\tt i}:A:=\kappa[[\x,\y,\z]]/(\z^2-F)\hookrightarrow A':=\kappa[\sqrt{a}][[\x,\y,\z]]/(\z^2-F)
$$
and the isomorphism of rings
$$
\Phi:A'\to B':=\kappa[\sqrt{a}][[\x,\y,\z]]/(\z^2-F'),\ h\mapsto h\Big(\frac{\x}{\sqrt{a}},\y,\z\Big),
$$
where 
$$
F'=\begin{cases}
\x^2+b\y^{2k}&\text{if $F=a\x^2+b\y^{2k}$,}\\
\x^2+\y^{2k+1}&\text{if $F=a\x^2+\y^{2k+1}$.}
\end{cases}
$$
Let $f,g\in\kappa[[\x,\y]]$ be such that $f+\z g\in\psd(A)$. Then $\Phi(f+\z g)\in\psd(B')=\Sos{B'}$ (as we have proved above, when we assumed $a=1$ and $F=\x^2+b\y^k$). Thus, $f+\z g\in\Sos{A'}$, so by Lemma \ref{ext} $f+\z g\in\Sos{A}$. In addition, by \cite[Prop.2.7]{frs2} $p(A)\leq4\tau(\kappa)$.

(iii) We use a `limit argument'. Let $f+\z g\in\psd(A)$ and $n\ge1$. As $a\in\Sos{\kappa}\setminus\{0\}$, 
\begin{align*}
&f\in\psd(\{a\x^2\geq0\})=\psd(\kappa[[\x,\y]]),\\
&f^2-a\x^2g^2\in\psd(\kappa[[\x,\y]]).
\end{align*}
If $k\geq 2n+1$, we have $2k-4n\geq 2$ and
\begin{multline*}
(f+\x^{2n}+\y^{2n})^2-(a\x^2+\y^{2k})g^2=(f^2-a\x^2g^2)+2f(\x^{2n}+\y^{2n})\\
+\x^{4n}+2x^{2n}\y^{2n}+\y^{4n}(\sqrt{1-\y^{2k-4n}g^2})^2\in\psd(\kappa[[\x,\y]]).
\end{multline*}
As $f+\x^{2n}+\y^{2n}\in\psd(\kappa[[\x,\y]])=\psd(\{a\x^2+\y^{2k}\geq0\})$, we deduce $f+\x^{2n}+\y^{2n}+\z g\in\psd(A_k)$ where $A_k:=\kappa[[\x,\y,\z]]/(\z^2-a\x^2-\y^{2k})$. As $\psd(A_k)=\Sosp{A_k}$ where $p:=4\tau(\kappa)$, there exist $a_{in},b_{in},q_n\in\kappa[[\x,\y]]$ such that
$$
f+(\x^{2n}+\y^{2n})+\z g=(a_{1n}+\z b_{1n})^2+\cdots+(a_{pn}+\z b_{pn})^2-(\z^2-a\x^2-\y^{2k})q_n.
$$ 
Consequently,
$$ 
f+\z g\equiv (a_{1n}+\z b_{1n})^2+\cdots+(a_{pn}+\z b_{pn})^2-(\z^2-a\x^2)q_n\ \mod\gtm_2^{2n}.
$$ 
By Strong Artin's approximation there exist $a_i,b_i,q\in\kappa[[\x,\y]]$ such that 
$$
f+\z g=(a_1+\z b_1)^2+\cdots+(a_p+\z b_p)^2-(z^2-ax^2)q\in\Sos{A},
$$
as required.
\end{proof}

\subsection{Order three cases}
Let $\kappa$ be a (formally) real field such that $\tau(\kappa)<+\infty$. To prove the property $\psd(A)=\Sos{A}$ for those rings $A=\kappa[[\x,\y,\z]]/(\z^2-F)$ where $F\in\kappa[[\x,\y]]$ is one of the series in the following list: 
\begin{itemize}
\item[(iv)] $\x^2\y+(-1)^ka\y^k$ where $a\not\in-\Sos{\kappa}$ and $k\geq 3$,
\item[(v)] $\x^2\y$,
\item[(vi)] $\x^3+a\x\y^2+b\y^3$ irreducible,
\item[(vii)] $\x^3+a\y^4$ where $a\not\in-\Sos{\kappa}$, 
\item[(viii)] $\x^3+\x\y^3$,
\item[(ix)] $\x^3+\y^5$, 
\end{itemize}
we develop a substantially different strategy to the one presented in \cite{f2}. As commented above if $\tau(\kappa)>1$, the classical procedure to face the qualitative problem for Brieskorn's singularity $A:=\kappa[[\x,\y,\z]]/(\z^2-\x^3-\y^5)$ does not work and new ideas are needed. In fact, we are going to provide a general tool (Theorem \ref{et}) that allows to solve all the cases of the list in Theorem \ref{list2} when $F\in\kappa[[\x,\y]]$ has order three. 

The formulation of this tool is quite technical (because it also points out the obstructions to obtain the property $\psd(A)=\Sos{A}$ for the rings $A$ satisfying the hypotheses in Theorem \ref{et}). We suggest the reader the following initial interpretation of Theorem \ref{et} for a better understanding: \em If $F\in\kappa[[\y]][\x]$ is a Weierstrass polynomial of order and degree $3$ that satisfies some mild conditions and all the positive elements of $f+\z g\in A:=\kappa[[\x,\y,\z]]/(\z^2-F)$ such that $f(\x,0)\neq0$ satisfies $\omega(f(\x,0))\geq2$, then $\psd(A)=\Sosp{A}$ for $p:=4\tau(\kappa)$\em. In addition, the rings $A$ (satisfying the hypotheses of Theorem \ref{et}) such that $\psd(A)\neq\Sos{A}$ arise when there exist positive semidefinite elements $f+\z g\in A$ such that $f(\x,0)\neq0$ and $\omega(f(\x,0))=1$.

\begin{thm}[Elephant's Theorem]\label{et}
Let $F\in\kappa[[\x,\y]]$ be a series of order $3$ with $\omega(F(\x,0))=3$ and $F(0,\y))=b\y^{\rho}v$ where $v\in\kappa[[\y]]$ is a unit such that $v(0)=1$, $b\in\kappa\setminus\{0\}$ and $\rho\geq3$ is either odd or if it is even, we have in addition $b\not\in-\Sos{\kappa}$. Let $A:=\kappa[[\x,\y,\z]]/(\z^2-F)$ and $f+\z g\in\psd(A)\setminus\{0\}$. We obtain:
\begin{itemize}
\item[(i)] There exist $s\geq1$ and $f_1+\z g_1\in\psd(A)$ such that $f_1(\x,0)\neq0$ and $f+\z g=\y^{2s}(f_1+\z g_1)$. In addition, if $\omega(f_1(\x,0))=q\geq 2$, then $\omega(g_1)\geq1$.
\item[(ii)] If $\omega(f(\x,0))=q\geq 2$ and $F$ is $\rho$-quasidetermined, then $f+\z g\in\Sosp{A}$ where $p:=4\tau(\kappa)$.
\end{itemize}
\end{thm}

Before proving this theorem, we develop some preliminary results.

\begin{lem}\label{quadratic}
Let $A$ be a ring of characteristic $0$ and $P:=a_0+a_1\z+a_2\z^2\in A[\z]$. Then $P\in\psd(A[\z])$ if and only if $a_0,a_2,4a_0a_2-a_1^2\in\psd(A)$. 
\end{lem}
\begin{proof}
Suppose first that $P\in\psd(A[\z])$. Let $\alpha\in\Sper(A)$ and let us check: \em $a_0\geq_\alpha0$, $a_2\geq_\alpha0$ and $4a_0a_2-a_1^2\geq_\alpha0$\em. 

Set $(K(\alpha):=\qf(A/\supp(\alpha)),<_\alpha)$ and denote $\theta_i:=a_i+\supp(\alpha)$ for $i=0,1,2$. Consider the homomorphism $\phi:A[\z]\to K(\alpha),\ Q(\z)\mapsto Q(0)+\supp(\alpha)$. We have $\phi(P)=\theta_0\geq_\alpha0$, so $a_0\geq_\alpha0$. We claim: \em If $\theta_2=0$, then $\theta_1=0$\em. If such is the case, then $4\theta_0\theta_2-\theta_1^2=0$.

Otherwise, consider the homomorphism
$$
\gamma:A[\z]\to K(\alpha),\ Q\mapsto Q\Big(\frac{-\theta_0-1}{\theta_1}\Big)
$$
and observe that $-1=\gamma(P)\geq_\alpha0$, which is a contradiction. 

Assume in the following $\theta_2\neq0$. Consider the homomorphism
$$
\varphi:A[\z]\to K(\alpha),\ Q(\z)\mapsto Q\Big(\frac{\theta_1^2}{\theta_2^2}+\frac{\theta_0^2}{\theta_2^2}+1\Big).
$$
We have
\begin{multline*}
\varphi(P)=\theta_2\Big(\Big(\frac{\theta_1^2}{\theta_2^2}+\frac{\theta_0^2}{\theta_2^2}+1\Big)^2+\frac{\theta_1}{\theta_2}\Big(\frac{\theta_1^2}{\theta_2^2}+\frac{\theta_0^2}{\theta_2^2}+1\Big)+\frac{\theta_0}{\theta_2}\Big)\\
=\theta_2\Big(\frac{\theta_1^2}{\theta_2^2}\Big(\frac{\theta_1}{\theta_2}+\frac{1}{2}\Big)^2+\Big(\frac{\theta_0^2}{\theta_2^2}+\frac{\theta_1}{2\theta_2}\Big)^2+\Big(\frac{\theta_1}{\theta_2}+\frac{1}{2}\Big)^2+\Big(\frac{\theta_0}{\theta_2}+\frac{1}{2}\Big)^2+\frac{2\theta_1^2\theta_0^2}{\theta_2^4}+\frac{\theta_1^2}{2\theta_2^2}+\frac{\theta_0^2}{\theta_2^2}+\frac{1}{2}\Big). 
\end{multline*}
As $\varphi(P)\geq_\alpha0$, we deduce $a_2\geq_\alpha0$. 

We prove next that $4a_0a_2-a_1^2\geq_\alpha0$. Write 
$$
\theta_2\z^2+\theta_1\z+\theta_0=\theta_2\Big(\Big(\z+\frac{\theta_1}{2\theta_2}\Big)^2+\frac{\theta_0}{\theta_2}-\frac{\theta_1^2}{4\theta_2^2}\Big).
$$
Consider the homomorphism
$$
\psi:A[\z]\to K(\alpha),\ Q\mapsto Q\Big(-\frac{\theta_1}{2\theta_2}\Big)
$$
and, as $\psi(P)=\theta_2((-\frac{\theta_1}{2\theta_2}+\frac{\theta_1}{2\theta_2})^2+\frac{\theta_0}{\theta_2}-\frac{\theta_1^2}{4\theta_2^2})=\frac{4\theta_0\theta_2-\theta_1^2}{4\theta_2}\geq_\alpha0$, we conclude $4a_0a_2-a_1^2\geq_\alpha0$. 

Assume in the following that $a_0,a_2,a_0a_2-a_1^2\in\psd(A)$ and let us check: $P\in\psd(A[\z])$. 

Let $\beta\in\Sper(A[\z])$ and consider the natural inclusion ${\tt j:}A\hookrightarrow A[\z]$. Then $\alpha:={\tt j}^{-1}(\beta)\in\Sper(A)$, and we consider the ordered field $(K(\alpha):=\qf(A/\supp(\alpha)),<_\alpha)$. Denote $\eta_i:=a_i+\supp(\alpha)$ for $i=0,1,2$. We distinguish two cases:

\noindent{\sc Case 1}. $\eta_2=0$. As $\eta_0\eta_2-\eta_1^2\geq_\alpha0$, we deduce $\eta_1=0$. Thus, $P+\supp(\beta)=\eta_0\geq_\alpha0$. Consequently, $P\geq_\beta0$.

\noindent{\sc Case 2}. $\eta_2\neq0$. As $\eta_2\geq_\alpha0$, to prove that $P\geq_\beta0$ it is enough to check $\eta_2P\geq_\beta0$. We have
$$
\eta_2P+\supp(\beta)=\eta_2(\eta_2\z^2+\eta_1\z+\eta_0)+\supp(\beta)=\Big(\eta_2\z+\frac{1}{2}\eta_1\Big)^2+\frac{1}{4}(4\eta_0\eta_2-\eta_1^2)\geq_\beta0.
$$
Thus, $P\geq_\beta0$.

We conclude $P\in\psd(A[\z])$, as required.
\end{proof}

The previous result has its counterpart for sums of squares, which we include for the sake of completeness.

\begin{lem}\label{sosq}
Let $A$ be an integral domain of characteristic $0$ and $P:=a_0+a_1\z+a_2\z^2\in A[\z]$. Then $P\in\Sosp{A[\z]}$ if and only if the following polynomial system has a solution in $A^{2p}$:
\begin{align*}
&\sum_{i=1}^p\x_i^2=a_0,\quad\sum_{j=1}^p\y_j^2=a_2,\\
&\sum_{1\leq i<j\leq p}(2\x_i\y_j-2\x_j\y_i)^2=4a_0a_2-a_1^2.
\end{align*}
\end{lem}
\begin{proof}
We will make use of Lagrange's identity in the proof:
\begin{multline*}
\Big(\sum_{i=1}^p\x_i^2\Big)\Big(\sum_{j=1}^p\y_j^2\Big)-\Big(\sum_{k=1}^p\x_k\y_k\Big)^2=\sum_{i=1}^p\sum_{j=1}^p\x_i^2\y_j^2-\sum_{i,j=1}^p\x_i\y_i\x_j\y_j\\
=\sum_{i,j=1,\ i\neq j}^n\x_i^2\y_j^2-2\sum_{1\leq i<j\leq p}\x_i\y_i\x_j\y_j=\sum_{1\leq i<j\leq p}(\x_i\y_j-\x_j\y_i)^2.
\end{multline*}
Suppose first $P\in\Sosp{A[\z]}$. Then there exist $\alpha_i,\beta_i\in A$ such that
\begin{equation}\label{quadr}
a_0+a_1\z+a_2\z^2=\sum_{i=1}^p(\alpha_i+\z\beta_i)^2.
\end{equation}
Thus,
$$
\sum_{i=1}^p\alpha_i^2=a_0,\quad\sum_{j=1}^p\beta_j^2=a_2,\quad2\sum_{k=1}^p\alpha_k\beta_k=a_1.
$$
We deduce
$$
4a_0a_2-a_1^2=4\Big(\sum_{i=1}^p\alpha_i^2\Big)\Big(\sum_{j=1}^p\beta_j^2\Big)-4\Big(\sum_{k=1}^p\alpha_k\beta_k\Big)^2=\sum_{1\leq i<j\leq p}(2\alpha_i\beta_j-2\alpha_j\beta_i)^2.
$$

Suppose next that the polynomial system
$$
\sum_{i=1}^p\x_i^2=a_0,\quad\sum_{j=1}^p\y_j^2=a_2,\quad
\sum_{1\leq i<j\leq p}(2\x_i\y_j-2\x_j\y_i)^2=4a_0a_2-a_1^2
$$
has a solution $(\alpha_1,\ldots,\alpha_p,\beta_1,\ldots,\beta_p)\in A^{2p}$. Observe that
$$
a_1^2=4\Big(\sum_{i=1}^p\alpha_i^2\Big)\Big(\sum_{j=1}^p\beta_j^2\Big)-4\sum_{1\leq i<j\leq p}(\alpha_i\beta_j-\alpha_j\beta_i)^2=\Big(2\sum_{k=1}^p\alpha_k\beta_k\Big)^2.
$$
Consequently, 
$$
\Big(a_1-2\sum_{k=1}^p\alpha_k\beta_k\Big)\Big(a_1+2\sum_{k=1}^p\alpha_k\beta_k\Big)=0
$$
and there exists $\veps=\pm1$ such that $a_1=2\veps\sum_{k=1}^p\alpha_k\beta_k$. Thus,
$$
a_0+a_1\z+a_2\z^2=\sum_{i=1}^p(\alpha_i+\z\veps\beta_i)^2\in\Sosp{A[\z]},
$$
as required.
\end{proof}

We formalize the following result, which is a consequence of the main result of \cite{frs2}.

\begin{lem}\label{sosok}
Let $\kappa$ be a (formally) real field with $\tau(\kappa)<+\infty$ and $\Qq:=a_0+a_1\z+a_2\z^2\in\kappa[[\x,\y]][\z]$ a positive semidefinite quadratic polynomial. Then $\Qq$ is a sum of $p:=4\tau(\kappa)$ squares of polynomials of degree $\leq1$ (with respect to $\z$) and coefficients in $\kappa[[\x,\y]]$.
\end{lem}
\begin{proof}
By Lemma \ref{quadratic} we have $a_0,a_2,4a_0a_2-a_1^2\in\psd(\kappa[[\x,\y]])$. Fix $n\geq1$ and observe that $a_0+(\x^2+\y^2)^n,a_2+(\x^2+\y^2)^n\in\psd^+(\kappa[[\x,\y]])$. In addition,
\begin{multline*}
4(a_0+(\x^2+\y^2)^n)(a_2+(\x^2+\y^2)^n)-a_1^2\\
=(4a_0a_2-a_1^2)+4(a_0+a_2)(\x^2+\y^2)^n+4(\x^2+\y^2)^{2n}\in\psd^+(\kappa[[\x,\y]]).
\end{multline*}
By Lemma \ref{psdkxy} there exists $r\geq1$ such that if $a_i-a_i'\in\gtm_2^r$, then
$$
a_0'+(\x^2+\y^2)^n,a_2'+(\x^2+\y^2)^n,4(a_0'+(\x^2+\y^2)^n)(a_2'+(\x^2+\y^2)^n)-a_1'^2\in\psd^+(\kappa[[\x,\y]]).
$$
In particular, we may assume that each $a_i'$ is a polynomial and proceeding similarly to the proof of Lemma \ref{density}, we find $M\in\Sos{\kappa}$ such that 
$$
a_0'+M^2(\x^2+\y^2)^n,a_2'+M^2(\x^2+\y^2)^n,4(a_0'+M^2(\x^2+\y^2)^n)(a_2'+M^2(\x^2+\y^2)^n)-a_1'^2\in\psd^+(\kappa[\x,\y]).
$$
By Lemma \ref{quadratic} 
$$
\Qq':=a_0'+M^2(\x^2+\y^2)^n+a_1'\z+(a_2'+M^2(\x^2+\y^2)^n)\z^2\in\psd(\kappa[\x,\y][\z])\subset\psd(\kappa((\x))[\y][\z]).
$$
As $\tau(\kappa((\x)))=\tau(\kappa)<+\infty$, we deduce by \cite[Thm.1.2]{frs2} that $\Qq'$ is a sum of $p$ squares of polynomials $A_i\in\kappa((\x))[\y][\z]$ of degree $\leq1$ with respect to $\z$, that is,
$$
\Qq'=A_1^2+\cdots+A_p^2.
$$
Let $m\geq1$ be such that $A'_i:=\x^mA_i\in\kappa[[\x]][\y][\z]$ for each $i$, so
$$
\x^{2m}\Qq'=A_1'^2+\cdots+A_p'^2.
$$
As $\kappa$ is a (formally) real field and $\x$ generates a real prime ideal of $\kappa[[\x]][\y][\z]$, we deduce $\x$ divides each $A'_i$. Proceeding inductively ($m$ times), we conclude $\Qq'$ is a sum of $p$ squares of polynomials $B_i\in\kappa[[\x]][\y][\z]$ of degree $\leq1$ with respect to $\z$. Thus,
\begin{equation}\label{a012}
\Qq=a_0+a_1\z+a_2\z^2=({\tt X}_1+\z{\tt Y}_1)^2+\cdots+({\tt X}_p+\z{\tt Y}_p)^2
\end{equation}
has a solution modulo $\gtm_2^n$ in $\kappa[[\x,\y]]$ for each $n\geq1$. By Strong Artin's approximation equation \eqref{a012} has a solution in $\kappa[[\x,\y]]$, so $\Qq$ is a sum of $p$ squares of polynomials of degree $\leq1$ (with respect to $\z$) and coefficients in $\kappa[[\x,\y]]$, as required. 
\end{proof}

We are now ready to prove Theorem \ref{et}.
\begin{proof}[Proof of Theorem \em\ref{et}]
By Weierstrass preparation theorem and Tschirnhaus trick there exist $a,b,c\in\kappa$ with $b,c\neq0$, a unit $U\in\kappa[[\x,\y]]$ with $U(0,0)=1$, units $u,v\in\kappa[[\y]]$ with $u(0)=1$ and $v(0)=1$ and a Weierstrass polynomial $P:=\x^3+a\x\y^\ell u^\ell+b\y^\rho v$ such that $\ell\geq2$, $\rho\geq3$ and $F=cU^2P$. In addition, either $\rho$ is odd or $\rho$ is even and $b\not\in-\Sos{\kappa}$. After the change of coordinates $(\x,\y,\z)\mapsto(c\x,c\y,c^2U\z)$ we assume $F=\x^3+a'\x\y^\ell u'^\ell+b'\y^{\rho}v'$ where $a',b'\in\kappa$ and $u',v'\in\kappa[[\y]]$ are units such that $u'(0)=1$ and $v'(0)=1$. Again, either $\rho$ is odd or $\rho$ is even and $b'\not\in-\Sos{\kappa}$. After the change of coordinates $\y u'\mapsto\y$ we suppose from the beginning $F=\x^3+a\x\y^\ell+b\y^\rho w(\y)$ where $a,b\in\kappa$, $b\neq0$ and $w$ is a unit such that $w(0)=1$. In addition, either $\rho$ is odd or $\rho$ is even and $b\not\in-\Sos{\kappa}$.

(i) As $f+\z g\in\psd(A)\setminus\{0\}$, we know that $f\in\psd(\{F\geq0\})$ and $f^2-Fg^2\in\psd(\kappa[[\x,\y]])$. Let $u\in\kappa[[\x]]$ be a unit such that $F(\x,0)=\x^3u$. As $-F$ is not positive semidefinite, if $f=0$, then $g=0$, so $f\neq0$. Assume $f(\x,0)=0$. We claim: \em $\y^2$ divides $f$\em.

Otherwise, $f=\y f'$ where $f'\in\kappa[[\x,\y]]$ and $f'(\x,0)\neq0$. Thus, there exists a Weierstrass polynomial $Q:=\x^q+\sum_{k=0}^{q-1}\y b_k(\y)\x^k$ (where $b_k\in\kappa[[\y]]$ for each $k$), $\mu\in\kappa\setminus\{0\}$ and a unit $V\in\kappa[[\x,\y]]$ with $V(0,0)=1$ such that $f'=Q\mu V^2$.

For each $m\geq1$ consider the homomorphism
$$
\varphi:\kappa[[\x,\y]]\to\kappa[[\t]],\ h\mapsto h(\t^2,\t^m).
$$
If $m$ is odd and large enough,
\begin{align*}
\varphi(F)&=\t^6+a\t^{m\ell+2}+b\t^{\rho m}w(\t^m)=\t^6(1+a\t^{m\ell-4}+b\t^{\rho m-6}w(\t^m)),\\
\varphi(f)&=\varphi(\y f')=\t^m\Big(\t^{2q}+\sum_{k=0}^{q-1}\t^{m+2k}b_k(\t^m)\Big)\mu V^2(\t^2,\t^m)\\
&=\mu\t^{m+2q}\Big(1+\sum_{k=0}^{q-1}\t^{m+2k-2q}b_k(\t^m)\Big)V^2(\t^2,\t^m).
\end{align*}
Observe that $\varphi(F)\in\Sos{\kappa[[\t]]}$, whereas $\varphi(f)\not\in\Sos{\kappa}$, as it has odd order. There exists a prime cone $\alpha\in\Sper(\kappa[[\t]])$ such that $\varphi(F)>_\alpha0$, whereas $\varphi(f)<_\alpha0$ (consider an ordering of $\kappa$ and define the sign of $\t$ to get $\mu\t<0$). This is a contradiction because $f\in\psd(\{F\geq0\})$.

We conclude $f'(\x,0)=0$, so $f=\y^2f''$ where $f''\in\kappa[[\x,\y]]$, as claimed. 

Consequently,
$$
f^2-Fg^2=\y^4f''^2-Fg^2=\y^4f''^2-(\x^3+a'\x\y^\ell u'^\ell+b'\y^{\rho}v')g^2\in\psd(\kappa[[\x,\y]]).
$$ 
Setting $\y=0$, we have $-\x^3g^2(\x,0)\in\psd(\kappa[[\x]])$, so $g(\x,0)=0$ and there exists $g'\in\kappa[[\x,\y]]$ such that $g=\y g'$. Thus,
$$
f^2-Fg^2=\y^2(\y^2f''^2-Fg'^2)\quad\leadsto\quad\y^2f''^2-Fg'^2\in\psd(\kappa[[\x,\y]]).
$$
Setting again $\y=0$, we have $-\x^3g'^2(\x,0)\in\psd(\kappa[[\x]])$, so $g'(\x,0)=0$ and there exists $g''\in\kappa[[\x,\y]]$ such that $g=\y^2g''$. This means $f+\z g=\y^2(f''+\z g'')$ where $f''+\z g''\in\psd(A)$. Proceeding recursively we find $s\geq1$ and $f_1+\z g_1\in\psd(A)$ such that $f_1(\x,0)\neq0$ and $f+\z g=\y^{2s}(f_1+\z g_1)$.

Assume $\omega(f_1(\x,0))=q\geq2$. Substitute $\y=0$ in $f_1^2-Fg_1^2\in\psd(\kappa[[\x,\y]])$ and observe that
$$
(f_1^2-Fg_1^2)(\x,0)=f_1^2(\x,0)-\x^3g_1^2(\x,0)\in\psd(\kappa[[\x,\y]]).
$$
Thus, $\omega(-\x^3g_1^2(\x,0))\geq2q$, that is, $\omega(g_1(\x,0))\geq q-\frac{3}{2}>0$, so $\omega(g_1)\geq1$.

(ii) The proof of this part is conducted in several steps. For the sake of simplicity once we have finished a step, we reset the local notation involved in such step. Recall that $F=\x^3+a\x\y^\ell+b\y^\rho w(\y)$ where $a,b\in\kappa$, $b\neq0$, $\ell\geq2$, $\rho\geq3$ and $w$ is a unit such that $w(0)=1$. In addition, either $\rho$ is odd or $\rho$ is even and $b\not\in-\Sos{\kappa}$.

As $F$ is $\rho$-quasidetermined, there exists a constant $c\in\kappa\setminus\{0\}$ and a unit $U\in\kappa[[\x,\y]]$ such that $F$ is right equivalent to $(\x^3+a\x\y^\ell+b\y^\rho)cU^2$. After the additional change of coordinates $(\x,\y,\z)\mapsto(c\x,c\y,c^2U\z)$, we may assume $F=\x^3+a\x\y^\ell+b\y^\rho$ for some $a,b\in\kappa$, $b\neq0$, $\ell\geq2$ and $\rho\geq1$. In addition, either $\rho$ is odd or $\rho$ is even and $b\not\in-\Sos{\kappa}$.

\paragraph{}\label{redest} Let $f+\z g\in\psd(A)$ be such that $\omega(f(\x,0))=k\geq 2$ and let us prove $f+\z g\in\Sosp{A}$. If $g=0$, then $f\in\psd(A)\setminus\{0\}$. Thus, $f+\z f=f(1+\z)\in\psd(A)$ and $f\neq0$. As $\frac{1}{1+\z}$ is a square in $A$, changing $f+\z g$ by $f+\z f$, we may assume $g\neq0$. By Lemma \ref{sosok} it is enough to find $\eta\in\psd(\kappa[[\x,\y]])\setminus\{0\}$ such that $f+\z g+\eta(\z^2-F)\in\psd(\kappa[[\x,\y]][\z])$. By Lemma \ref{quadratic} this is equivalent to find $\eta\in\psd(\kappa[[\x,\y]])\setminus\{0\}$ such that $4\eta(f-\eta F)-g^2,f-\eta F\in\psd(\kappa[[\x,\y]])$. We claim that to prove (ii) in the statement (of Theorem \ref{et}): \em It is enough to find $\eta\in\psd(\kappa[[\x,\y]])$ such that $4\eta(f-\eta F)-g^2\in\psd(\kappa[[\x,\y]])$\em. We have to check that under such hypotheses, also $f-\eta F\in\psd(\kappa[[\x,\y]])$. 

Let $\beta\in\Sper(\kappa[[\x,\y]])$ and $\alpha\in\Sper(\kappa)$ be such that $\beta\to\alpha$ (see \cite[Prop.II.2.4]{abr}). If $\supp(\beta)=\gtm_2$, then $f,F\in\supp(\beta)$ and $f-\eta F\in\supp(\beta)$, so we assume $\supp(\beta)\neq\gtm_2$. If $\supp(\beta)=0$, by \cite[Prop.VII.5.1]{abr} there exists $\beta_1\in\Sper(\kappa[[\x,\y]])$ such that $\beta\to\beta_1\to\alpha$ and $\supp(\beta_1)$ is a real prime ideal of height $1$, so we may assume ${\rm ht}(\supp(\beta))=1$. We obtain 
$$
\eta+\supp(\beta)\geq_\beta0\quad\text{and}\quad4\eta(f-\eta F)-g^2+\supp(\beta)\geq_\beta0. 
$$

Assume first $\supp(\beta)=(\y)$. If $f-\eta F+\supp(\beta)=f(\x,0)-\eta(\x,0)F(\x,0)<_\beta0$, then $\eta(\x,0)=0$ and $g(\x,0)=0$ (because $4\eta(\x,0)(f(\x,0)-\eta(\x,0)F(\x,0))-g^2(\x,0)+\supp(\beta)\geq_\beta0$), so $f(\x,0)<_\beta0$. Let $m\geq1$ be such that $\eta=\y^{2m}\eta'$ where $\eta'\in\psd(\kappa[[\x,\y]])$ and $\eta'(\x,0)\neq0$. As $4\eta(f-\eta F)-g^2\in\psd(\kappa[[\x,\y]])$, we deduce that $\y^m$ divides $g$, so there exists $g'\in\kappa[[\x,\y]]$ such that $g=\y^mg'$. Thus, $4\eta'(f-F\eta)-g'^2\in\psd(\kappa[[\x,\y]])$. Hence, $4\eta'(\x,0)f(\x,0)-g'^2(\x,0)\geq_\beta0$, so $\eta'(\x,0)f(\x,0)\geq_\beta0$ and $\eta'(\x,0)\geq_\beta0$. This means $f(\x,0)\geq_\beta0$, which is a contradiction. Consequently, $f-\eta F+\supp(\beta)\geq_\beta0$.

Assume next $\y\not\in\supp(\beta)$ and $f-\eta F<_\beta0$. Without loss of generality we suppose $\y>_\beta0$. Let $R(\alpha)$ be the real closure of $(\kappa,\leq_\alpha)$. By Lemma \ref{dimen1} there exist $\zeta\in R(\alpha)[[\y]]$ and $q\geq1$ such that $(f-\eta F)(\zeta,\y^q)<0$, $\eta(\zeta,\y^q)\geq0$ and $(4\eta(f-\eta F)-g^2)(\zeta,\y^q)\geq0$ (for $\y>0$). Thus, $\eta(\zeta,\y^q)=0$ and $g(\zeta,\y^q)=0$, so $f(\zeta,\y^q)<0$. Write $\eta=\y^{2m}\eta'$ where $m\geq0$ and $\eta'\in\kappa[[\x,\y]]$ satisfies $\eta'(\x,0)\neq0$. Thus, $\eta'$ is a regular series with respect to $\x$ and there exist a Weierstrass polynomial $Q\in\kappa[[\y]][\x]$ and a unit $V\in\kappa[[\x,\y]]$ such that $\eta'=QV$ and $\eta=\y^{2m}QV$. As $\y^{2m}Q(\zeta(\y^{1/q}),\y)V(\zeta(\y^{1/q}),\y)=\eta(\zeta(\y^{1/q}),\y)=0$, we deduce $Q(\zeta(\y^{1/q}),\y)=0$ and $\zeta(\y^{1/q})\in R(\alpha)[[\y^{1/q})]]$ is integral over $\kappa[[\y]]$. Thus, the minimal polynomial of $\zeta(\y^{1/q})$ over $\kappa((\y))$ is an irreducible Weierstrass polynomial $Q_0\in\kappa[[\y]][\x]$. As $\eta\in\psd(\kappa[[\x,\y]])\subset\psd(R(\alpha)[[\x,\y]])$, we conclude there exist $\mu\geq1$ and $\eta_1\in\psd(\kappa[[\x,\y]])$ such that $\eta=Q_0^{2\mu}\eta_1$ and $\eta_1(\zeta(\y^{1/q}),\y)\neq0$. As $4\eta(f-\eta F)-g^2\in\psd(\kappa[[\x,\y]])\subset\psd(R(\alpha)[[\x,\y]])$, we have $Q_0^\mu$ divides $g$, so there exists $g_1\in\kappa[[\x,\y]]$ such that $g=Q_0^\mu g_1$. Thus, $4\eta_1(f-FQ_0^{2\mu}\eta_1)-g_1^2\in\psd(\kappa[[\x,\y]])$. Hence, $4\eta_1(\zeta,\y^q)f(\zeta,\y^q)-g_1^2(\zeta,\y^q)\geq0$, so $\eta_1(\zeta,\y^q)f(\zeta,\y^q)\geq0$. As $\eta_1(\zeta,\y^q)>0$, we deduce $f(\zeta,\y^q)\geq0$, which is a contradiction. Consequently, $f-\eta F+\supp(\beta)\geq_\beta0$, as claimed.

\paragraph{}\label{conest} We may assume: \em $f,g\in\kappa[\x,\y]\setminus\{0\}$, $f$ is a Weierstrass polynomial with respect to $\x$ and there exists $r>q$ such that $f-\y^{2r}+\z g\in\psd^\oplus(A)$\em.

Write $f(\x,0):=c\x^q+\cdots$ for some $c\in\kappa\setminus\{0\}$. Let $\alpha\in\Sper(\kappa)$ and $R(\alpha)$ be the real closure of $(\kappa,\leq_\alpha)$. Consider the homomorphism $\varphi:\kappa[[\x,\y]]\to R(\alpha)[[\y]],\ h\mapsto h(\x^2,0)$. As $\varphi(F)=\x^6>0$, we have $f(\x^2,0)=c\x^{2q}+\cdots>0$, so $c\geq_\alpha0$. Thus, $c$ is non-negative for each ordering of $\kappa$, so $c\in\psd(\kappa)=\Sos{\kappa}$. Hence, we divide $f+\z g$ by $c$ and assume in the following $c=1$. 

By Weierstrass preparation theorem there exist a Weierstrass polynomial $P\in\kappa[[\y]][\x]$ of degree $q=\omega(f(\x,0))$ and a unit $U\in\kappa[[\x,\y]]$ such that $U(0,0)=1$ and $f=PU^2$. We divide $f+\z g$ by $U^2$ and assume that $f$ is a Weierstrass polynomial with respect to $\x$ of degree $q$. Observe that $f(\x,0)=\x^q$ and $f^2(\x,0)-F(\x,0)g^2(\x,0)=\x^{2q}-\x^3g^2(\x,0)\geq0$, so $f^2(\x,0)-F(\x,0)g^2(\x,0)\neq0$. As also $g\neq0$, if $n$ is large enough, then $f+\y^{2n}+\z g\in\psd^\oplus(A)$ (use Lemma \ref{pd+}(iii)). By Corollaries \ref{orderings} and \ref{psdkxy} there exists $r>\max\{2n,q\}$ such that for each $f_n,g_n\in\kappa[\x,\y]$ satisfying $f-f_n,g-g_n\in\gtm_2^r$ it holds $f_n+\y^{2n}-\y^{2r}+\z g_n\in\psd^\oplus(A)$. As $f$ is a Weierstrass polynomial with respect to $\x$, also $f_n+\y^{2n}$ is a Weierstrass polynomial with respect to $\x$.

If we show that each $f_n+\y^{2n}+\z g_n\in\Sosp{A}$, Strong Artin's approximation guarantees that $f+\z g\in\Sosp{A}$. Thus, we assume in the following that \ref{conest} holds. 

\paragraph{}\label{onlythis0} Strong Artin's approximation implies the following: \em Let $f,g\in\kappa[\x,\y]\setminus\{0\}$ and $r>q$ be such that $f$ is a Weierstrass polynomial with respect to $\x$ and $f-\y^{2r}+\z g\in\psd^\oplus(A)$. There exists $n_0\geq1$ such that if $f+2\x^{2n}+\z g\in\Sosp{A}$ for some $n\geq n_0$, then $f+\z g\in\Sosp{A}$\em. 

\paragraph{}\label{onlythis} To prove \ref{onlythis0} it is enough by \ref{redest} to show: \em Let $f,g\in\kappa[\x,\y]\setminus\{0\}$ and $r>q$ be such that $f$ is a Weierstrass polynomial with respect to $\x$ and $f-\y^{2r}+\z g\in\psd^\oplus(A)$. Then there exist $\eta\in\kappa[\x,\y]\cap\psd(\kappa[[\x,\y]])$ and $n\geq n_0$ such that $4\eta(f+2\x^{n}-\eta F)-g^2\in\psd(\kappa[[\x,\y]])$\em.

To find $\eta$ we need some preliminary work. Write $f^*:=f-\y^{2r}:=\x^q+\sum_{k=0}^{q-1}\y\lambda_k(\y)\x^k\in\kappa[[\y]][\x]$. We have
\begin{align}
f^*(\x,\u\x^q)&=\x^q+\sum_{j=0}^{q-1}\x^q\u \lambda_j(\u\x^q)\x^j=\x^q\Big(1+\sum_{j=0}^{q-1}\lambda_j(\u\x^q)\u\x^j\Big)\label{f1}\\
F(\x,\u\x^q)&=\x^3(1+a\x^{\ell q-2}\u^\ell+b\x^{\rho q-3}\u^{\rho})\label{F1}
\end{align}

Define $W:=1+a\x^{\ell q-2}\u^\ell+b\x^{\rho q-3}\u^{\rho}\in\kappa[\x,\u]$, which satisfies $W(0,0)=1$ and $F(\x,\u\x^q)=\x^3W$. Define $A':=\kappa[[\x,\u,\v]]/(\v^2-\x W)$ and consider the homomorphism 
$$
\psi:A\to A',\ h(\x,\y,\z)\mapsto h(\x,\u\x^q,\v\x).
$$ 
As $f^*+\z g\in\psd(A)$, we have
$$
\psi(f^*+\z g)=f^*(\x,\u\x^q)+\v\x g(\x,\u\x^q)\in\psd(A').
$$

As $W(0,0)=1$, there exists by the Implicit Function Theorem $\varphi(\u,\v)\in\kappa[[\u,\v]]$ such that $\x=\varphi(\u,\v)$ is the unique solution of the equation $\v^2-\x W=0$ satisfying $\varphi(0,0)=0$. As $\psi(f^*+\z g)\in\psd(A')$, we get
$$
h(\u,\v):=f^*(\varphi(\u,\v),\u\varphi(\u,\v)^q)+\v\varphi(\u,\v)g(\varphi(\u,\v),\u\varphi(\u,\v)^q)\in\psd(\kappa[[\u,\v]])=\Sosp{\kappa[[\u,\v]]},
$$
so there exist $h_i\in\kappa[[\u,\v]]$ such that $h=\sum_{i=1}^ph_i^2$. 

Using the relation $\v^2-\x W$, we find series $a_i,b_i,c\in\kappa[[\x,\u]]$ satisfying
$$
f^*(\x,\u\x^q)+\v\x g(\x,\u\x^q)=\sum_{i=1}^p(a_i+\v b_i)^2-(\v^2-\x W)c.
$$
Comparing coefficients with respect to $\v$, we obtain
\begin{align}
c&=\sum_{i=1}^pb_i^2,\\
f^*(\x,\u\x^q)&=\sum_{i=1}^pa_i^2+\x Wc,\label{equ2}\\
\x g(\x,\u\x^q)&=2\sum_{i=1}^pa_ib_i.\label{equ3}
\end{align}
As $g\neq0$, we have $c\neq0$. If we set $\x=0$ in \eqref{equ2}, we obtain $0=\sum_{i=1}^pa_i^2(0,\u)$, so $a_i(0,\u)=0$ for each $i$ and there exists $a_i'\in\kappa[[\x,\u]]$ such that $a_i=\x a_i'$. As $q:=\omega(f(\x,0))\geq2$, we deduce $\x^2$ divides $f^*(\x,\u\x^q)$. Thus, $\x^2$ divides $f^*(\x,\u\x^q)-\sum_{i=1}^pa_i^2=\x Wc$, so $\x$ divides $c=\sum_{i=1}^pb_i^2$. Hence, $\x$ divides $b_i$ for each $i$, so $b_i=\x b_i'$ for some $b_i'\in\kappa[[\x,\u]]$. Write $c':=\sum_{i=1}^pb_i'^2\in\Sos{\kappa[[\x,\u]]}$, so $c=\x^2c'$. Rewrite equations \eqref{equ2} and \eqref{equ3} as
\begin{align*}
&f^*(\x,\u\x^q)=\sum_{i=1}^p(\x a_i')^2+\x^3 Wc',\\
&\x g(\x,\u\x^q)=2\x^2\sum_{i=1}^pa_i'b_i'.
\end{align*}
As $c'\in\psd(\kappa[[\x,\u]])$ and 
$$
f^*(\x,\u\x^q)+\v\x g(\x,\u\x^q)+(\v^2-\x W)\x^2c'=\sum_{i=1}^p(a_i+\v b_i)^2\in\psd(\kappa[[\x,\u]][\v]),
$$
we deduce by Lemma \ref{quadratic} 
\begin{align*}
&4\x^2c'(f^*(\x,\u\x^q)-\x W\x^2c')-(\x g(\x,\u\x^q))^2\in\psd(\kappa[[\x,\u]]),\\
&f^*(\x,\u\x^q)-\x^3Wc'=\sum_{i=1}^p(\x a_i')^2\in\psd(\kappa[[\x,\u]]).
\end{align*}
As $f^*=f-\y^{2r}$, we have $f^*(\x,\u\x^q)=f(\x,\u\x^q)-(\u\x^q)^{2r}$. As in addition $c'\neq0$, we deduce
\begin{align*}
&4\x^2c'(f(\x,\u\x^q)-\x W\x^2c')-(\x g(\x,\u\x^q))^2\in\psd(\kappa[[\x,\u]])\setminus\{0\},\\
&f(\x,\u\x^q)-\x^3Wc'=(\u\x^q)^{2r}+\sum_{i=1}^p(\x a_i')^2\in\psd(\kappa[[\x,\u]])\setminus\{0\}.
\end{align*}
Consequently,
\begin{align}
\xi&:=4c'(f(\x,\u\x^q)-F(\x,\x^q\u)c')-(g(\x,\u\x^q))^2\in\psd(\kappa[[\x,\u]])\setminus\{0\}\label{xi},\\
\psi&:=f(\x,\u\x^q)-F(\x,\x^q\u)c'\in\psd(\kappa[[\x,\u]])\setminus\{0\}\label{psi}.
\end{align}

\paragraph{}One would like to construct $\eta\in\kappa[\x,\y]\cap\psd(\kappa[[\x,\y]])$ from $c'$ substituting $\u$ by $\frac{\y}{\x^q}$. But at this point this does not work because $c'\in\kappa[[\x,\u]]$ and the desired substitution is not possible. We need to modify first both $c'$ and $\xi$ to have polynomials instead of series, but keeping both inside $\psd(\kappa[[\x,\y]])$.

Let $\ell,m\geq1$ be such that $\x^{2\ell}$ divides $c'$ but $\x^{2\ell+1}$ does not and $\x^{2m}$ divides $\xi$ but $\x^{2m+1}$ does not. Write $c'':=\frac{c'}{\x^{2\ell}}\in\psd(\kappa[[\x,\u]])$ and $\xi':=\frac{\xi}{\x^{2m}}\in\psd(\kappa[[\x,\u]])$, so $c'=c''\x^{2\ell}$ and $\xi=\xi'\x^{2m}$. Observe that $c''(0,\u)\neq0$ and $\xi'(0,\u)\neq0$. Thus, $c''+\x^{2n},\xi'+\x^{2n}\in\psd^+(\kappa[[\x,\u]])$ for each $n\geq1$. Write $k:=n_0+\ell$. We have 
\begin{equation}\label{xiast}
\begin{split}
\xi^*:\hspace{-1mm}&=4(c'+\x^{2k+2m})(f(\x,\u\x^q)+2\x^{2k+2m}-F(\x,\x^q\u)(c'+\x^{2k+2m}))-(g(\x,\u\x^q))^2\\
&=\xi+4(\x^{4k+4m}+\x^{2k+2m}c'(2-F(\x,\x^q\u))+\x^{2k+2m}\psi
+\x^{4k+4m}(1-F(\x,\x^q\u)))\\
&=\x^{2m}(\xi'+4(\x^{4k+2m}+\x^{2k}c'(2-F(\x,\x^q\u))+\x^{2k}\psi
+\x^{4k+2m}(1-F(\x,\x^q\u)))).
\end{split}
\end{equation}
As $\xi'+\x^{4k+2m}\in\psd^+(\kappa[[\x,\u]])$ and $\psi\in\psd(\kappa[[\x,\u]])$, also 
$$
\xi'':=\xi'+4(\x^{4k+2m}+\x^{2k}c'(2-F(\x,\x^q\u))+\x^{2k}\psi+\x^{4k+2m}(1-F(\x,\x^q\u)))\in\psd^+(\kappa[[\x,\u]]).
$$
We have $\xi^*=\x^{2m}\xi''$ and $c''+\x^{2k+2m-2\ell}\in\psd^+(\kappa[[\x,\u]])$. As $f$ is a Weierstrass polynomial, write
\begin{align*}
&f(\x,\u\x^q)=\x^q\Big(1+\sum_{i=1}^d\x^if_i(\u)\Big),\\
&g(\x,\u\x^q)=\x^s\Big(\sum_{j=0}^e\x^jg_j(\u)\Big),\\
&c''(\x,\u)=\sum_{i\geq0}\x^ic''_i(\u)
\end{align*}
where $f_i,g_j\in\kappa[\u]$, $c''_i\in\kappa[[\u]]$, $g_0,f_d,g_e\neq0$ and $c''_0=c''(0,\u)\neq0$. Let us collect more information concerning the structure and the properties of $c'$.

\paragraph{}\label{H} Denote $\veps=1$ if $q$ is odd and $\veps=0$ if $q$ is even. We claim: \em $f(\x,\u\x^q)-F(\x,\x^q\u)c'=\x^{q+\veps}H$ where $H\in\psd(\kappa[[\x,\u]])$ and $c_i''\in\kappa[\u]$ for $i=0,\ldots,2m-2\ell-q-1$\em. Recall that $c'=c''\x^{2\ell}$.

We have
{\small\begin{equation}\label{xiexp}
\begin{split}
\xi&=4c'(f(\x,\u\x^q)-F(\x,\x^q\u)c')-(g(\x,\u\x^q))^2\\
&=4\x^{2\ell}\Big(\sum_{i\geq0}\x^ic''_i(\u)\Big)\Big(\x^q\Big(1+\sum_{i=1}^d\x^if_i(\u)\Big)-(\x^3 W)\x^{2\ell}\Big(\sum_{i\geq0}\x^ic''_i(\u)\Big)\Big)-\x^{2s}\Big(\sum_{j=0}^e\x^jg_j(\u)\Big)^2\\
&=4\x^{2\ell+q}\Big(\sum_{i\geq0}\x^ic''_i(\u)\Big)\Big(1+\sum_{i=1}^d\x^if_i(\u)-\x^{2\ell+3-q} W\Big(\sum_{i\geq0}\x^ic''_i(\u)\Big)\Big)-\x^{2s}\Big(\sum_{j=0}^e\x^jg_j(\u)\Big)^2.
\end{split}
\end{equation}}
Although we have not yet proved it, we will also see: 
\begin{equation}\label{bound0}
2\ell+3-q\geq0\quad\text{and}\quad 2m-2\ell-q\geq0.
\end{equation} 

We distinguish two cases:

\noindent{\sc Case 1.} {\em $q$ is odd}. As
\begin{equation}\label{qood}
\x^q\Big(1+\sum_{i=1}^d\x^if_i(\u)\Big)-\x^{2\ell+3} W\Big(\sum_{i\geq0}\x^ic''_i(\u)\Big)\in\psd(\kappa[[\x,\u]]),
\end{equation}
its initial form is positive semidefinite, so it cannot have degree $q$ odd. As $c_0''(\u)\neq0$ and $W=1+a\x^{\ell q-2}\u^\ell+b\x^{\rho q-3}\u^{\rho}$ (where $\ell\geq1$, $\rho\geq3$ and $q\geq2$), we deduce $\x^q-\x^{2\ell+3}c_0''(\u)=0$, so $q=2\ell+3$ and 
$c_0''(\u)=1$. Thus, $f(\x,\u\x^q)-F(\x,\x^q\u)c'=\x^{q+1}H$ where $H\in\psd(\kappa[[\x,\u]])$.

As $\xi\in\psd(\kappa[[\x,\u]])$, we deduce $2\ell+q+1\leq 2s$ (see \eqref{xiexp}). Write $W=1+\x^2\Gamma$ where $\Gamma:=a\x^{\ell q-4}\u^\ell+b\x^{\rho q-5}\u^{\rho}\in\kappa[\x,\u]$ (as $q\geq2$, $\ell\geq2$ and $\rho\geq3$, we have $\ell q-4\geq0$ and $\rho q-5\geq1$). We obtain from \eqref{xiexp} (using that $q=2\ell+3$, $2\ell+q+1\leq 2s$ and $c_0''(\u)=1$)
\begin{equation*}
\begin{split}
\x^{2m}\xi'&=\xi=4\x^{2\ell+q}\Big(\Big(1+\sum_{i\geq1}\x^ic''_i(\u)\Big)
\Big(1+\sum_{i=1}^d\x^if_i(\u)-(1+\x^2\Gamma)\Big(1+\sum_{i\geq1}\x^ic''_i(\u)\Big)\Big)\\
&-\x^{2s-2\ell-q}\Big(\sum_{j=0}^e\x^jg_j(\u)\Big)^2\Big)=4\x^{2\ell+q+1}\Big(\Big(1+\sum_{i\geq1}\x^ic''_i(\u)\Big)\\
&\cdot\Big(\sum_{i=1}^d\x^{i-1}f_i(\u)-\sum_{i\geq1}\x^{i-1}c''_i(\u)-\x\Gamma\Big(1+\sum_{i\geq1}\x^ic''_i(\u)\Big)\Big)-\x^{2s-2\ell-q-1}\Big(\sum_{j=0}^e\x^jg_j(\u)\Big)^2\Big).
\end{split}
\end{equation*}
Consequently, $2m-2\ell-q-1\geq0$ and
\begin{multline*}
4\Big(1+\sum_{i\geq1}\x^ic''_i(\u)\Big)\Big(\sum_{i=1}^d\x^{i-1}(f_i(\u)-c''_i(\u))-\x\Gamma\Big(1+\sum_{i\geq1}\x^ic''_i(\u)\Big)\Big)\\
-\x^{2s-2\ell-q-1}\Big(\sum_{j=0}^e\x^jg_j(\u)\Big)^2=\x^{2m-2\ell-q-1}\xi',
\end{multline*}
or equivalently,
\begin{multline}\label{solve}
\sum_{i\geq1}\x^{i-1}(f_i-c_i'')+\x\Big(\Big(\sum_{i\geq1}\x^{i-1}c_i''\Big)\Big(\sum_{i\geq1}\x^{i-1}(f_i-c_i'')\Big)-\Gamma\Big(1+\sum_{i\geq1}\x^ic''_i\Big)^2\Big)\\
-\x^{2s-2\ell-q-1}\frac{1}{4}\Big(\sum_{j=0}^e\x^jg_j(\u)\Big)^2=\x^{2m-2\ell-q-1}\frac{1}{4}\xi'
\end{multline}
where $f_i=0$ for $i\geq d+1$. Observe that the coefficient of $\x^i$ on the right hand side of equation \eqref{solve} is zero for $0\leq i\leq 2m-2\ell-q-2$.

Recall that $\Gamma\in\kappa[\x,\u]$ and $c_0''=1$. If one compares coefficients with respect to $\x$ in equation \eqref{solve}, one realizes that $c_i''-P_i(f_1,\ldots,f_d,g_0,\ldots,g_e,c_1'',\ldots,c_{i-1}'')=0$ for some polynomial 
$$
P_i\in\kappa[\x_1,\ldots,\x_d,\y_0,\ldots,\y_e,\z_1,\ldots,\z_{i-1}]
$$ 
if $i=1,\ldots,2m-2\ell-q-1$. We conclude inductively that $c_i''\in\kappa[\u]$ for $i=0,\ldots,2m-2\ell-q-1$.

\noindent{\sc Case 2.} {\em $q$ is even}. As
$$
\x^q\Big(1+\sum_{i=1}^d\x^if_i(\u)\Big)-\x^{2\ell+3}W\Big(\sum_{i\geq0}\x^ic''_i(\u)\Big)\in\psd(\kappa[[\x,\u]])
$$
and $c''_0(\u)\neq0$, we have $q<2\ell+3$, so $f(\x,\u\x^q)-F(\x,\x^q\u)c'=\x^qH$ where $H\in\psd(\kappa[[\x,\u]])$. 

As $\xi\in\psd(\kappa[[\x,\u]])$ and $\xi=\x^{2m}\xi'$, we deduce $2\ell+q\leq 2s$ and $2\ell+q\leq 2m$ (see \eqref{xiexp}). Thus, 
\begin{multline}\label{solve1}
\Big(\sum_{i\geq0}\x^ic''_i(\u)\Big)\Big(1+\sum_{i=1}^d\x^if_i(\u)-\x^{2\ell+3-q}W\Big(\sum_{i\geq0}\x^ic''_i(\u)\Big)\Big)\\
-\x^{2s-2\ell-q}\frac{1}{4}\Big(\sum_{j=0}^e\x^jg_j(\u)\Big)^2=\x^{2m-2\ell-q}\frac{1}{4}\xi'.
\end{multline}
Observe that the coefficient of $\x^i$ on the right hand side of equation \eqref{solve1} is zero for $0\leq i\leq 2m-2\ell-q-1$.

Recall that $W\in\kappa[\x,\u]$. If one compares coefficients with respect to $\x$ in equation \eqref{solve1}, one realizes that $c_i''-Q_i(f_1,\ldots,f_d,g_0,\ldots,g_e,c_1'',\ldots,c_{i-1}'')=0$ for some polynomial 
$$
Q_i\in\kappa[\x_1,\ldots,\x_d,\y_0,\ldots,\y_e,\z_1,\ldots,\z_{i-1}]
$$ 
if $i=0,\ldots,2m-2\ell-q-1$. We conclude inductively that $c_i''\in\kappa[\u]$ for $i=0,\ldots,2m-2\ell-q-1$, as claimed.

\paragraph{} We are ready to modify $c',\xi$ in order to obtain polynomials $c^\bullet,\xi^\bullet\in\kappa[\x,\u]\cap\psd(\kappa[[\x,\u]])$ (from which we will construct $\eta\in\kappa[\x,\y]\cap\psd(\kappa[[\x,\y]])$ and $n\geq n_0$ in the statement of \ref{onlythis}) instead of the positive semidefinite series $c',\xi$.

As $\xi'',c''+\x^{2k+2m-2\ell}\in\psd^+(\kappa[[\x,\u]])$, there exists by Corollary \ref{psdkxy} $r'\geq1$ such that if $\zeta,\theta\in\kappa[[\x,\u]]$ satisfy $\zeta-\xi'',\theta-c''\in\gtm_2^{r'}$, then $\zeta,\theta+\x^{2k+2m-2\ell}\in\psd^+(\kappa[[\x,\u]])$. We may assume $r'$ is strictly greater than the maximum of the degrees of the polynomials $\x^ic_i''$ for $i=0,\ldots,2m-2\ell-q-1$.

Let $c^\circ\in\kappa[\x,\u]$ be a polynomial such that $\omega(c^\circ-c'')\geq r'+2m$. Then $c^\circ=\sum_{i=0}^{2m-2\ell-q-1}c_i''\x^i+\x^{2m-2\ell-q}h$ for some polynomial $h\in\kappa[\x,\u]$, so $c^\circ=c''+\x^{2m-2\ell-q}h'$ for some $h'\in\kappa[[\x,\u]]$ and $\x^{2\ell}c^\circ=c'+\x^{2m-q}h'$. Observe that
$$
\omega(c^\circ+\x^{2k+2m-2\ell}-(c''+\x^{2k+2m-2\ell}))=\omega(c^\circ-c'')\geq r'+2m>r'.
$$
Thus, $c^\circ+\x^{2k+2m-2\ell}\in\kappa[\x,\u]\cap\psd^+(\kappa[[\x,\u]])$ and $c^\bullet:=\x^{2\ell}c^\circ+\x^{2k+2m}\in\kappa[\x,\u]\cap\psd(\kappa[[\x,\u]])$. Consider the polynomial
\begin{equation}\label{xibull}
\xi^\bullet:=4c^\bullet(f(\x,\u\x^q)+2\x^{2k+2m}-F(\x,\x^q\u)c^\bullet)-(g(\x,\u\x^q))^2.
\end{equation}
We claim: \em $\xi^\bullet-\xi=\x^{2m}E$ for some $E\in\kappa[[\x,\u]]$\em. 

By \eqref{xi} and \ref{H} we have:
\begin{align*}
&4c'(f(\x,\u\x^q)-F(\x,\x^q\u)c')-(g(\x,\u\x^q))^2=\xi,\\
&f(\x,\u\x^q)-F(\x,\x^q\u)c'=\x^{q+\veps}H.
\end{align*}
As $F(\x,\x^q\u)=\x^3W$, $\x^{2\ell}c^\circ=c'+\x^{2m-q}h$ and $c^\bullet=\x^{2\ell}c^\circ+\x^{2k+2m}=c'+\x^{2m-q}h'+\x^{2k+2m}$, we conclude
{\small\begin{equation*}
\begin{split}
\xi^\bullet&=4c^\bullet(f(\x,\u\x^q)+2\x^{2k+2m}-F(\x,\x^q\u)c^\bullet)-(g(\x,\u\x^q))^2\\
&=4(c'+\x^{2m-q}h'+\x^{2k+2m})(f(\x,\u\x^q)+2\x^{2k+2m}-F(\x,\x^q\u)(c'+\x^{2m-q}h'+\x^{2k+2m}))
-(g(\x,\u\x^q))^2\\
&=4c'(f(\x,\u\x^q)-F(\x,\x^q\u)c')-(g(\x,\u\x^q))^2+4(\x^{2m-q}h'+\x^{2k+2m})(f(\x,\u\x^q)-F(\x,\x^q\u)c')\\
&+4(c'+\x^{2m-q}h'+\x^{2k+2m})(2\x^{2k+2m}-\x^3W(\x^{2m-q}h'+\x^{2k+2m}))\\
&=\xi+4(\x^{2m-q}h'+\x^{2k+2m})\x^{q+\veps}H+4(c'+\x^{2m-q}h'+\x^{2k+2m})(2\x^{2k+2m}-\x^3W(\x^{2m-q}h'+\x^{2k+2m})).
\end{split}
\end{equation*}}
Thus, to justify that 
\begin{multline*}
\xi^\bullet-\xi=4\x^{2m+\veps}(h'+\x^{2k+q})H+4\x^{2k+2m}(2\x^{2k+2m}-\x^3W(\x^{2m-q}h'+\x^{2k+2m}))\\
+4(c'+\x^{2m-q}h')(2\x^{2k+2m}-\x^3W\x^{2k+2m})+4(c'+\x^{2m-q}h')(-\x^3W\x^{2m-q}h')
\end{multline*}
is divisible by $\x^{2m}$, we only have to clarify why the product 
\begin{multline*}
4(c'+\x^{2m-q}h')(-\x^3W(\x^{2m-q}h'))=4(\x^{2\ell}c''+\x^{2m-q}h')(-\x^3W(\x^{2m-q}h'))\\
=-4\x^{2m+2\ell+3-q}c''Wh'-4\x^{4m+3-2q}Wh'^2
\end{multline*} 
is divisible by $\x^{2m}$. The first addend on the right hand side is divisible by $\x^{2m}$ because $q\leq2\ell+3$ (see \eqref{bound0}), whereas the second addend is divisible by $\x^{2m}$ because, as $2m-2\ell-q\geq0$ and $2\ell+3-q\geq0$ (see \eqref{bound0}), 
$$
4m+3-2q=2m+(2m-2\ell-q)+(2\ell+3-q)\geq 2m.
$$
We conclude $\xi^\bullet-\xi=\x^{2m}E$ for some $E\in\kappa[[\x,\u]]$. 

As $\x^{2m}$ divides $\xi$, also $\x^{2m}$ divides $\xi^\bullet$, so there exists a polynomial $\xi^\circ\in\kappa[\x,\u]$ such that $\xi^\bullet=\x^{2m}\xi^\circ$. By \eqref{xiast} and \eqref{xibull} 
\begin{multline*}
\xi^\bullet-\xi^*=4\x^{2\ell}(f(\x,\u\x^q)+2\x^{2k+2m}-F(\x,\x^q\u)\x^{2k+2m})(c^\circ-c'')\\
-4\x^{2k+2m+2\ell}F(\x,\x^q\u)(c^\circ-c'')-4\x^{4\ell}F(\x,\x^q\u)(c^\circ+c'')(c^\circ-c''),
\end{multline*}
so $\omega(\xi^\bullet-\xi^*)\geq\omega(c^\circ-c'')$. As $\xi^\bullet-\xi^*=\x^{2m}(\xi^\circ-\xi'')$, we have $\omega(\xi^\circ-\xi'')=\omega(\xi^\bullet-\xi^*)-2m$. As $\omega(c^\circ-c'')\geq r'+2m$, it holds $\omega(\xi^\circ-\xi'')\geq r'$, so $\xi^\circ\in\psd^+(\kappa[[\x,\u]])$. This means that $\xi^\bullet=\x^{2m}\xi^\circ\in\psd(\kappa[[\x,\u]])$. 

\paragraph{} Write $\u:=\frac{\y}{\x^q}$ in $\xi^\bullet(\x,\frac{\y}{\x^q})\in\kappa[\x,\u]$: 
$$
\xi^\bullet\Big(\x,\frac{\y}{\x^q}\Big)=4c^\bullet\Big(\x,\frac{\y}{\x^q}\Big)\Big(f(\x,\y)+2\x^{2k+2m}-F(\x,\y)c^\bullet\Big(\x,\frac{\y}{\x^q}\Big)\Big)-(g(\x,\y))^2,
$$
where $c^\bullet\in\kappa[\x,\u]$. After clearing denominators, by Lemma \ref{blowup} we find an integer $\mu\geq1$ and a polynomial $B\in\kappa[\x,\y]\cap\psd(\kappa[[\x,\y]])$ such that
$$
4B(\x^{2\mu}f+2\x^{2k+2m+2\mu}-FB)-\x^{4\mu}g^2\in\kappa[\x,\y]\cap\psd(\kappa[[\x,\y]]).
$$
Thus, there exist series $A_i\in\kappa[[\x,\y]]$ satisfying
\begin{equation}\label{FB}
4B(\x^{2\mu}f+2\x^{2k+2m+2\mu}-FB)-\x^{4\mu}g^2=\sum_{i=1}^pA_i^2.
\end{equation}
Setting $\x=0$ in \eqref{FB}, we deduce
$$
-4F(0,\y)B^2(0,\y)=-4b\y^{\rho}B^2(0,\y)=\sum_{i=1}^pA_i^2(0,\y).
$$
As $\rho$ is either odd or it is even and $b\not\in-\Sos{\kappa}$, we deduce $B(0,\y)=0$ and $A_i(0,\y)=0$, so there exist $B',A_i'\in\kappa[[\x,\y]]$ such that $B=\x^2 B'$ (recall that $B\in\psd(\kappa[[\x,\y]])$) and $A_i=\x A_i'$. We deduce
\begin{equation}\label{FB2}
4\x^4B'(\x^{2\mu-2}f+2\x^{2k+2m+2\mu-2}-FB')-\x^{4\mu}g^2=\x^2\sum_{i=1}^pA_i'^2.
\end{equation}
Dividing \eqref{FB2} by $\x^2$ and setting $\x=0$ next, we deduce $A_i'(0,\y)=0$, so there exists $A_i''\in\kappa[[\x,\y]]$ such that $A_i=\x A_i''$. Consequently,
$$
4B'(\x^{2\mu-2}f+2\x^{2k+2m+2\mu-2}-FB')-\x^{4\mu-4}g^2=\sum_{i=1}^pA_i''^2.
$$
Proceeding recursively with $\mu$ we find the sought polynomial $\eta\in\kappa[\x,\y]\cap\psd(\kappa[[\x,\y]])$ such that
$$
4\eta(f+2\x^{2k+2m}-F\eta)-g^2\in\kappa[\x,\y]\cap\psd(\kappa[[\x,\y]])
$$
and $n:=k+m\geq n_0+\ell$, as required.
\end{proof}

We are ready to prove Theorem \ref{list2} for the rings $A=\kappa[[\x,\y,\z]]/(\z^2-F)$ where $F\in\kappa[[\x,\y]]$ is one of the series of order three quoted in its statement.

\begin{thm}[Order three]\label{order3}
Let $\kappa$ be a (formally) real field such that $\tau(\kappa)<+\infty$. Denote $A:=\kappa[[\x,\y,\z]]/(\z^2-F)$ where $F\in\kappa[[\x,\y]]$ is one of the following series:
\begin{itemize}
\item[(iv)] $\x^2\y+(-1)^ka\y^k$ where $a\not\in-\Sos{\kappa}$ and $k\geq 3$,
\item[(v)] $\x^2\y$,
\item[(vi)] $\x^3+a\x\y^2+b\y^3$ irreducible,
\item[(vii)] $\x^3+a\y^4$ where $a\not\in-\Sos{\kappa}$, 
\item[(viii)] $\x^3+\x\y^3$,
\item[(ix)] $\x^3+\y^5$. 
\end{itemize}
Then $\psd(A)=\Sos{A}$ and $p(A)\leq4\tau(\kappa)$.
\end{thm}
\begin{proof}
The general strategy to prove this result (except for case (v)) is to show that (maybe after a suitable change of coordinates when needed) there exists no positive semidefinite element $f+\z g\in\psd(A)$ such that $f(\x,0)$ has order $1$, whereas $F(\x,0)$ has order $3$ and $F$ satisfies the remaining hypotheses in the statement of Theorem \ref{et}. Once this is done, Theorem \ref{et} applies and we conclude $\psd(A)=\Sos{A}$ and $p(A)\leq4\tau(\kappa)$. 

We use Lemma \ref{trans} freely along the proof. We do not have to care about positive semidefinite units of $A$, because they can be written as $cU^2$ for some $c\in\kappa\setminus\{0\}$ and some unit $U\in A$ with $U(0,0,0)=1$. Thus, $c\in\psd(A)$, so $c>_\alpha0$ for each $\alpha\in\Sper(A)$ such that $\supp(\alpha)=\gtm$. Such prime cones are in bijection with the orderings of $\kappa$, so $c\in\Sos{\kappa}$.

(iv) This case is somehow special because with the current coordinates $F(\x,0)=0$, so a suitable change of coordinates is needed to apply Theorem \ref{et}. We prove first: \em if $f+\z g\in\psd(A)$, then $\omega(f)\geq2$\em.

Suppose that $f+\z g\in\psd(A)$ and $\omega(f)=1$. Observe that $f\in\psd(\{F\geq0\})$. If $\omega(f(\x,0))=1$, there exist $\zeta\in\kappa[[\y]]$ of order $\geq1$, $b\in\kappa\setminus\{0\}$ and a unit $U\in\kappa[[\x,\y]]$ such that $U(0,0)=1$ and $f=(\x-\zeta(\y))bU^2$. Consider the homomorphism
$$
\varphi:\kappa[[\x,\y]]\to\kappa[[\t]],\ h\mapsto h(\zeta(\t^3)-b\t^2,\t^3).
$$
We have $\varphi(F)=(\zeta(\t^3)-b\t^2)^2\t^3+(-1)^ka\t^{3k}$ and $\varphi(f)=-b^2\t^2\varphi(U)^2\in-\Sos{\kappa[[\t]]}$. As $b\neq0$ and $k\geq3$, we deduce $\omega(\varphi(F))=7$ and $\varphi(F)=b^2\t^7+\cdots$. Fix an ordering $\alpha\in\Sper(\kappa)$ and let $\beta_0\in\Sper(\kappa[[\t]])$ be a prime cone that extends $\alpha$ such that $\t>_\beta0$. Thus, $\varphi(F)>_{\beta_0}0$, whereas $\varphi(f)=-b^2\t^2\varphi(U)^2<_{\beta_0}0$, which is a contradiction because $f\in\psd(\{F\geq0\})$. Consequently, $\omega(f(\x,0))\geq 2$.

Assume next $\omega(f(\x,0))\geq2$ and $\omega(f(0,\y))=1$. There exist $\xi:=\sum_{\ell\geq2}\xi_\ell\x^\ell\in\kappa[[\x]]$ of order $\geq2$, $c\in\kappa\setminus\{0\}$ and a unit $V\in\kappa[[\x,\y]]$ such that $V(0,0)=1$ and $f=(\y-\xi(\x))cV^2$. Let $\alpha\in\Sper(\kappa)$ and $M:=\xi_2^2+1$. Consider the homomorphism
$$
\phi:\kappa[[\x,\y]]\to\kappa[[\t]],\ h\mapsto h(\t,M\t^2).
$$
We have $\phi(F)=M\t^4-a(-1)^k\t^{2k}>0$ because $k\geq3$, so $\phi(f)=c(M\t^2-\xi(\t))\phi(V)^2>0$. As $M>\xi_2$, we deduce $c>_\alpha0$. As this holds for each $\alpha\in\Sper(\kappa)$, we conclude $c\in\Sos{\kappa}\setminus\{0\}$. 

Thus, we may assume $f=\y-\xi(\x)$. As the coefficient $a\not\in-\Sos{\kappa}$, there exists $\alpha\in\Sper(\kappa)$ such that $a>_\alpha0$. Consider the homomorphism
$$
\psi:\kappa[[\x,\y]]\to \kappa[[\t]],\ h\mapsto h(\t^k,-\t^2).
$$
We have $\psi(F)=-\t^{2k+2}+a\t^{2k}>0$ and $\psi(f)=-\t^2-\xi(\t^k)<0$ because $k\geq3$. Thus, there exists a prime cone $\beta_1\in\Sper(\kappa[[\t]])$ extending $\alpha$ such that $\psi(F)>_{\beta_1}0$, whereas $\psi(f)<_{\beta_1}0$, which is a contradiction because $f\in\psd(\{F\geq0\})$.

Consequently, $\omega(f)\geq2$. By Example \ref{examples}(ii) we know that $F$ is $k$-determined. After a linear change of coordinates $\y\mapsto\lambda\x+\y$ (for some $\lambda\in\kappa$) the series $F$ becomes $F(\x,\y)=\x^2(\lambda\x+\y)+(-1)^ka(\lambda\x+\y)^k$ and satisfies $F(\x,0)=\lambda\x^3+(-1)^ka\lambda^k\x^k$ and $F(0,\y)=(-1)^ka\y^k$. In case $k\geq4$, we choose $\lambda=1$, whereas in case $k=3$, we choose $\lambda$ such that $\lambda(1-a\lambda^2)\neq0$ (in such a way that in both cases $\omega(F(\x,0))=3$). As we have already proved, there exist no elements $f+\z g\in\psd(A)$ such that $\omega(f)=1$.

(v) For this case we use a `limit argument'. Let $f+\z g\in\psd(A)$. We may assume $g\neq0$ (Lemma \ref{pd+}(iv)). Then by Corollary \ref{psd3}
\begin{align*}
&f\in\psd(\{\x^2\y\geq0\})\subset\psd(\{\y\geq0\}),\\
&f^2-\x^2\y g^2\in\psd(\kappa[[\x,\y]]).
\end{align*}
By Lemma \ref{pd+} if $n$ is large enough, then $f+(\x^2+\y^2)^n+\z g\in\psd^\oplus(A)$, that is, by \S\ref{pde}
\begin{align*}
&f+(\x^2+\y^2)^n\in\psd^+(\{\x^2\y\geq0\})\subset\psd(\{\y\geq0\}),\\
&(f+(\x^2+\y^2)^n)^2-\x^2\y g^2\in\psd^+(\kappa[[\x,\y]]).
\end{align*}
By Corollary \ref{psdkxy} we deduce that if $r\geq1$ is large enough, then
$$
(f+(\x^2+\y^2)^n)-(\x^2\y+\y^{2r})g^2\in\psd^+(\kappa[[\x,\y]]).
$$
Let us check: \em if $r\geq 2n+2$ is large enough, then $f+(\x^2+\y^2)^n\in\psd^+(\{\x^2\y+\y^{2r}\geq0\})$\em.

As $f+(\x^2+\y^2)^n\in\psd^+(\{\x^2\y\geq0\})$, the series $f+(\x^2+\y^2)^n$ is positive for each $\gamma\in\Sper(\kappa[[\x,\y]])$ such that $\supp(\gamma)=(\x)$. This means that there exist $1\leq s\leq n$, $c\in\Sos{\kappa}\setminus\{0\}$ and a unit $u\in\kappa[[\y]]$ such that $u(0)=1$ and $f(0,\y)+\y^{2n}=\y^{2s}cu^2$. Let $\beta\in\Sper(\kappa[[\x,\y]])$ be such that $\y(\x^2+\y^{2r-1})\geq_\beta0$. If $\y\geq_\beta0$, then $\x^2\y\geq_\beta0$, so $f\geq_\beta0$ and $f+(\x^2+\y^2)^n\geq_\beta0$. If $\y<_\beta0$, then $\x^2+\y^{2r-1}\leq_\beta0$ and $\y\not\in\supp(\beta)$. Thus, as $2r-1\geq 4n+3$,
$$
\x^2\leq_\beta-\y^{2r-1}\leq_\beta\y^{4n+2}\leq_\beta\y^{4s+2}.
$$
As $\y<_\beta0$, then $|\x|_\beta<_\beta-\y^{2s+1}$ where
$$
|h|_\beta:=\begin{cases}
h&\text{if $h\geq_\beta0$,}\\
-h&\text{if $h<_\beta0$}
\end{cases}
$$
for each $h\in\kappa[[\x,\y]]$. Write 
$$
f+(\x^2+\y^2)^n=f(0,\y)+\y^{2n}+\x\xi=\y^{2s}cu^2+\x\xi
$$ 
where $\xi\in\kappa[[\x,\y]]$. We have
$$
f+(\x^2+\y^2)^n=\y^{2s}cu^2+\x\xi\geq_\beta\y^{2s}cu^2-|\x|_\beta|\xi|_\beta\geq\y^{2s}cu^2+\y^{2s+1}(1+\xi^2(0,0))>_\beta0.
$$
Thus, $f+(\x^2+\y^2)^n\in\psd^+(\{\x^2\y+\y^{2r}\geq0\})$, so $f+(\x^2+\y^2)^n+\z g\in\psd^\oplus(A_r)$ where $A_r:=\kappa[[\x,\y,\z]]/(\z^2-\x^2\y-\y^{2r})$. As $\psd(A_r)=\Sosp{A_r}$, there exist $a_{in},b_{in},q_n\in\kappa[[\x,\y]]$ such that:
$$
f+(\x^{2n}+\y^{2n})+\z g=(a_{1n}+\z b_{1n})^2+\cdots+(a_{pn}+\z b_{pn})^2-(\z^2-\x^2\y-\y^{2r})q_n.
$$
Consequently,
$$
f+\z g\equiv (a_{1n}+\z b_{1n})^2+\cdots+(a_{pn}+\z b_{pn})^2-(\z^2-\x^2\y)q_n\ \mod\gtm_2^{2n}.
$$
By Strong Artin's approximation there exist $a_i,b_i,q\in\kappa[[\x,\y]]$ such that 
$$
f+\z g\equiv (a_1+\z b_1)^2+\cdots+(a_p+\z b_p)^2-(\z^2-\x^2\y)q, 
$$
that is, $f+\z g\in\Sosp{A}$.

(vi) In this case $F(\x,0)=\x^3$, $F(0,\y)=b\y^3$ and $F$ is $3$-determined (Example \ref{examples}(iii)). We have to prove: \em if $f+\z g\in\psd(A)$ and $f(\x,0)\neq0$, then $\omega(f(\x,0))\geq 2$\em.

Assume $\omega(f(\x,0))=1$. There exist $\zeta\in\kappa[[\y]]$ of order $\geq1$, $c\in\kappa\setminus\{0\}$ and a unit $U\in\kappa[[\x,\y]]$ such that $U(0,0)=1$ and $f=(\x-\zeta(\y))cU^2$. Consider the homomorphism
$$
\varphi:\kappa[[\x,\y]]\to\kappa[[\t]],\ h\mapsto h(\zeta(\t)-c\t^2,\t).
$$
We have $\varphi(F)=(\zeta(\t)-c\t^2)^3+a\t^2(\zeta(\t)-c\t^2)+b\t^3$ and $\varphi(f)=-c^2\t^2U^2(\zeta(\t)-c\t^2,\t)$. As $b\neq0$, we deduce
$$
\omega((\zeta(\t)-c\t^2)^3+a\t^2(\zeta(\t)-c\t^2)+b\t^3)=\begin{cases}
\omega(F(d,1)\t^3+\cdots)=3&\text{if $\omega(\zeta)=1$},\\
\omega(b\t^3+\cdots)=3&\text{if $\omega(\zeta)\geq2$},
\end{cases}
$$
for some $d\in\kappa\setminus\{0\}$. We only have to explain the first row. As $\zeta\in\kappa[[\t]]$ and $\omega(\zeta)=1$, then $\zeta=d\t+\cdots$, so 
$$
\varphi(F)=(\zeta(\t)-c\t^2)^3+a\t^2(\zeta(\t)-c\t^2)+b\t^3=(d^3+ad+b)\t^3+\cdots=F(d,1)\t^3+\cdots.
$$
As $F(\x,1)\in\kappa[\x]$ is an irreducible polynomial, $F(d,1)\neq0$ and $\varphi(F)$ has order $3$. 

Fix an ordering $\alpha$ of $\kappa$ and let $\beta\in\Sper(\kappa[[\t]])$ be a prime cone that extends $\alpha$ such that $F(d,1)\t>0$ in the first case and $b\t>0$ in the second case. Thus, $\varphi(F)>_\beta0$, whereas $\varphi(f)=-c^2\t^2U^2(\zeta(\t)-c\t^2,\t)<_\beta0$. This is a contradiction because $f\in\psd(\{F\geq0\})$. Consequently, $\omega(f(\x,0))\geq 2$.

(vii) In this case $F(\x,0)=\x^3$, $F(0,\y)=a\y^4$ and $F$ is $4$-determined (Example \ref{examples}(v)). We have to prove: \em if $f+\z g\in\psd(A)$ and $f(\x,0)\neq0$, then $\omega(f(\x,0))\geq 2$\em. 

Assume $\omega(f(\x,0))=1$. There exists $\zeta\in\kappa[[\y]]$ of order $\geq1$, $b\in\kappa\setminus\{0\}$ and a unit $U\in\kappa[[\x,\y]]$ such that $U(0,0)=0$ and $f=(\x-\zeta(\y))bU^2$. As $a\not\in-\Sos{\kappa}$, there exists an ordering $\alpha\in\Sper(\kappa)$ such that $a>_\alpha0$. Consider the homomorphism
$$
\varphi:\kappa[[\x,\y]]\to\kappa[[\t]],\ h\mapsto h(\zeta(\t)-b\t^2,\t).
$$
We have $\varphi(F)=(\zeta(\t)-b\t^2)^3+a\t^4$ and $\varphi(f)=-b^2\t^2\varphi(U)^2\in-\Sos{\kappa[[\t]]}$. In addition, 
$$
\omega((\zeta(\t)-b\t^2)^3+a\t^4)=\begin{cases}
3&\text{if $\omega(\zeta)=1$},\\
4&\text{if $\omega(\zeta)\geq2$.}
\end{cases}
$$ 
Let $\beta\in\Sper(\kappa[[\t]])$ be a prime cone that extends $\alpha$ and satisfies $(\zeta(\t)-b\t^2)^3+a\t^4>_\beta0$. Observe that $\varphi(f)=-b^2\t^2\varphi(U)^2<_\beta0$. This is a contradiction because $f\in\psd(\{F\geq0\})$. Consequently, $\omega(f(\x,0))\geq 2$.

(viii) We prove first: \em if $f+\z g\in\psd(A)$, then $\omega(f(\x,0))\geq2$\em.

Suppose that $f+\z g\in\psd(A)$ and $\omega(f(\x,0))=1$. There exist $\zeta\in\kappa[[\y]]$ of order $\geq1$, $a\in\kappa\setminus\{0\}$ and a unit $U\in\kappa[[\x,\y]]$ such that $U(0,0)=1$ and $f=(\x+\zeta(\y))aU^2$. Let $\alpha\in\Sper(\kappa)$ and consider the homomorphism
$$
\varphi:\kappa[[\x,\y]]\to\kappa[[\t]],\ h\mapsto h(\t^2,0).
$$
We have $\varphi(F)=\t^6>0$, so $\varphi(f)=a\t^2\phi(U)^2>0$. This means that $a>_\alpha0$ and as this holds for each $\alpha\in\Sper(\kappa)$, we conclude $a\in\Sos{\kappa}\setminus\{0\}$. Thus, we may assume $f=\x+\zeta(\y)$. Consider the homomorphism
$$
\varphi:\kappa[[\x,\y]]\to\kappa[[\t]],\ h\mapsto h(0,\t).
$$
We have $\varphi(F)=0$ and $\varphi(f)=\zeta(\t)$. As $f\in\psd(\{F\geq0\})$ and $\varphi(F)=0$, we deduce $\zeta(\t)\in\psd(\kappa[[\t]])=\Sos{\kappa[[\t]]}$. In particular, $\omega(\zeta(\t))\geq2$ and we choose $M:=\zeta_2^2+1$ where $\zeta(\t):=\sum_{k\geq2}\zeta_k\t^k$. Consider the homomorphism
$$
\varphi:\kappa[[\x,\y]]\to\kappa[[\t]],\ h\mapsto h(-M\t^2,-\t).
$$
Choose $\beta\in\Sper(\kappa[[\t]])$ such that $\t>_\beta0$. Then $\varphi(F)=-M^3\t^6+M\t^5>_\beta0$ and $\varphi(f)=-M\t^2+\zeta(\t)<_\beta0$ because $M>_\beta\zeta_2$. This is a contradiction because $f\in\psd(\{F\geq0\})$. Consequently, $\omega(f(\x,0))\geq2$.

By Example \ref{examples}(iv) $F$ is $5$-determined. After the change of coordinates $\x\mapsto\x+\y^2$ the series $F$ becomes $F(\x,\y)=(\x+\y^2)^3+(\x+\y^2)\y^3$ and satisfies $F(\x,0)=\x^3$ and $F(0,\y)=\y^5+\y^6$. As we have already proved, there exists no element $f+\z g\in\psd(A)$ such that $\omega(f(\x,0))=1$ (and this enough in our situation, because the performed change of coordinates was $\x\mapsto\x+\y^2$).

(ix) In this case $F(\x,0)=\x^3$, $F(0,\y)=\y^5$ and $F$ is $5$-determined (Example \ref{examples}(v)). We have to prove: \em if $f+\z g\in\psd(A)$ and $f(\x,0)\neq0$, then $\omega(f(\x,0))\geq 2$\em. 

Assume $\omega(f(\x,0))=1$. There exists $\zeta\in\kappa[[\y]]$ of order $\geq1$, $b\in\kappa\setminus\{0\}$ and a unit $U\in\kappa[[\x,\y]]$ such that $U(0,0)=1$ and $f=(\x-\zeta(\y))bU^2$. Consider the homomorphism
$$
\varphi:\kappa[[\x,\y]]\to\kappa[[\t]],\ h\mapsto h(\zeta(\t)-b\t^2,\t).
$$
We have $\varphi(F)=(\zeta(\t)-b\t^2)^3+\t^5$ and $\varphi(f)=-b^2\t^2\varphi(U)^2\in-\Sos{\kappa[[\t]]}$. In addition, 
$$
\omega((\zeta(\t)-b\t^2)^3+\t^5)=\begin{cases}
3&\text{if $\omega(\zeta)=1$},\\
5&\text{if $\omega(\zeta)\geq2$.}
\end{cases}
$$ 
Fix an ordering $\alpha\in\Sper(\kappa)$ and let $\beta\in\Sper(\kappa[[\t]])$ be a prime cone that extends $\alpha$ such that $(\zeta(\t)-b\t^2)^3+\t^5>0$. Thus, $\varphi(F)>_\beta0$, whereas $\varphi(f)=-b^2\t^2\varphi(U)^2<_\beta0$. This is a contradiction because $f\in\psd(\{F\geq0\})$. Consequently, $\omega(f(\x,0))\geq 2$, as required.
\end{proof}

\begin{remarks}
(i) Although Theorem \ref{et} has been thought as a main tool to prove Theorem \ref{order3}, it provides further information. If $A=\kappa[[\x,\y,\z]]/(\z^2-F)$ where 
\begin{itemize}
\item $\kappa$ is a (formally) real field such that $\tau(\kappa)<+\infty$, 
\item $F(0,\y))=b\y^{\rho}v$ where $v\in\kappa[[\y]]$ is a unit such that $v(0)=1$, $b\neq0$ and $\rho\geq3$ is either odd or $\rho$ is even and $b\not\in-\Sos{\kappa}$,
\item $\omega(F(\x,0))=3$ and $F$ is $\rho$-quasidetermined, 
\end{itemize}
then the possible elements belonging to the difference $\psd(A)\setminus\Sos{A}$ are series of order $1$. Thus, for this type of rings there exist either none or `few' series that are positive semidefinite but they are not sums of squares. 

(ii) In the statement of Theorem \ref{et}(ii) we have required that the series $F\in\kappa[[\x,\y]]$ is $\rho$-quasideter\-mined where $\rho=\omega(F(0,\y))$. In fact, it is enough that it is $k$-quasidetermined for some $k\geq1$. The reason for our assumption is that under such hypotheses the proof has a simpler presentation and all the cases to which we apply Theorem \ref{et}(ii) to prove Theorem \ref{order3} are included. However, the alternative formulation provides further singularities to which we can apply the previous remark, see Examples \ref{examples}(iii)-(vii).

(iii) In \cite{sch6} Scheiderer gives a negative answer (if $n\geq2$) to the following question raised by Sturmfels: \em Let $f\in\Q[\x_1,\ldots,\x_n]$ be a polynomial, which is a sum of squares in $\R[\x_1,\ldots,\x_n]$. Is $f$ necessarily a sum of squares in $\Q[\x_1,\ldots,\x_n]$?\em

We can formulate a similar question for the rings $A=\kappa[[\x,\y,\z]]/(\z^2-F)$ where $\kappa$ admits a unique ordering, $\tau(\kappa)<+\infty$ and $F\in\kappa[[\x,\y]]$ satisfies the hypotheses of Theorem \ref{et}(ii) (even the relaxed ones quoted in the previous remark). Let $R$ be the real closure of $\kappa$ (endowed with its unique ordering) and $B:=R[[\x,\y,\z]]/(\z^2-F)$. By Lemma \ref{trans}, an adapted version of Lemma \ref{dimen1} and \cite[Prop.VII.5.1]{abr} one deduces that $\psd(A)=\psd(B)\cap A$. Let $f+\z g\in A\cap\Sos{B}\subset A\cap\psd(B)=\psd(A)$. By Theorem \ref{et}(i) there exists $\ell\geq0$ and $f_1+\z g_1\in\psd(A)$ such that $f_1(\x,0)\neq0$ and $f+\z g=\y^{2\ell}(f_1+\z g_1)$. As $f+\z g\in\Sos{B}$, there exist $a_i,b_i,q\in R[[\x,\y]]$ such that
\begin{align}
\y^{2\ell}(f_1+\z g_1)&=\sum_{i=1}^p(a_i+\z b_i)^2-(\z^2-F)q,\label{514}\\
\y^{2\ell}f_1&=\sum_{i=1}^pa_i^2+F\sum_{i=1}^pb_i^2,\label{515}\\
\y^{2\ell}g_1&=2\sum_{i=1}^pa_ib_i.
\end{align}
Observe that $q=\sum_{i=1}^pb_i^2$. If $\ell\geq1$, we set $\y=0$ in \eqref{515} and deduce 
$$
0=\sum_{i=1}^pa_i^2(\x,0)+\x^3\sum_{i=1}^pb_i^2(\x,0),
$$
so $a_i(\x,0)=0$ and $b_i(\x,0)=0$ for each $i$. There exist $a_i',b_i',q'\in R[[\x,\y]]$ satisfying $a_i=\y a_i'$, $b_i=\y b_i'$ and $q=\sum_{i=1}^pb_i^2=\y^2q'$. This means that we can divide \eqref{514} by $\y^2$. Proceeding recursively we conclude $f_1+\z g_1\in A\cap\Sos{B}$. Thus, either $f_1+\z g_1$ is a unit in $A$ and $f_1+\z g_1\in\Sos{A}$ (see the beginning of the proof of Theorem \ref{order3}) or there exist series $c_i,d_i\in R[[\x]]$ such that $f_1(\x,0)=\sum_{i=1}^pc_i^2+\x^3\sum_{i=1}^pd_i^2$ (set $\y=0,\z=0$ in a representation of $f_1+\z g_1$ as a sum of squares in $B$). Consequently, $\omega(f_1(\x,0))\geq2$ and by Theorem \ref{et}(ii) $f_1+\z g_1\in\Sos{A}$, so $f+\z g\in\Sos{A}$. Thus, the answer to the corresponding Sturmfels question for this family of rings is positive!
\end{remarks}

\section{Applications: Principal saturated preorderings of low order}\label{s6}

In this section we generalize some results of \cite{sch4,sch5} concerning the characterization of the principal saturated preorderings $T:={\rm PO}(F):=\{s_1+s_2F:\ s_1,s_2\in\Sos{A}\}$ of low order in a $2$-dimensional excellent henselian regular local ring $(A,\gtm)$ whose residue field $\kappa:=A/\gtm$ is (formally) real and has $\tau(\kappa)<+\infty$. If we denote $X(T):=\{\alpha\in\Sper(A):\ f\geq_{\alpha}0\ \forall f\in T\}$, the \em saturation of $T$ \em is ${\rm Sat}(T):=\{f\in A:\ f\geq_{\alpha}0\ \forall\alpha\in X(T)\}$. As one can expect, $T$ is \em saturated \em if $T={\rm Sat}(T)$. If $F\in A\setminus\gtm^2$, then $T$ is always saturated \cite[Lem.3.1]{sch3}, so we focus our attention on elements $F\in\gtm^2$. A general strategy in \cite{sch4,sch5} to prove that a principal preordering $T:={\rm PO}(F)$ of $\kappa[[\x,\y]]$ is saturated is to consider the extended ring $B:=\kappa[[\x,\y,\z]]/(\z^2-F)$ and to observe that ${\rm Sat}(T)=\psd(\{F\geq0\})\subset\psd(B)$ (Corollary \ref{psd3}). If $f\in{\rm Sat}(T)$ and $\psd(B)=\Sos{B}$, then there exist $a_i,b_i,q\in\kappa[[\x,\y]]$ for $i=1,\ldots,p$ such that
$$
f=\sum_{i=1}^p(a_i+b_i\z)^2-(\z^2-F)q\quad\leadsto\quad f=\sum_{i=1}^pa_i^2+F\sum_{i=1}^pb_i^2\in T.
$$
The previous implication is used in the sketches of proofs of the following results.

\begin{cor}[Order two]\label{ord2p}
Let $(A,\gtm)$ be an excellent henselian regular local ring of dimension $2$ such that its residue field $\kappa:=A/\gtm$ admits a unique ordering and has $\tau(\kappa)<+\infty$. Let $F\in\gtm^2\setminus\gtm^3$ and $T$ be the principal preordering of $A$ generated by $F$. Then $T$ is saturated in $A$ if and only if there exists $c\in\kappa\setminus{0}$ such that $c^2F$ is right equivalent in $\widehat{A}\cong\kappa[[\x,\y]]$ to one of the following series:
\begin{itemize}
\item[(i)] $a\x^2+b\y^{2k}$ such that $a>0$, $b\neq0$ and $k\geq 1$,
\item[(ii)] $a\x^2+\y^{2k+1}$ where $a>0$ and $k\geq 1$,
\item[(iii)] $a\x^2$ where $a>0$.
\end{itemize}
\end{cor}
\begin{proof}[Sketch of proof]
The result follows straightforwardly from Corollary \ref{scht}, the proof of Theorem \ref{ord2} and Theorem \ref{order2}. A key fact for our purposes is that in the proof of Theorem \ref{ord2} if the difference $\psd(A)\setminus\Sos{A}\neq\varnothing$, then $\psd(\{F\geq0\})\setminus\Sos{A}\neq\varnothing$. As a guideline, the reader can have a look at \cite[\S5.7]{sch4}.
\end{proof}

In the case of order $3$ the result is even more satisfactory, because we do not need to impose that $\kappa$ admits a unique ordering, but only that $\kappa$ is a (formally) real field.

\begin{cor}[Order three]\label{ord3p}
Let $(A,\gtm)$ be an excellent henselian regular local ring of dimension $2$ such that its residue field $\kappa:=A/\gtm$ is a (formally) real field and has $\tau(\kappa)<+\infty$. Let $F\in\gtm^3\setminus\gtm^4$ and $T$ be the principal preordering of $A$ generated by $F$. Then $T$ is saturated in $A$ if and only if there exists $c\in\kappa\setminus{0}$ such that $c^2F$ is right equivalent in $\widehat{A}\cong\kappa[[\x,\y]]$ to one of the following series:
\begin{itemize}
\item[(iv)] $\x^2\y+(-1)^ka\y^k$ where $a>0$ and $k\geq 3$,
\item[(v)] $\x^2\y$,
\item[(vi)] $\x^3+a\x\y^2+b\y^3$ irreducible,
\item[(vii)] $\x^3+a\y^4$ where $a>0$, 
\item[(viii)] $\x^3+\x\y^3$,
\item[(ix)] $\x^3+\y^5$. 
\end{itemize}
\end{cor}
\begin{proof}[Sketch of proof]
The result follows straightforwardly from Corollary \ref{scht}, the proof of Theorem \ref{ord3} and Theorem \ref{order3}. A key fact for our purposes is that in the proof of Theorem \ref{ord3} if the difference $\psd(A)\setminus\Sos{A}\neq\varnothing$, then $\psd(\{F\geq0\})\setminus\Sos{A}\neq\varnothing$. As a guideline the reader can have a look at \cite[\S6.5]{sch5}.
\end{proof}

In \cite[Rem.6.8]{sch5} Scheiderer explains that going beyond order three seems difficult because we do not even know whether the saturatedness of ${\rm PO}(F)$ depends only on the pair $(\widehat{A},F)$. In \cite[Thm.6.3 \& Thm.6.6]{sch5} Scheiderer extends the analogous results to Corollaries \ref{ord2p} and \ref{ord3p} when $\kappa$ is a real closed field to the setting of excellent regular local rings of dimension $2$, that is, he erases the `henselian condition'. The strategy proposed there uses strongly the fact that the residue field is real closed. It seems to us that the extension of our results (Corollaries \ref{ord2p} and \ref{ord3p}) to the setting of an excellent regular local ring of dimension $2$ when the residue field $\kappa$ is (formally) real and has $\tau(\kappa)<+\infty$ needs a new strategy, even if we impose that $\kappa$ admits a unique ordering. 

\appendix

\section{Examples}\label{a}

We revisit case (3.vi) of Theorem \ref{list2}. Recall that $\kappa$ is a (formally) real field such that $\tau(\kappa)<+\infty$. We follow similar ideas to those proposed in Lemma \ref{ext}.

\begin{example}\label{irredrevisted}
We prove alternatively the property $\psd(A)=\Sos{A}$ for the ring 
$$
A:=\kappa[[\x,\y,\z]]/(\z^2-F)
$$ 
where $F:=\x^3+a\x\y^2+b\y^3$ is an irreducible polynomial in $\kappa[\x,\y]$ under the additional assumption $-a,-4a^3-27b^2\in\Sos{\kappa}\setminus\{0\}$.

A Sturm's sequence for the polynomial $F(\x,1)$ is
$$
\begin{cases}
F(\x,1),3\x^2+a,-2a\x-3b,-4a^3-27b^2&\text{if $a\neq 0$,}\\
F(\x,1),\x^2,-b&\text{if $a=0$.}
\end{cases}
$$
As $F$ is irreducible, $-4a^3-27b^2,b\neq 0$. If $-a,-4a^3-27b^2\in\Sos{\kappa}\setminus\{0\}$, then $F$ has three roots in each real closure of $\kappa$ endowed with the corresponding ordering (use Sturm's Theorem \cite[Cor.1.2.10]{bcr}). Let $\zeta_1,\zeta_2,\zeta_3\in\ol{\kappa}$ be the roots of $F(\x,1)$ in the algebraic closure $\ol{\kappa}$ of $\kappa$. For each $\alpha\in\Sper(\kappa)$ let us denote the real closure of $(\kappa,\leq_\alpha)$ with $R(\alpha)$. By Sturm's Theorem \cite[Cor.1.2.10]{bcr} $\zeta_1,\zeta_2,\zeta_3\in R(\alpha)$ for each $\alpha\in\Sper(\kappa)$. Consider the extension $\kappa(\zeta_1)|\kappa$ and the polynomial
$$
P(\x,\y):=(\x-\zeta_2\y)(\x-\zeta_3\y)=\x^2-(\zeta_2+\zeta_3)\x\y+\zeta_2\zeta_3\y^2\in\kappa(\zeta_1)[\x,\y]\ \leadsto\ \zeta_2+\zeta_3,\zeta_2\zeta_3\in\kappa(\zeta_1).
$$
We claim: $(\zeta_2-\zeta_3)^2\in\Sos{\kappa(\zeta_1)}$.

Observe that $(\zeta_2-\zeta_3)^2=(\zeta_2+\zeta_3)^2-4\zeta_2\zeta_3\in\kappa(\zeta_1)\subset R(\alpha)$ for each $\alpha\in\Sper(\kappa)$. Thus, $(\zeta_2-\zeta_3)^2$ is positive in all the orderings of the field $\kappa(\zeta_1)$, so $(\zeta_2-\zeta_3)^2\in\Sos{\kappa(\zeta_1)}$.

Denote $\u:=\x-\zeta_1\y$, $c:=\zeta_1-\zeta_2$, $d:=\zeta_1-\zeta_3$, $\v:=\y+(\frac{1}{2c}+\frac{1}{2d})\u$. We have
\begin{align*}
cd&=(\zeta_1-\zeta_2)(\zeta_1-\zeta_3)=P(\zeta_1,1)\in\kappa(\zeta_1),\\
\frac{1}{2c}+\frac{1}{2d}=\frac{c+d}{2cd}&=\frac{2\zeta_1-\zeta_2-\zeta_3}{2P(\zeta_1,1)}=\frac{3\zeta_1}{2P(\zeta_1,1)}\in\kappa(\zeta_1),\\
\Big(\frac{1}{2c}-\frac{1}{2d}\Big)^2=\Big(\frac{d-c}{2cd}\Big)^2&=\frac{(\zeta_2-\zeta_3)^2}{4P^2(\zeta_1,1)}\in\Sos{\kappa(\zeta_1)}.
\end{align*}
Note that $\u,\v\in\kappa(\zeta_1)[\x,\y]$ and
\begin{equation*}
\begin{split}
F&=(\x-\zeta_1\y)(\x-\zeta_2\y)(\x-\zeta_3\y)=\u(\u+(\zeta_1-\zeta_2)\y)(\u+(\zeta_1-\zeta_3)\y)\\
&=\u(\u+c\y)(\u+d\y)=cd\u\Big(\y+\frac{1}{c}\u\Big)\Big(\y+\frac{1}{d}\u\Big)\\
&=cd\u\Big(\Big(\y+\Big(\frac{1}{2c}+\frac{1}{2d}\Big)\u\Big)+\Big(\frac{1}{2c}-\frac{1}{2d}\Big)\u\Big)\Big(\Big(\y+\Big(\frac{1}{2c}+\frac{1}{2d}\Big)\u\Big)-\Big(\frac{1}{2c}-\frac{1}{2d}\Big)\u\Big)\\
&=cd\u(\v^2-\theta^2\u^2)\in\kappa(\zeta_1)[\u,\v]
\end{split}
\end{equation*}
where $\theta:=\frac{1}{2c}-\frac{1}{2d}$ and $\theta^2\in\Sos{\kappa(\zeta_1)}$. Using the additional change of coordinates $(\u,\v,\z)\mapsto(cd\u,cd\v,(cd)^2\z)$, we deduce
$$
A\hookrightarrow B:=\kappa(\zeta_1)[[\x,\y,\z]]/(\z^2-F)\cong B':=\kappa(\zeta_1)[[\u,\v,\z]]/(\z^2-\u(\v^2-\theta^2\u^2)).
$$
By Theorem \ref{list2} $\psd(B')=\Sos{B'}$, so $\psd(A)\subset\psd(B)=\Sos{B}$. If $\psi\in\psd(A)$, there exist $a_i,b_i,c_i,d_j\in\kappa[[\x,\y,\z]]$ satisfying
\begin{equation*}
\begin{split}
\psi&=\sum_{i=1}^p(a_i+\zeta_1b_i+\zeta_1^2c_i)^2+(d_1+\zeta_1d_2+\zeta_1^2d_3)(\z^2-F)\\
&=\sum_{i=1}^p(a_i^2+\zeta_1^2b_i^2+\zeta_1^4c_i^2+2a_ib_i\zeta_1+2a_ic_i\zeta_1^2+2b_ic_i\zeta_1^3)+(d_1+\zeta_1d_2+\zeta_1^2d_3)(\z^2-F).
\end{split}
\end{equation*}
Let $\sigma_j:\kappa(\zeta_1)\rightarrow\kappa(\zeta_j)$ be the $\kappa$-isomorphism such that $\sigma_j(\zeta_1)=\zeta_j$. We apply $\sigma_j$ to the previous equality and obtain
$$
\psi=\sum_{i=1}^p(a_i^2+\zeta_j^2b_i^2+\zeta_j^4c_i^2+2a_ib_i\zeta_j+2a_ic_i\zeta_j^2+2b_ic_i\zeta_j^3)+(d_1+\zeta_jd_2+\zeta_j^2d_3)(\z^2-F)
$$
for $j=2,3$. We add the three equations:
\begin{multline}\label{add3}
3\psi=\sum_{i=1}^p(3a_i^2+(\zeta_1^2+\zeta_2^2+\zeta_3^2)b_i^2+(\zeta_1^4+\zeta_2^4+\zeta_3^4)c_i^2\\
+2a_ib_i(\zeta_1+\zeta_2+\zeta_3)+2a_ic_i(\zeta_1^2+\zeta_2^2+\zeta_3^2)+2b_ic_i(\zeta_1^3+\zeta_2^3+\zeta_3^3))\\
+(3d_1+(\zeta_1+\zeta_2+\zeta_3)d_2+(\zeta_1^2+\zeta_2^2+\zeta_3^2)d_3)(\z^2-F).
\end{multline}
After elementary computations, we obtain
\begin{align*}
\zeta_1+\zeta_2+\zeta_3&=0,\\
\zeta_1^2+\zeta_2^2+\zeta_3^2&=-2a,\\
\zeta_1^3+\zeta_2^3+\zeta_3^3&=-3b,\\
\zeta_1^4+\zeta_2^4+\zeta_3^4&=2a^2.
\end{align*}
We substitute these values in \eqref{add3} and deduce
\begin{equation*}
\begin{split}
3\psi=&\sum_{i=1}^p(3a_i^2-2ab_i^2+2a^2c_i^2-4aa_ic_i-6bb_ic_i)+(3d_1-2ad_3)(\z^2-F)\\
=&\sum_{i=1}^p\Big(3\Big(a_i^2-\frac{2a}{3}c_i\Big)^2-2a\Big(b_i+\frac{3b}{2a}c_i\Big)^2+\frac{6}{36a^2}(-a)(-4a^3-27b^2)c_i^2\Big)\\
&+(3d_1-2ad_3)(\z^2-F).
\end{split}
\end{equation*}
As $-a,-4a^3-27b^2\in\Sos{\kappa}$, we conclude (after multiplying the previous equation by $\frac{3}{9}$) that $\psi\in\Sos{A}$, so $\psd(A)=\Sos{A}$, as required.
\end{example}

In the following, we approach an example of a different nature. We prove $\psd(A)=\Sosd{A}$ for the case (3.iii) of the list of Theorem \ref{gp=s} in a more general situation than the one we are able to solve in Theorem \ref{order2} (we do not require that the coefficient $a\in\Sos{\kappa}$). The proof is inspired by some constructions of Ruiz developed in \cite[\S3, pages 6--7]{rz2} related to Whitney's umbrella singularity.

\begin{example}\label{newtrends}
Let $\kappa$ be a (formally) real field with $\tau(\kappa)=1$ and $A:=\kappa((\t))[[\x,\y,\z]]/(\z^2-\psi\x^2)$ where $\psi\in\kappa((\t))$ is a series of odd order. Then $\psd(A)=\Sos{A}$.
\end{example}
\begin{proof}
As $\psi$ has odd order, there exists a series $\zeta\in\kappa((\t))$ such that $\psi=\t\zeta^2$. After the change of coordinates $(\x,\y,\z)\mapsto(\frac{1}{\zeta}\x,\y,\z)$ we assume $A=\kappa((\t))[[\x,\y,\z]]/(\z^2-\t\x^2)$. Now, consider the ring homomorphism
$$
\Phi:\kappa((\t))[[\x,\y,\z]]\to\kappa((\s))[[\x,\y]],\ f(\t,\x,\y,\z)\mapsto f(\s^2,\x,\y,\s\x).
$$
The kernel of the previous homomorphism is the ideal generated by $\z^2-\t\x^2$. Thus, $\Phi$ induces an inclusion $\ol{\Phi}:A\hookrightarrow\kappa((\s))[[\x,\y]]$.

Let $\varphi:=f+\z g\in\psd(A)$ and $\psi:=\Phi(\varphi)=f(\s^2,\x,\y)+\s\x g(\s^2,\x,\y)\in\psd(\kappa((\s))[[\x,\y]])$. By \eqref{p2} and Theorem \ref{liso} $\psi$ is a sum of two squares in $\kappa((\s))[[\x,\y]]$, that is, there exist $a,b\in\kappa((\s))[[\x,\y]]$ such that $\psi=a^2+b^2$. To prove that $\varphi\in\Sigma A^2$, it is enough to show $\psi\in\Sigma_2{\ol\Phi}(A)^2$. We will find below $a',b'\in{\ol\Phi}(A)$ such that $2\psi=a'^2+b'^2$. Once this is done, 
$$
\psi=\frac{2}{4}(a'^2+b'^2)=\Big(\frac{a'+b'}{2}\Big)^2+\Big(\frac{a'-b'}{2}\Big)^2\in\Sigma_2{\ol\Phi}(A)^2.
$$

Each $h\in\kappa((\s))[[\x,\y]]$ can be written uniquely as $h=h_0+\s h_1$ where $h_0,h_1\in\kappa((\s^2))[[\x,\y]]$ and $h\in\ol{\Phi}(A)$ if and only if $\x$ divides $h_1$ or, equivalently, if $h(\s,0,\y)\in\kappa((\s^2))[[\y]]$. Thus, let us find $a',b'\in\kappa((\s))[[\x,\y]]$ such that $2\psi=a'^2+b'^2$ and $a'(\s,0,\y),b'(\s,0,\y)\in\kappa((\s^2))[[\y]]$. 

Write $a:=a_0+\s a_1$ and $b:=b_0+\s b_1$ where $a_i,b_i\in\kappa((\s^2))[[\x,\y]]$. Denote $A_i(\s^2,\y):=a_i(\s^2,0,\y)$ and $B_i(\s^2,\y):=b_i(\s^2,0,\y)$. If we set $\x=0$, we get 
\begin{multline*}
\varphi(\s^2,0,\y,0)=\psi(\s,0,\y)=(A_0+\s A_1)^2+(B_0+\s B_1)^2\\
=A_0^2+B_0^2+\s^2(A_1^2+B_1^2)+2\s (A_0A_1+B_0B_1).
\end{multline*}
We deduce
\begin{itemize}
\item $\varphi(\s^2,0,\y,0)=A_0^2(\s^2,\y)+B_0^2(\s^2,\y)+\s^2(A_1^2(\s^2,\y)+B_1^2(\s^2,\y))$. 
\item $A_0A_1+B_0B_1=0$.
\end{itemize}

If $u,v\in\kappa((\s))[[\x,\y]]$, we have
$$
(u^2+v^2)(a^2+b^2)=(ua-vb)^2+(va+ub)^2.
$$
Let us find $u(\s,\y),v(\s,\y)\in\kappa((\s))[[\y]]$ such that
\begin{multline*}
ua-vb,va+ub\in \ol{\Phi}(A),\quad u^2+v^2\in \kappa((\s^2))[[\y]]\subset\ol{\Phi}(A),\\
\quad (u^2+v^2)(\s^2,0)\in\kappa[[\s^2]],\quad (u^2+v^2)(0,0)=k
\end{multline*}
where $k$ is either $1$ or $2$. The latter condition implies that $u^2+v^2$ is a unit of $\kappa((\s^2))[[\y]]$ and 
$$
w:=\sqrt{\frac{u^2+v^2}{k}}\in\kappa((\s^2))[[\y]]. 
$$
Write $u=:u_0+\s u_1$ and $v:=v_0+\s v_1$, where $u_i,v_i\in\kappa((\s^2))[[\y]]$. It holds 
\begin{align*}
ua-vb&=(u_0a_0+\s^2u_1a_1-v_0b_0-\s^2v_1b_1)+\s(u_0a_1+u_1a_0-v_0b_1-v_1b_0),\\
va+ub&=(v_0a_0+\s^2v_1a_1+u_0b_0+\s^2u_1b_1)+\s(v_0a_1+v_1a_0+u_0b_1+u_1b_0).
\end{align*}
Recall that $ua-vb,va+ub\in A$ if and only if
$$
\begin{cases}
u_0A_1+u_1A_0-v_0B_1-v_1B_0=(u_0a_1+u_1a_0-v_0b_1-v_1b_0)(\s,0,\y)=0,\\
v_0A_1+v_1A_0+u_0B_1+u_1B_0=(v_0a_1+v_1a_0+u_0b_1+u_1b_0)(\s,0,\y)=0.
\end{cases}
$$
Let us write the equalities above as the system of linear equations
$$
\begin{cases}
A_0u_1-B_0v_1=-A_1u_0+B_1v_0,\\
B_0u_1+A_0v_1=-B_1u_0-A_1v_0.
\end{cases}
$$
Using that $A_0A_1+B_0B_1=0$
\begin{multline}\label{AB}
(A_0^2+B_0^2)
\left(\begin{array}{c}
u_1\\
v_1
\end{array}
\right)=
\left(\begin{array}{cc}
-(A_0A_1+B_0B_1)&A_0B_1-A_1B_0\\
-(A_0B_1-A_1B_0)&-(A_0A_1+B_0B_1)
\end{array}
\right)
\left(\begin{array}{c}
u_0\\
v_0
\end{array}
\right)\\
=
\left(\begin{array}{cc}
0&A_0B_1-A_1B_0\\
-(A_0B_1-A_1B_0)&0
\end{array}
\right)
\left(\begin{array}{c}
u_0\\
v_0
\end{array}
\right).
\end{multline}
As $\varphi\in\psd(A)$, also 
\begin{equation}\label{p00}
\varphi(\t,0,\y,0)=A_0^2(\t,\y)+B_0^2(\t,\y)+\t(A_1^2(\t,\y)+B_1^2(\t,\y))\in\psd(\kappa((\t))[\y]).
\end{equation}
If $A_0^2+B_0^2=0$, then $\varphi(\t,0,\y,0)=\t(A_1^2(\t,\y)+B_1^2(\t,\y))\in\psd(\kappa((\t))[[\y]])$, so $A_1,B_1=0$. If we take $u=1,v=0$, we choose $a':=a,b':=b$, which are both elements of $\ol{\Phi}(A)$, so $\psi=a^2+b^2\in\Sosd{\ol{\Phi}(A)}$ and $\varphi\in\Sosd{A}$. 

Suppose next $A_0^2+B_0^2\neq 0$. As $A_0A_1+B_0B_1=0$, we have
$$
A_0B_1-A_1B_0=\frac{B_1}{A_0}(A_0^2+B_0^2)=-\frac{A_1}{B_0}(A_0^2+B_0^2)
$$
(whenever the previous expressions make sense, that is, $A_0\neq0$ or $B_0\neq0$). Denote $\lambda:=\frac{B_1}{A_0}=-\frac{A_1}{B_0}\in\kappa((\s^2))((\y))$ and rewrite \eqref{AB} as
\begin{multline*}
(A_0^2+B_0^2)
\left(\begin{array}{c}
u_1\\
v_1
\end{array}
\right)=
\left(\begin{array}{cc}
0&\lambda(A_0^2+B_0^2)\\
-\lambda(A_0^2+B_0^2)&0
\end{array}
\right)
\left(\begin{array}{c}
u_0\\
v_0
\end{array}
\right)\\
=(A_0^2+B_0^2)\left(\begin{array}{cc}
0&\lambda\\
-\lambda&0
\end{array}
\right)
\left(\begin{array}{c}
u_0\\
v_0
\end{array}
\right).
\end{multline*}
Consequently, $u_1=\lambda v_0$, $v_1=-\lambda u_0$. We claim: \em $\lambda\in\kappa((\s^2))[[\y]]$ and $\lambda(\s^2,0)\in\kappa[[\s^2]]$\em. 

Denote the order of a series with respect to the variable $\y$ with $\omega_\y(.)$. By \eqref{p00} we deduce
$$
\omega_\y(A_0^2(\t,\y)+B_0^2(\t,\y))\leq\omega_\y(A_1^2(\t,\y)+B_1^2(\t,\y)).
$$
We distinguish two cases:

(1) If $\omega_\y(A_0^2(\t,\y)+B_0^2(\t,\y))< \omega_\y(A_1^2(\t,\y)+B_1^2(\t,\y))$, then $\omega_\y(A_0(\t,\y))< \omega_\y(B_1(\t,\y))$ or $\omega_\y(B_0(\t,\y))<\omega_\y(A_1(\t,\y))$. Consequently, 
$$
\omega_\y(\lambda)=\omega_\y\Big(\frac{B_1}{A_0}\Big)=\omega_\y\Big(-\frac{A_1}{B_0}\Big)>0,
$$ 
so $\lambda\in\kappa((\s^2))[[\y]]$ and $\lambda(\s^2,0)=0\in\kappa[[\s^2]]$.

(2) If $\omega_\y(A_0^2(\t,\y)+B_0^2(\t,\y))=\omega_\y(A_1^2(\t,\y)+B_1^2(\t,\y))=2r$, then $\omega_\y(\lambda)\geq 0$. If $\omega_\y(\lambda)>0$, we have $\lambda\in\kappa((\s^2))[[\y]]$ and $\lambda(\s^2,0)=0\in\kappa[[\s^2]]$, so we assume $\omega_\y(\lambda)=0$, hence $\lambda\in\kappa((\s^2))[[\y]]$. By \eqref{p00}
$$
\varphi(\t,0,0,0)=(A_0^2(\t,0)+B_0^2(\t,0))+\t(A_1^2(\t,0)+B_1^2(\t,0))\in\psd(\kappa((\t))),
$$
so $\omega(A_0^2(\t,0)+B_0^2(\t,0))\leq\omega(A_1^2(\t,0)+B_1^2(\t,0))$. Thus, $\omega(\lambda(\s^2,0))\geq 0$ and $\lambda(\s^2,0)\in\kappa[[\s^2]]$.

If we take $u_0=v_0=1$, we obtain $u_1=\lambda(\s^2,\y)$ and $v_1=-\lambda(\s^2,\y)$, so
$$
u:=1-\s\lambda(\s^2,\y)\quad\text{and}\quad v:=1+\s\lambda(\s^2,\y).
$$
Note that $u^2+v^2=2+2\s^2\lambda(\s^2,\y)^2\in\kappa((\s^2))[[\y]]$, $(u^2+v^2)(\s^2,0)=2+2\s^2\lambda(\s^2,0)\in\kappa[[\s^2]]$ and $(u^2+v^2)(0,0)=2$. Define
$$
a':=\frac{au-bv}{w}\in\ol{\Phi}(A)\quad\text{and}\quad b':=\frac{av+bu}{w}\in\ol{\Phi}(A)
$$
where $w:=\sqrt{\frac{u^2+v^2}{2}}\in\kappa((\s^2))[[\y]]$. We obtain
$$
a'^2+b'^2=\frac{(au-bv)^2+(av+bu)^2}{w^2}=\frac{2(a^2+b^2)(u^2+v^2)}{u^2+v^2}=2(a^2+b^2)=2\psi,
$$
as required. 
\end{proof}

\section{Finite determinacy}\label{b}

The purpose of this appendix is to present an elementary explicit proof of Theorem \ref{fdqd} (see also \cite{cs,fr2}). Beforehand, we need some preliminary results:

\begin{example}\label{comput}
Let $A$ be a ring containing $\Q$ and let $u\in A[[\t]]$. Then for each $a\in A$ the differential equation $\frac{d\y}{d\t}=u\y$ has a unique solution $y\in A[[\t]]$ such that $y(0)=a$.

We write $u:=\sum_{j\geq0}u_j\t^j$ and $\y:=\sum_{\ell\geq0}\y_\ell \t^{\ell}$. In order to obtain a solution of the differential equation $\frac{d\y}{dt}=u\y$ such that $\y(0)=a$, it is enough to find a solution $(y_k)_{k\geq0}$ of the recursive family of equations:
\begin{equation}\label{difeq}
\begin{cases}
\y_0=a,\\
\y_{k}=\frac{1}{k}\sum_{j+\ell=k-1}\u_j\y_\ell\in\Q[\u_0,\ldots,\u_{k-1},\y_0,\y_1,\ldots,\y_{k-1}]&\text{for $k\geq 1$}.
\end{cases}
\end{equation}
Proceeding by induction one can check that for each $k\geq 1$ there exists a polynomial $F_k\in\Q[\u_0,\ldots,\u_{k-1},\y_0]$ such that $(y_k)_{k\geq0}$ is a solution of the family of equations \eqref{difeq} if and only if $y_k=F_k(u_0,\ldots,u_{k-1},a)$ for each $k\geq 1$. Thus, for each $a\in A$ there exists a unique solution $y\in A[[\t]]$ of the previous differential equation such that $y(0)=y_0=a$.\qed 
\end{example}

In the same way we reach the following result.

\begin{lem}\label{comput1}
Let $A$ be a ring containing $\Q$ and denote $\x:=(\x_1,\ldots,\x_n)$. Let $a_{10},\ldots,a_{n0}\in A[[\x]]$ and $b_1,\ldots,b_n\in A[[\x,\t]]$ be such that $a_{i0}(0)=0$ and $b_i(0,0)=0$ for $i=1,\ldots,n$. Then there exist unique series $\phi_1,\ldots,\phi_n\in A[[\x_1,\ldots,\x_n,\t]]$ such that
\begin{equation}\label{diffeq}
\frac{\partial\phi_i}{\partial\t}=b_i(\phi_1,\ldots,\phi_n,\t)
\end{equation} 
and $\phi_i(\x,0)=a_{i0}$ for $i=1,\ldots,n$.
\end{lem}
\begin{proof}
Write $\phi_i:=\sum_{k\geq0}a_{ik}\t^k$ and $\frac{\partial\phi_i}{\partial\t}=\sum_{k\geq1}ka_{ik}\t^{k-1}$ where the $a_{ik}$ are undetermined for $i=1,\ldots,n$ and $k\geq1$. Observe that $\phi_1,\dots,\phi_n$ satisfy the equalities \eqref{diffeq} if and only if
$$
a_{ik}=\frac{1}{k!}\frac{\partial^k\phi_i}{\partial\t^k}(0)=\frac{1}{k!}\frac{\partial^{k-1}}{\partial\t^{k-1}}\Big(\frac{\partial\phi_i}{\partial\t}\Big)(0)=\frac{1}{k!}\frac{\partial^{k-1}}{\partial\t^{k-1}}(b_i(\phi_1,\ldots,\phi_n,\t))(0).
$$
Denote $\nu:=(\nu_1,\ldots,\nu_n)\in\N^n$ and $\mu\in\N$. Using the chain rule recursively, we find polynomials 
$$
P_{i,k-1}\in\Q[\z_{\nu,\mu},\y_{j\ell}:\ 0\leq|\nu|+\mu\leq k-1,0\leq\ell\leq k-1,1\leq j\leq n]
$$ 
such that
$$
\frac{\partial^{k-1}}{\partial\t^{k-1}}b_i(\phi_1,\ldots,\phi_n,\t)=P_{i,k-1}\Big(\frac{\partial^{|\nu|+\mu}b_i}{\partial\x^\nu\partial\t^\mu}(\phi_1,\ldots,\phi_n,\t),\frac{\partial^\ell\phi_j}{\partial\t^\ell}\Big)
$$
where $0\leq|\nu|+\mu\leq k-1$, $0\leq\ell\leq k-1$, $1\leq j\leq n$ and $\z_{\nu,\mu},\y_{j\ell}$ are variables. Consequently,
$$
a_{ik}=\frac{1}{k!}P_{i,k-1}\Big(\frac{\partial^{|\nu|+\mu}b_i}{\partial\x^\nu\partial\t^\mu}(a_{10},\ldots,a_{n0},0),a_{j\ell}\Big).
$$ 
We find recursively polynomials 
$$
Q_{i,k}\in\Q[\z_{\nu,\mu},\y_j: 0\leq|\nu|+\mu\leq k-1,1\leq j\leq n]
$$
such that 
$$
a_{ik}=Q_{i,k}\Big(\frac{\partial^{|\nu|+\mu}b_i}{\partial\x^\nu\partial\t^\mu}(a_{10},\ldots,a_{n0},0),a_{j0}\Big)
$$
where $0\leq|\nu|+\mu\leq k-1$ and $1\leq j\leq n$. Thus, each $a_{ik}$ is completely determined by $b_1,\ldots,b_n$ and $a_{10},\ldots,a_{n0}\in A$. This means that equation \eqref{diffeq} has a unique solution $\phi_1,\ldots,\phi_n$ such that $\phi_i(\x,0)=a_{i0}$ for $i=1,\ldots,n$.
\end{proof}

\begin{lem}\label{var}
Let $F\in\kappa[[\x,\y]]$ where $\x:=(\x_1,\ldots,\x_n)$, $\y:=(\y_1,\ldots,\y_m)$ and let $c\geq0$ be an integer. The following conditions are equivalent:
\begin{enumerate}
\item[(1)] $\frac{\partial F}{\partial\y_j}\in(\x)^c(\frac{\partial
F}{\partial\x_1},\ldots,\frac{\partial F}{\partial\x_n})+(F)$ for each $j=1,\ldots,m$.
\item[(2)] There exist $\varphi_1,\ldots,\varphi_n,u\in\kappa[[\x,\y]]$ such that:
\begin{itemize}
\item $u(\x,0)=1$,
\item $\varphi_i(\x,0)=\x_i$,
\item $\varphi_i-\x_i\in(\x)^c\kappa[[\x,\y]]$,
\item $F(\x,\y)=uF(\varphi,0)$ where $\varphi:=(\varphi_1,\ldots,\varphi_n)$.
\end{itemize}
\end{enumerate}
If in addition $\frac{\partial F}{\partial\y_j}\in(\x)^c(\frac{\partial F}{\partial\x_1},\ldots,\frac{\partial F}{\partial\x_n})$ for $j=1,\ldots,m$, we can choose $u=1$.
\end{lem}
\begin{proof}[Proof of the implication $(2)$ $\Longrightarrow$ $(1)$]
As $\varphi_i(\x,0)=\x_i$, the homomorphism
$$
\Phi:\kappa[[\x,\y]]\to\kappa[[\x,\y]],\ f\mapsto f(\varphi,\y)
$$
is by the Inverse Function Theorem an automorphism of $\kappa[[\x,\y]]$. Let $\psi:=(\psi_1,\ldots,\psi_n)\in\kappa[[\x,\y]]^n$ be such that 
$$
\Psi:\kappa[[\x,\y]]\to\kappa[[\x,\y]],\ g\mapsto g(\psi,\y)
$$
is the inverse automorphism of $\Phi$. As $\x_i-\varphi_i\in(\x)^c\kappa[[\x,\y]]$, it holds $\psi_i-\x_i=\Psi(\x_i-\varphi_i)\in(\x)^c\kappa[[\x,\y]]$, so $\frac{\partial\psi_i}{\partial\y_j}\in(\x)^c\kappa[[\x,\y]]$. 

By hypothesis $uF(\varphi,0)=F(\x,\y)$ and applying $\Psi$, we obtain $F(\x,0)=u^{-1}(\psi,\y)F(\psi,\y)$ (resp. $F(\x,0)=F(\psi,\y)$ if $u=1$). The derivative of the series $u^{-1}(\psi,\y)F(\psi,\y)$ with respect to $\y_j$ is zero for $j=1,\ldots,m$, that is,
$$
u^{-1}(\psi,\y)\Big(\sum_{i=1}^n\frac{\partial F}{\partial\x_i}(\psi,\y)\frac{\partial\psi_i}{\partial\y_j}+\frac{\partial F}{\partial\y_j}(\psi,\y)\Big)+\frac{\partial}{\partial\y_j}(u^{-1}(\psi,\y))F(\psi,\y)=0.
$$
We apply the automorphism $\Phi$ to the previous equality and obtain
$$
u^{-1}\Big(\frac{\partial F}{\partial\y_j}+\sum_{i=1}^n\frac{\partial F}{\partial\x_i}\frac{\partial\psi_i}{\partial\y_j}(\varphi,\y)\Big)+\frac{\partial}{\partial\y_j}(u^{-1}(\psi,\y))(\varphi,\y)F=0.
$$
Thus, $\frac{\partial F}{\partial\y_j}\in(\x)^c(\frac{\partial F}{\partial\x_1},\ldots,\frac{\partial F}{\partial\x_n})+(F)$ (resp. $\frac{\partial F}{\partial\y_j}\in (\x)^c(\frac{\partial F}{\partial\x_1},\ldots,\frac{\partial F}{\partial\x_n})$ if $u=1$), as required.
\end{proof}

\begin{proof}[Proof of the implication $(1)$ $\Longrightarrow$ $(2)$]
Write $\y_{(k)}:=(\y_1,\ldots,\y_k)$. Assume that we have series $\psi_1,\ldots,\psi_n\in\kappa[[\x,\y]]$ such that 
\begin{itemize}
\item[(1)] $\psi_i(\x,\y_{(k)},0)=\x_i$ for $i=1,\ldots,n$,
\item[(2)] $\psi_i-\x_i\in(\x)^c\kappa[[\x,\y]]$ for $i=1,\ldots,n$.
\end{itemize}
We rewrite the previous conditions as follows
\begin{itemize}
\item[(1)] $\psi_i=\x_i+\sum_{\ell=k+1}^ma_{i\ell}\y_\ell$ where $a_{i\ell}\in\kappa[[\x,\y]]$ for $i=1,\ldots,n$,
\item[(2)] $\psi_i-\x_i=\sum_{\nu,|\nu|=c}\x^\nu b_{\nu,i}\in(\x)^c\kappa[[\x,\y]]$ where $b_{\nu,i}\in\kappa[[\x,\y]]$ for $i=1,\ldots,n$.
\end{itemize}

\paragraph{}\label{red1} Define $G:=F(\psi,\y_{(k)},0)$. We claim: \em $\frac{\partial G}{\partial\y_j}\in(\x)^c(\frac{\partial G}{\partial\x_1},\ldots,\frac{\partial G}{\partial\x_n})+(G)$ for each $j=1,\ldots,m$\em. In addition, \em if $\frac{\partial F}{\partial\y_j}\in(\x)^c(\frac{\partial F}{\partial\x_1},\ldots,\frac{\partial F}{\partial\x_n})$, then $\frac{\partial G}{\partial\y_j}\in(\x)^c(\frac{\partial G}{\partial\x_1},\ldots,\frac{\partial G}{\partial\x_n})$\em.

Observe that
\begin{align*}
\frac{\partial G}{\partial\x_p}&=\sum_{i=1}^n\frac{\partial F}{\partial\x_i}(\psi,\y_{(k)},0)\frac{\partial\psi_i}{\partial\x_p},\\
\frac{\partial G}{\partial\y_j}&=\sum_{i=1}^n\frac{\partial F}{\partial\x_i}(\psi,\y_{(k)},0)\frac{\partial\psi_i}{\partial\y_j}+\begin{cases}
\frac{\partial F}{\partial\y_j}(\psi,\y_{(k)},0)&\text{if $j=1,\ldots,k$,}\\
0&\text{if $j=k+1,\ldots,m$.}
\end{cases}
\end{align*}
In addition,
\begin{align*}
\frac{\partial\psi_i}{\partial\x_p}&=\delta_{ip}+\sum_{\ell=k+1}^m\frac{\partial a_{i\ell}}{\partial\x_p}\y_\ell,\\
\frac{\partial\psi_i}{\partial\y_j}&=\sum_{\nu,|\nu|=c}\x^\nu\frac{\partial b_{\nu,i}}{\partial\y_j}\in(\x)^c\kappa[[\x,\y]],\\ 
\end{align*}
where $\delta_{ip}$ is the Kronecker-$\delta$. As $G:=F(\psi,\y_{(k)},0)$, we get
$$
\frac{\partial G}{\partial\y_j}\in(\x)^c\Big(\frac{\partial F}{\partial\x_1}(\psi,\y_{(k)},0),\ldots,\frac{\partial F}{\partial\x_n}(\psi,\y_{(k)},0)\Big)+
(G).
$$
In the additional case we can erase the addend $(G)$ above. Thus, for our purposes it is enough to check
$$
\Big(\frac{\partial F}{\partial\x_1}(\psi,\y_{(k)},0),\ldots,\frac{\partial F}{\partial\x_n}(\psi,\y_{(k)},0)\Big)=\Big(\frac{\partial G}{\partial\x_1},\ldots,\frac{\partial G}{\partial\x_n}\Big).
$$
But this follows from the fact that 
$$
\det\Big(\frac{\partial\psi_i}{\partial\x_p}\Big)=\det\Big(\delta_{ip}+\sum_{\ell=k+1}^m\frac{\partial a_{i\ell}}{\partial\x_p}\y_\ell\Big)\in\kappa[[\x,\y]]
$$ 
is a unit (set $\y_{k+1}=0,\ldots,\y_m=0$ to check this). 

\paragraph{}\label{red2} Let $\phi_1,\ldots,\phi_n\in\kappa[[\x,\y]]$ be such that
\begin{itemize}
\item[(1)] $\phi_i(\x,\y_{(k-1)},0)=\x_i$ for $i=1,\ldots,n$,
\item[(2)] $\phi_i-\x_i\in(\x)^c\kappa[[\x,\y]]$ for $i=1,\ldots,n$.
\end{itemize}
Denote $\phi:=(\phi_1,\ldots,\phi_n)$ and $\psi_i':=\psi_i(\phi,\y)$. It is straightforward to check:
\begin{itemize}
\item[(1)] $\psi_i'(\x,\y_{(k-1)},0)=\psi_i(\phi(\x,\y_{(k-1)},0),\y_{(k-1)},0)=\x_i$ for $i=1,\ldots,n$,
\item[(2)] $\psi_i'-\x_i=\psi_i(\phi,\y)-\x_i\in(\x)^c\kappa[[\x,\y]]$ for $i=1,\ldots,n$.
\end{itemize}

\paragraph{}\label{red3}
Recall that $\y_{(m-1)}:=(\y_1,\ldots,\y_{m-1})$. We claim: \em There exist $u,\phi_1,\ldots,\phi_n\in\kappa[[\x,\y]]$ such that
\begin{itemize}
\item $u=1$ if $\frac{\partial F}{\partial\y_j}\in(\x)^c(\frac{\partial F}{\partial\x_1},\ldots,\frac{\partial F}{\partial\x_n})$ for $j=1,\ldots,m$,
\item $u(\x,\y_{(m-1)},0)=1$,
\item $\phi_i(\x,\y_{(m-1)},0)=\x_i$ for $i=1,\ldots,n$,
\item $\phi_i-\x_i\in(\x)^c\kappa[[\x,\y]]$ for $i=1,\ldots,n$,
\item $F(\phi,\y_{(m-1)},0)=uF(\x,\y)$ where $\phi:=(\phi_1,\ldots,\phi_m)$.
\end{itemize}
\em

Let $\xi_1,\ldots,\xi_n\in(\x)^c\kappa[[\x,\y]]$ and $\zeta\in\kappa[[\x,\y]]$ be such that \begin{equation}\label{F}
\frac{\partial F}{\partial\y_m}=-\zeta F+\sum_{i=1}^n\xi_i\frac{\partial F}{\partial\x_i}
\end{equation} 
(the case $\zeta=0$ is included). Consider the formal differential system of equations
$$
\frac{\partial\Phi_i}{\partial\t}=\xi_i(\Phi_1,\ldots,\Phi_n,\y_{(m-1)},\y_m-\t),
$$
and $\Phi_i(\x,\y,0)=\x_i$ for $i=1,\ldots,n$. By Lemma \ref{comput1} the previous formal differential system has a unique solution $\Phi_1,\ldots,\Phi_n\in\kappa[[\x,\y,\t]]$. As $\xi_i\in(\x)^c\kappa[[\x,\y]]$, we deduce
$$
\frac{\partial\Phi_i}{\partial\t}\in(\x)^c\kappa[[\x,\y,\t]].
$$
Write $\frac{\partial\Phi_i}{\partial\t}=\sum_{k\geq0}a_{ik}\t^k$ with $a_{ik}\in\kappa[[\x,\y]]$. We have $a_{i0}=\x_i$ and
$$
a_{ik}=\frac{1}{k!}\frac{\partial^{k+1}\Phi_i}{\partial\t^{k+1}}(0)\in(\x)^c\kappa[[\x,\y]]
$$
for each $k\geq1$. Consequently, $\Phi_i-\x_i\in(\x)^c\kappa[[\x,\y,\t]]$ for $i=1,\ldots,n$.

If $\Phi:=(\Phi_1,\ldots,\Phi_n)$, we deduce by the chain rule and \eqref{F}
\begin{equation}\label{both}
\begin{split}
\frac{\partial F(\Phi,\y_{(m-1)},\y_m-\t)}{\partial\t}&=\sum_{i=1}^n\frac{\partial F}{\partial\x_i}(\Phi,\y_{(m-1)},\y_m-\t)\frac{\partial\Phi_i}{\partial\t}-\frac{\partial F}{\partial\y_m}(\Phi,\y_{(m-1)},\y_m-\t)\\
&=\sum_{i=1}^n\Big(\frac{\partial F}{\partial\x_i}\xi_i\Big)(\Phi,\y_{(m-1)},\y_m-\t)-\frac{\partial F}{\partial\y_m}(\Phi,\y_{(m-1)},\y_m-\t)\\
&=(\zeta F)(\Phi,\y_{(m-1)},\y_m-\t).
\end{split}
\end{equation}
By Example \ref{comput} there exists a series $U\in\kappa[[\x,\y]][[\t]]$ such that
\begin{equation}\label{both2}
\frac{\partial U}{\partial\t}=\zeta(\Phi,\y_{(m-1)},\y_m-\t)U
\end{equation}
and $U(0)=1$ (note that $U=1$ if $\zeta=0$). 

By \eqref{both} and \eqref{both2} both $UF(\x,\y)$ and $F(\Phi,\y_{(m-1)},\y_m-\t)$ are solutions of the differential equation
$$
\frac{\partial\z}{\partial t}=\zeta(\Phi,\y_{(m-1)},\y_m-\t)\z\quad\text{and}\quad\z(0)=F(\x,\y).
$$
By Example \ref{comput}
\begin{equation}\label{F2}
F(\Phi,\y_{(m-1)},\y_m-\t)=UF(\x,\y).
\end{equation}
Substitute $\t=\y_m$ and define 
$$
\begin{cases}
\phi_i:=\Phi(\x,\y,\y_m),\\
u:=U(\x,\y,\y_m).
\end{cases}
$$ 
We have:
\begin{itemize}
\item $u=1$ if $\frac{\partial F}{\partial\y_j}\in(\x)^c(\frac{\partial F}{\partial\x_1},\ldots,\frac{\partial F}{\partial\x_n})$ for $j=m$ because in this case $\zeta=0$,
\item $u(\x,\y_{(m-1)},0)=1$ because $U(\x,\y,0)=1$, 
\item $\phi_i(\x,\y_{(m-1)},0)=\Phi_i(\x,\y,0)=\x_i$,
\item $\phi_i-\x_i\in(\x)^c\kappa[[\x,\y]]$ because $\Phi_i-\x_i\in(\x)^c\kappa[[\x,\y,\t]]$,
\item $F(\phi,\y_{(m-1)},0)=F(\Phi(\x,\y,\y_m),\y_{(m-1)},0)=U(\x,\y,\y_m)F(\x,\y)=uF(\x,\y)$.
\end{itemize}

\paragraph{} Let us show how we can construct the series $\phi_1,\ldots,\phi_n\in\kappa[[\x,\y]]$ recursively using \ref{red1}, \ref{red2} and \ref{red3}. For $k=m$ we take $u^{(m)}:=1$ and $\psi^{(m)}_i:=\x_i$ for $i=1,\ldots,n$. Assume there exist $u^{(k)},\psi^{(k)}:=(\psi_1^{(k)},\ldots,\psi_n^{(k)})\in\kappa[[\x,\y]]$ such that
\begin{itemize}
\item $u^{(k)}=1$ if $\frac{\partial F}{\partial\y_j}\in(\x)^c(\frac{\partial F}{\partial\x_1},\ldots,\frac{\partial F}{\partial\x_n})$ for $j=1,\ldots,m$,
\item $u^{(k)}(\x,\y_{(k)},0)=1$,
\item $\psi_i^{(k)}(\x,\y_{(k)},0)=\x_i$ for $i=1,\ldots,n$,
\item $\psi_i^{(k)}-\x_i\in(\x)^c\kappa[[\x,\y]]$ for $i=1,\ldots,n$,
\item $F(\psi^{(k)},\y_{(k)},0)=u^{(k)}F(\x,\y)$.
\end{itemize}

Denote $G:=F(\psi^{(k)},\y_{(k)},0)\in\kappa[[\x,\y]]$ and $\y':=(\y_{k+1},\ldots,\y_m)$. By \ref{red3} there exist $u,\phi:=(\phi_1,\ldots,\phi_n)\in\kappa[[\x,\y]]$ such that
\begin{itemize}
\item $u=1$ if $\frac{\partial G}{\partial\y_j}\in(\x)^c(\frac{\partial G}{\partial\x_1},\ldots,\frac{\partial G}{\partial\x_n})$ for $j=1,\ldots,m$,
\item $u(\x,\y_{(k-1)},0,\y')=1$,
\item $\phi_i(\x,\y_{(k-1)},0,\y')=\x_i$ for $i=1,\ldots,n$,
\item $\phi_i-\x_i\in(\x)^c\kappa[[\x,\y]]$ for $i=1,\ldots,n$,
\item $G(\phi,\y_{(k-1)},0,\y')=uG(\x,\y)$.
\end{itemize}
Observe that
\begin{multline*}
F(\psi^{(k)}(\phi,\y_{(k-1)},0,\y')),\y_{(k-1)},0)=G(\phi,\y_{(k-1)},0,\y')\\
=uG(\x,\y)=uF(\psi^{(k)},\y_{(k)},0)=uu^{(k)}F(\x,\y).
\end{multline*}
If we take $u^{(k-1)}:=uu^{(k)}$ and $\varphi^{(k-1)}:=(\varphi^{(k-1)}_1,\ldots,\varphi^{(k-1)}_n):=\psi^{(k)}(\phi,\y_{(k-1)},0,\y'))$, the reader can check using \ref{red1} and \ref{red2} that $u^{(k-1)},\varphi^{(k-1)}_1,\ldots,\varphi^{(k-1)}_n$ satisfy the desired properties in this step. The series in the statement $u,\varphi_1,\ldots,\varphi_n\in\kappa[[\x_1,\ldots,\x_n]]$ correspond to the final step $k=0$ when we have got rid of all the variables $\y_1,\ldots,\y_m$, as required.
\end{proof}

Now we are ready to prove Theorem \ref{fdqd}.

\begin{proof}[Proof of the Theorem \em \ref{fdqd}]
We approach the case $\gtm^{k}\subset\gtm J(f)+(f)$ (and we indicate in the precise points the differences to the case $\gtm^{k}\subset\gtm J(f)$). Let us prove that $f$ is $k$-quasidetermined (resp. $k$-determined). Let $g\in\kappa[[\x]]$ be such that $h:=g-f\in\gtm^{k+1}$. Then there exist series $h_1,\ldots,h_n\in\gtm^k$ such that $h=\x_1h_1+\cdots+\x_nh_n$. We introduce new variables $\y:=(\y_1,\ldots,\y_n)$ and consider the series
\begin{equation}\label{defF}
F(\x,\y):=f+\sum_{i=1}^n(\x_i+\y_i)h_i\in\kappa[[\x,\y]],
\end{equation}
which satisfies $F(\x,0)=f+h=g$. Denote the maximal ideal of $\kappa[[\x,\y]]$ with $\gtn$.

We claim: $\frac{\partial F}{\partial\y_j}\in\gtm(\frac{\partial F}{\partial\x_1},\ldots,\frac{\partial F}{\partial\x_n})+(F)$ (resp. $\frac{\partial F}{\partial\y_j}\in\gtm(\frac{\partial F}{\partial\x_1},\ldots,\frac{\partial F}{\partial\x_n})$).

We differentiate $F$ with respect to $\x_j$ in \eqref{defF} and obtain
\begin{equation}\label{diffF}
\frac{\partial F}{\partial\x_j}-\frac{\partial f}{\partial\x_j}=h_j+\sum_{i=1}^n(\x_i+\y_i)\frac{\partial h_i}{\partial\x_j}\in\gtn\gtm^{k-1}.
\end{equation}
Consequently, by \eqref{defF} and \eqref{diffF} we deduce
\begin{multline*}
\gtm J(f)\kappa[[\x,\y]]+(f)\kappa[[\x,\y]]\subset\gtm J(F)+(F)+\gtn\gtm^k\\
\subset\gtm J(F)+(F)+\gtn(\gtm J(f)\kappa[[\x,\y]]+(f)\kappa[[\x,\y]])
\end{multline*}
(resp. $\gtm J(f)\kappa[[\x,\y]]\subset\gtm J(F)+\gtn\gtm^k\subset\gtm J(F)+\gtn(\gtm J(f)\kappa[[\x,\y]])$ if $\gtm^{k}\subset\gtm J(f)$). By Naka\-yama's Lemma (for inclusions) \cite[Lem.6.13, p.25]{b} we conclude
$$
\gtm J(f)\kappa[[\x,\y]]+(f)\kappa[[\x,\y]]\subset\gtm J(F)+(F)
$$
(resp. $\gtm J(f)\kappa[[\x,\y]]\subset\gtm J(F)$). Thus,
$$
\frac{\partial F}{\partial\y_j}=h_j\in\gtm^k\subset\gtm J(f)+(f)\subset\gtm J(f)\kappa[[\x,\y]]+(f)\kappa[[\x,\y]]\subset\gtm J(F)+(F)
$$
(resp. $\frac{\partial F}{\partial\y_j}=h_j\in\gtm J(F)$) for $j=1,\ldots,n$.

By Lemma \ref{var} there exist $\varphi_1,\ldots,\varphi_n,u\in\kappa[[\x,\y]]$ such that
\begin{itemize}
\item $u(\x,0)=1$,
\item $\varphi_i(\x,0)=\x_i$,
\item $\varphi_i-\x_i\in \gtm\kappa[[\x,\y]]$,
\item $F(\x,\y)=uF(\varphi,0)$, where $\varphi=(\varphi_1,\ldots,\varphi_n)$.
\end{itemize}
In addition, if $\frac{\partial F}{\partial\y_j}\in\gtm(\frac{\partial F}{\partial\x_1},\ldots,\frac{\partial F}{\partial\x_n})$ for $j=1,\ldots,n$ (that is, if $\gtm^k\subset\gtm J(f)$), we can choose $u=1$. As $\varphi_i(\x,0)=\x_i$ and $\varphi_i-\x_i\in\gtm\kappa[[\x,\y]]$, we write
$$
\varphi_i=\x_i+\sum_{j,\ell=1}^n\x_j\y_\ell\zeta_{j\ell}(\x,\y)
$$
where $\zeta_{j\ell}\in\kappa[[\x,\y]]$. Thus, the map 
$$
\Phi:\kappa[[\x]]\to\kappa[[\x]], h\mapsto h(\varphi(\x,-\x))
$$
is an automorphism of $\kappa[\x]$. We obtain
\begin{multline*}
f=F(\x,-\x)=u(\x,-\x)F(\varphi(\x,-\x),0)\\
=u(\x,-\x)(f+h)(\varphi(\x,-\x))=u(\x,-\x)\Phi(f+h)=u(\x,-\x)\Phi(g),
\end{multline*}
where $u=1$ if $\gtm^k\subset\gtm J(f)$. Thus, $f$ is $k$-quasidetermined (resp. $k$-determined), as required.
\end{proof}

\section{Declarations}

\subsection{Funding}

Author is supported by Spanish STRANO MTM2017-82105-P and Grupos UCM 910444.

\subsection{Conflicts of interest/Competing interests declaration}

The corresponding author states that there is no conflict of interest. Competing interests: The author declares none.


\begin{thebibliography}{CDLR}

\bibitem[ABR]{abr} C. Andradas, L. Br\"ocker, J.M. Ruiz: Constructible sets in real geometry. {\em Ergebnisse der Mathematik und ihrer Grenzgebiete} (3) {\bf33}. Springer-Verlag, Berlin, (1996).

\bibitem[BGV]{bgv} K.-J. Becher, D. Grimm, J. Van Geel: Sums of squares in algebraic function fields over a complete discretely valued field. {\em Pacific J. Math.} {\bf267} (2014), no. 2, 257--276. 

\bibitem[BCR]{bcr} J. Bochnak, M. Coste, M.F. Roy: Real Algebraic Geometry. \em Ergeb. Math. \em {\bf 36}. Berlin Heidelberg New York: Springer-Verlag, (1998).

\bibitem[B]{b} J.W. Bruce: Classifications in singularity theory and their applications (p.3--33). New developments in singularity theory. {\em Proceedings of the NATO Advanced Study Institute held in Cambridge}, July 31–August 11, (2000). Edited by D. Siersma, C. T. C. Wall and V. Zakalyukin. NATO Science Series II: Mathematics, Physics and Chemistry, {\bf21}. Kluwer Academic Publishers, Dordrecht, (2001).

\bibitem[Ca]{ca} J. W. S. Cassels: On the representation of rational functions as sums of squares. {\em Acta Arith.} {\bf9} (1964), 79--82.

\bibitem[Che]{che} A. Chenciner: Courbes alg\'ebriques planes. [Plane algebraic curves] 
{\em Publications Math\'ematiques de l'Universit\'e Paris} VII [Mathematical Publications of the University of Paris VII], {\bf4}. Universit\'e de Paris VII, U.E.R. de Math\'ematiques, Paris, (1978).

\bibitem[CDLR]{cdlr} M.D. Choi, Z.D. Dai, T.Y.Lam, B. Reznick: The Pythagoras number of some affine algebras and local algebras, {\em J. Reine Angew. Math.} {\bf336} (1982), 45--82. 

\bibitem[CLRR]{clrr} M.D. Choi, T.Y.Lam, B. Reznick, A. Rosenberg: Sums of squares in some integral domains, {\em J. Algebra} {\bf65} (1980), no. 1, 234--256.

\bibitem[Ch]{ch} S. Choi: The divisor class group of surfaces of embedding dimension $3$. {\em J. Algebra} {\bf119} (1988), no. 1, 162--169.

\bibitem[CK]{ck} S.-D. Cutkosky, O. Kashcheyeva: Algebraic series and valuation rings over nonclosed fields. {\em J. Pure Appl. Algebra} {\bf 212} (2008), no. 8, 1996--2010. 

\bibitem[CS]{cs} S.-D. Cutkosky, H. Srinivasan: Equivalence and finite determinacy of mappings. {\em J. Algebra} {\bf188} (1997), no. 1, 16--57. 

\bibitem[DST]{dst} M. Dickmann, N. Schwartz, M. Tressl: Spectral spaces. {\em New Mathematical Monographs}, {\bf35}. Cambridge University Press, Cambridge, (2019). 

\bibitem[E]{e} D. Eisenbud: Commutative algebra. With a view toward algebraic geometry. {\em Graduate Texts in Mathematics}, {\bf150}. Springer-Verlag, New York, (1995).

\bibitem[Fe1]{f1} J.F. Fernando: On the Pythagoras numbers of real analytic rings. {\em J. of Algebra}, {\bf243} (2001), 321--338.

\bibitem[Fe2]{f2} J.F. Fernando: Positive semidefinite germs in real analytic surfaces, {\em Math. Ann.} {\bf 322} (2002), no. 1, 49--67.

\bibitem[Fe3]{f3} J.F. Fernando: Sums of squares in real analytic rings, {\em Trans. Amer. Math. Soc.} {\bf 354} (2002), no. 5, 1909--1919.

\bibitem[Fe4]{f4} J.F. Fernando: Analytic germs of minimal Pythagoras number. {\em Math. Z.} {\bf244} (2003), no. 4, 725--752.

\bibitem[Fe5]{f5} J.F. Fernando: Erratum: Analytic germs of minimal Pythagoras number. {\em Math. Z.} {\bf250} (2005), no. 4, 967--969.

\bibitem[Fe6]{f6} J.F. Fernando: Sums of squares in excellent henselian local rings. {\em Advances in Mathematics Research}. {\bf7}. Edited by Gabriel Oyibo. Nova Science Publishers, Inc., Hauppauge, NY, (2007). Pages: 75--100, ISBN: 1-59454-458-1.

\bibitem[Fe7]{f7} J.F. Fernando: On the positive extension property and Hilbert's 17th problem for real analytic sets. {\em J. Reine Angew. Math.} {\bf618} (2008), 1--49. 

\bibitem[FR1]{fr1} J.F. Fernando, J.M. Ruiz: Positive semidefinite germs on the cone, {\em Pacific J. Math.} {\bf205} (2002), no. 1, 109--118.

\bibitem[FR2]{fr2} J.F. Fernando, J.M. Ruiz: Finite determinacy in rings of power series. (Spanish) {\em Mathematical contributions in honor of Professor Enrique Outerelo Dom\'\i nguez} (Spanish), 183--199, Homen. Univ. Complut., Editorial Complutense, Madrid, (2004).

\bibitem[FR3]{fr3} J.F. Fernando, J.M. Ruiz: On the Pythagoras numbers of real analytic set germs. {\em Bull. Soc. Math. France} {\bf133} (2005), no. 3, 349--362.

\bibitem[FRS1]{frs1} J.F. Fernando, J.M. Ruiz, C. Scheiderer: Sums of squares in real rings. {\em Trans. Amer. Math. Soc.} {\bf356} (2004), no. 7, 2663--2684.

\bibitem[FRS2]{frs2} J.F. Fernando, J.M. Ruiz, C. Scheiderer: Sums of squares of linear forms. {\em Math. Res. Lett.} {\bf13} (2006), no. 6, 945--954.

\bibitem[GLS]{gls} G.-M. Greuel, C. Lossen, E. Shustin: Introduction to singularities and deformations. {\em Springer Monographs in Mathematics}. Springer, Berlin, (2007).

\bibitem[H]{h} D.W. Hoffmann: Pythagoras numbers of fields. {\em J. Amer. Math. Soc.} {\bf12} (1999), no. 3, 839--848.

\bibitem[Hu]{hu} Y. Hu: The Pythagoras number and the u-invariant of Laurent series fields in several variables. {\em J. Algebra} {\bf426} (2015), 243--258.

\bibitem[JP]{jp} T. de Jong, G. Pfister: Local analytic geometry. Basic theory and applications. {\em Advanced Lectures in Mathematics}. Friedr. Vieweg \& Sohn, Braunschweig (2000).

%\bibitem[KS]{ks} M. Knebusch, C. Scheiderer: Einf\"uhrung in die reelle Algebra. Vieweg Studium: Aufbaukurs Mathematik [Vieweg Studies: Mathematics Course], {\bf63}. Friedr. Vieweg \& Sohn, Braunschweig, (1989).

\bibitem[L]{l} T. Y. Lam: Introduction to quadratic forms over fields. {\em Graduate Studies in Mathematics}, {\bf67}. American Mathematical Society, Providence, RI (2005).

\bibitem[Pf]{pf} A. Pfister: Zur Darstellung definiter Funktionen als Summe von Quadraten. {\em Invent. Math.} {\bf4} (1967), 229--237.

\bibitem[P]{p} Y. Pourchet: Sur la repr\'esentation en somme de carr\'es des polyn\^omes \`a une ind\'etermin\'ee sur un corps de nombres alg\'ebriques. {\em Acta Arith.} {\bf19} (1971), 89--104.

\bibitem[Rt]{rt} C. Rotthaus: On the approximation property of excellent rings. {\em Invent. Math.} {\bf88} (1987), no. 1, 39--63.

\bibitem[Rz2]{rz2} J.M. Ruiz: Sums of two squares in analytic rings, {\em Math. Z.} {\bf230} (1999), no. 2, 317--328.

\bibitem[Sch1]{sch1} C. Scheiderer: Sums of squares of regular functions on real algebraic varieties, {\em Trans. Amer. Math. Soc.} {\bf 352} (1999) no. 3, 1039--1069.

\bibitem[Sch2]{sch2} C. Scheiderer: On sums of squares in local rings. {\em J. reine angew. Math.} {\bf 540} (2001), 205--227.

\bibitem[Sch3]{sch3} C. Scheiderer: Sums of squares on real algebraic surfaces. {\em Manuscr. Math.} {\bf119}, 395--410.

\bibitem[Sch4]{sch4} C. Scheiderer: Weighted sums of squares in local rings and their completions, I. {\em Math. Z.} {\bf266} (2010), no. 1, 1--19. 

\bibitem[Sch5]{sch5} C. Scheiderer: Weighted sums of squares in local rings and their completions, II. {\em Math. Z.} {\bf266} (2010), no. 1, 21--42.

\bibitem[Sch6]{sch6} C. Scheiderer: Sums of squares of polynomials with rational coefficients. {\em J. Eur. Math. Soc.} (JEMS) {\bf18} (2016), no. 7, 1495--1513.

\bibitem[S]{sj} G. Scheja: Einige Beispiele faktorieller lokaler Ringe. {\em Math. Ann.} {\bf 172} (1967), 124--134.

\bibitem[Si]{si} C. Siegel: Darstellung total positiver Zahlen durch Quadrate. {\em Math. Z.} {\bf11} (1921), no. 3-4, 246–275.

\bibitem[ZS]{zs} O. Zariski, P. Samuel: Commutative algebra. Vol. II. {\em Graduate Texts in Mathematics}, {\bf29}. Springer-Verlag, New York-Heidelberg, (1975).
\end{thebibliography}
\end{document}